\documentclass[a4paper,11pt]{amsart}
\usepackage[latin1]{inputenc}
\usepackage{xcolor}

\usepackage{multicol}
\usepackage{amsfonts, amssymb, amsmath, amsthm, amscd,epsfig,mathrsfs, euscript}
\usepackage{latexsym}
\usepackage{graphicx}

\usepackage{cancel}
\usepackage[normalem]{ulem}

\usepackage[all]{xy}

\usepackage{enumerate}

\usepackage{verbatim}

\usepackage{hyperref}

\newtheorem{theorem}{Theorem}[section]

\newtheorem{corollary}[theorem]{Corollary}

\newtheorem{lemma}[theorem]{Lemma}

\newtheorem{proposition}[theorem]{Proposition}

	{\par\noindent{\bf Proposition \ref{res:hiper}.}\!\!
	\nopagebreak\normalsize\it}%

	{\par\noindent{\bf Theorem \ref{result43}.}\!\!
	\nopagebreak\normalsize\it}%

	{\par\noindent{\it Sketch of the proof}.  
	\nopagebreak\normalsize}%
	{\hfill\linebreak[2]\hspace*{\fill}$\circlearrowleft$}

	{\par\noindent{\it Proof of Proposition }\ref{prop:stab:smc}.  
	\nopagebreak\normalsize}%
	{\hfill\linebreak[2]\hspace*{\fill}$\circlearrowleft$}

	{\par\noindent{\it Proof of Propositions }\ref{adap:mon}{\it and }\ref{simult:adap}.\!\!\!
	\nopagebreak\normalsize}%
	{\hfill\linebreak[2]\hspace*{\fill}$\circlearrowleft$}

{\theoremstyle{definition}
       \newtheorem{definition}[theorem]{Definition}

       \newtheorem{remark}[theorem]{Remark}
       \newtheorem{example}[theorem]{Example}
    
       \newtheorem{parrafo}[theorem]{{\!}}}

\numberwithin{equation}{theorem}

\newcommand{\calo}{{\mathcal {O}}}

\DeclareMathOperator{\Max}{\underline{Max}}
\DeclareMathOperator{\mult}{mult}
\DeclareMathOperator{\ord}{ord}

\DeclareMathOperator{\Sing}{Sing}
\DeclareMathOperator{\Spec}{Spec}

\DeclareMathOperator{\Hord}{H-ord}
\DeclareMathOperator{\Slaux}{-Sl}
\DeclareMathOperator{\In}{In}
\DeclareMathOperator{\Gr}{Gr}

\newcommand{\G}{{\mathcal G}}

\renewcommand{\H}{{\mathcal H}}

\newcommand{\p}{{\mathfrak p}}

\newcommand{\m}{{\mathfrak {m}}}

\newcommand{\SSl}{\mathcal{S}\!\Slaux}

\DeclareMathAlphabet{\mathpzc}{OT1}{pzc}{m}{it}
\newcommand{\nub}{\overline{\nu}}

\newcommand{\Mm}{\mathrm{\underline{Max}\; mult}_X}
\newcommand{\mm}{\mathrm{max\, mult}_X}
\newcommand{\Gn}{\G^{(n)}}

\newcommand{\Gne}{\G^{(n-e)}}
\newcommand{\Gd}{\G^{(d)}}

\newcommand{\Ovne}{\mathcal O_{V^{(n-e)}}}  
\newcommand{\Vn}{V^{(n)}}
\newcommand{\Vne}{V^{(n-e)}}
\newcommand{\Vd}{V^{(d)}}

\newcommand{\Diff}{{\mathrm{Diff}}}

\title{The asymptotic Samuel function and invariants of singularities}
\author{A. Benito, A. Bravo, S. Encinas}
\thanks{The authors were partially supported by PGC2018-095392-B-I00; The second author  was partially  supported from the Spanish Ministry of Economy and Competitiveness, through the ``Severo Ochoa'' Program for Centres of Excellence in R\&D (SEV-2015-0554)}


\AtEndDocument{\bigskip{\footnotesize
\textsc{Depto. Did\'acticas Espec\'{\i}ficas,
Facultad de Educaci\'on, Universidad Aut\'onoma de Madrid, Cantoblanco 28049 Madrid, Spain} \newline
\textit{E-mail address}, A. Benito: \texttt{angelica.benito@uam.es}
\medskip

\noindent\textsc{Depto. Matem\'aticas, Facultad de Ciencias, Universidad Aut\'onoma de Madrid and Instituto de Ciencias Matem\'aticas CSIC-UAM-UC3M-UCM, Cantoblanco 28049 Madrid, Spain} \newline
\textit{E-mail address}, A. Bravo: \texttt{ana.bravo@uam.es}
\medskip

\noindent\textsc{Depto. \'Algebra, An\'alisis Matem\'atico, Geometr\'{\i}a y Topolog\'{\i}a, and IMUVA, Instituto de Matem\'aticas.	Universidad de Valladolid} \newline
\textit{E-mail address}, S. Encinas: \texttt{santiago.encinas@uva.es}
}}

\begin{document}
	
\begin{abstract}
The asymptotic Samuel function generalizes to arbitrary rings the usual order function of a regular local   ring.
In this paper, we use this function to introduce the notion of the \emph{Samuel slope} of a Noetherian local ring,
and we study some of its properties.
In particular, we focus on the case of a local ring at singular point of a variety,
and, among other results, we prove
that the Samuel slope of these rings is related to some invariants  used in algorithmic resolution of singularities. 

\end{abstract}
	
	\maketitle
	

\section{Introduction}
Let $X$ be an equidimensional algebraic variety of dimension $d$ defined over a perfect field $k$. 
If $X$ is not regular, then the set of points of maximum multiplicity, $\Max\mult_X$, is a closed proper set in $X$.
We will denote by $\max\mult_X$ the maximum value of the multiplicity at points of $X$.
A \emph{simplification of the multiplicity} of $X$ is a finite sequence of blow ups,
\begin{equation} \label{SimplificaMult}
			\xymatrix@R=0pt@C=30pt{
				X=X_0 & X_1 \ar[l]_{\ \ \ \ \pi _1} & \ldots \ar[l]_{\pi _2} & X_{L-1} \ar[l]_{\pi_{L-1}} & X_L \ar[l]_{\ \ \ \ \pi _L}
			}
\end{equation}
with
$$\max\mult_{X_0}  =  \max\mult_{X_1}  =\cdots =   \max\mult_{X_{L-1}}  >   \max\mult_{X_L},$$
where $\pi_i:X_i\to X_{i-1}$ is the blow up at a regular center contained in $\Max\mult_{X_{i-1}}$.
		
Simplifications of the multiplicity exist if the characteristic of $k$ is zero (see \cite{V}), and resolution of singularities follows from there. Recall that Hironaka's line of approach to resolution makes use of the Hilbert-Samuel function instead of the multiplicity \cite{Hir}.
The centers in the sequence (\ref{SimplificaMult}) are determined by resolution functions.
These are upper semi-continuous functions
		$$ \begin{array}{rrcl} 
			f_{X_i}:  &  X_i &  \to  & (\Gamma, \geq) \\
			& \zeta & \mapsto &  f_{X_i}(\zeta)
		\end{array}, \qquad i=0,\ldots, L-1$$
		and  their maximum value, $\max f_{X_i}$, achieved in a closed  regular   subset $\underline{\text{Max}} f_{X_i}\subseteq \Max\mult_{X_{i}}$,  selects the center to blow up.
		Hence, a simplification of the multiplicity of $X$, $X\leftarrow X_L$, is defined as  a sequence of blow ups at regular centers.
		\begin{equation}
			\label{constructive}
			X=X_0\leftarrow X_1 \leftarrow \ldots \leftarrow X_L.
		\end{equation}
		so that 
		$$\max f_{X_0}> \max f_{X_1} > \ldots > \max f_{X_L},$$
		where $\max f_{X_i}$ denotes the maximum value of $f_{X_i}$ for $i=0,1,\ldots,L$.   
		\medskip
		
		Usually, $f_X$ is defined at each point as a sequence of rational numbers. The first coordinate of $f_X$ is the  multiplicity, and
		the second is what we refer to as {\em Hironaka's order function in dimension $d$}, $\ord_X^{(d)}$, where $d$ is the dimension of $X$.  The function $\ord_X^{(d)}$  is a  positive rational number.
		At a given singular point $\zeta\in X$, $f_X(\zeta)$ would look as follows: 
		\begin{equation}
			\label{resol_coord}
			f_X(\zeta)=(\mult_{{\mathfrak m}_{\zeta}} ({\mathcal O}_{X,\zeta}), \ord_X^{(d)}(\zeta),\ldots) \in {\mathbb N}\times {\mathbb Q}^r,
		\end{equation}
		where  $\mult_{{\mathfrak m}_{\zeta}}({\mathcal O}_{X,\zeta})$ denotes   the multiplicity of the local ring $\mathcal{O}_{X,\zeta}$ at the maximal ideal $\mathfrak{m}_{\zeta}$. The remaining coordinates of $f_X(\zeta)$ can be shown to depend on $\ord_X^{(d)}(\zeta)$ (see \cite[Theorem 7.6 and \S 7.11]{E_V97}), 
		thus, we usually say that this is the main invariant in constructive resolution.  Therefore, the last set of coordinates can be though as a {\em refinement of the function} $\ord_X^{(d)}$.
As we will see, the function $\ord_X^{(d)}$ can always be defined if $k$ a perfect field.
\medskip

\begin{example} \label{primer_ejemplo} Let $k$ be a perfect field,   let $S$ be a smooth $k$-algebra of dimension $d$ and define  $R=S[x]$ as the polynomial ring in one variable with coefficients in $S$. Suppose $X$ is a hypersurface in $\Spec(R)$ of maximum multiplicity $m>1$ given by an equation of the form
			$$ f(x)=x^m+a_1x^{m-1}+\ldots+a_m \in   S[x].$$ 
			Set $\beta: \Spec(R) \to \Spec(S)$ and let $\zeta\in X$ be a point of multiplicity $m$. Then one can define a Rees algebra, ${\mathcal R}$,  on $S$, which we refer to as \emph{elimination algebra},  that collects information on the coefficients $a_i\in S$, $i=1,\ldots, m$. Hironaka's order function at the point $\zeta$, $\ord_X^{(d)}(\zeta)$,  is defined using ${\mathcal R}$ (see Section \ref{ElimAlg}). If the characteristic of the field  $k$ does not divide $m$, then, after a translation on the variable $x$, we can assume that the equation is  on  Tschirnhausen form:
			$$(x')^m+a_2'(x')^{m-2}+\ldots+a_m' \in   S[x].$$ 
			And, in such case,  it can be shown that:
			\begin{equation} \label{ordord2}
				\ord_X^{(d)}(\zeta):=\ord_{\zeta}({\mathcal R}) =\min_{i=2,\ldots,m}\left\{\frac{\nu_{\beta(\zeta)}(a'_i)}{i}\right\},
			\end{equation}
where $\nu_{\beta(\zeta)}$ denotes the usual order at the regular local ring $S_{\mathfrak m_{{\beta(\zeta)}}}$. As it turns out,  with the information provided by the elimination algebra ${\mathcal R}$, which is generated by {\em weighted functions} on the coefficients of $f(x)$, one has all the information needed to find a simplification of the multiplicity, at least in the characteristic zero case.
	 
			However, if   the characteristic of the field is $p$, and if   $p$ divides $m$, then, in general, the equality  (\ref{ordord2})  does not hold  (even if,  by chance,  the polynomial were in Tschirnhausen form). 	Philosophically speaking, the elimination algebra ${\mathcal R}$ collects information about the coefficients of $f(x)$, but somehow falls short in collecting {\em the  sufficient amount of information} when the characteristic is positive. This problem motivated in part the papers \cite{BVComp} and \cite{BVIndiana}. There,   the   function $\Hord^{(d)}_X$ was introduced by the first author in collaboration with O. Villamayor. In \cite{Benito_V}, this function played  a role in the proof of desingularization of two dimensional varieties. 
\medskip
			
To give some insight on how $\Hord^{(d)}_X$ is defined, 	suppose, for simplicity, that 
$m=p^{\ell}$ for some $\ell\in {\mathbb Z}_{\geq 1}$,
$f(x)=x^{p^{\ell}}+a_1x^{{p^{\ell}}-1}+\ldots+a_{p^{\ell}} \in   S[x]$, 
and let  $\zeta$ be a point  of multiplicity ${p^{\ell}}$. Then it can be proved that
$$ \ord_X^{(d)}(\zeta)\leq \frac{\nu_{\beta({\zeta})}(a_i)}{i}, \qquad i=1,\ldots,p^{\ell}-1.$$
But there are examples where 
$$\frac{\nu_{\beta({\zeta})}(a_{p^{\ell}})}{p^{\ell}} <\ord_X^{(d)}(\zeta),$$
and the inequality remains even after considering translations of the form $x'=x+s$, $s\in S_{\mathfrak{q}}$, where $\mathfrak{q}=\mathfrak{m}_{\beta(\zeta)}$. This pathology is part of the reasons why the resolution strategy (that works in characteristic zero) cannot be extended to the positive characteristic case.
			
The previous discussion  motivates the definition of the slope of $f(x)$ at $\zeta$ as:
$$Sl(f(x))(\zeta)=\min\left\{
\frac{\nu_{\beta({\zeta})}(a_{p^{\ell}})}{p^{\ell}}, \ord_X^{(d)}(\zeta)\right\}.$$
Changes of variables of the form $x=x'+s$ with $s\in S_{\mathfrak{q}}$ produce  changes on the coefficients of the equation: 
\begin{equation} \label{ChangeVar}
	f(x')=(x')^{p^{\ell}}+a'_1(x')^{p^{\ell}-1}+\ldots+a'_{p^{\ell}}
\end{equation}
which may lead to a different value of the slope.  
However,  it is possible to  construct an invariant from these numbers  by setting:
$$\Hord^{(d)}_X(\zeta):=\sup_{s\in S_{\mathfrak{q}}}\{Sl(f(x+s))(\zeta)\}.$$ 
Moreover this supremum is a maximum since there is a change of variables as in 
(\ref{ChangeVar}) for which
$$\Hord^{(d)}_X(\zeta)=\min\left\{
\frac{\nu_{\beta({\zeta})}(a'_{p^{\ell}})}{p^{\ell}}, \ord_X^{(d)}(\zeta)\right\}.$$
\end{example}
$\Hord^{(d)}_X$ can be defined for any hypersurface with maximum multiplicity $m$, even when $m$ is not a $p$-th power (see Section \ref{presentaciones}). Observe that the previous discussion   takes care of the case in which $X$ is locally a hypersurface at a singular point $\zeta$, since, after considering a suitable \'etale neighborhood of $X$ at $\zeta$, it can be assumed that the equation defining $X$ can be written as a polynomial in one variable with coefficients in some regular ring $S$. 
  
When  $X$ is an arbitrary algebraic $d$-dimensional variety defined over a perfect field,      $\Hord^{(d)}_X$ can also be defined (in \'etale topology) using \cite{BVComp}, \cite{BVIndiana} and Villamayor's {\em presentations of the multiplicity} in  \cite{V}.
In the latter paper it is proven that, locally, in an \'etale  neighborhood of a closed point $\xi$ of maximum multiplicity $m>0$, one can find  a smooth $k$-algebra $S$ of dimension $d$ and polynomials in different variables $x_i$ with coefficients in $S$, $f_i(x_i)\in S[x_i]$, of degrees $m_1,\ldots, m_e$,
with the following property:
If we consider
\begin{equation}
\label{polinomios}
		f_1(x_1),\ldots, f_e(x_e)\in R=S[x_1,\ldots, x_e],
\end{equation}	
then each $f_i(x_i)$ defines a hypersurface of maximum multiplicity $m_i$, $H_i=\{f_i=0\}$, so that, 
$X\subset \Spec(R)$ and
\begin{equation}\label{local_presentation} 
		\Mm =\cap_{i}\text{\underline{Max}\; mult}_{H_i}.
\end{equation}
In fact, the link between $X$ and the hypersurfaces $H_i$ is stronger as we will see in Section~\ref{Rees_Algebras}.
	
As in the hypersurface case, Hironaka's order function, $\ord_X^{(d)}$, is defined by constructing  an {\em elimination algebra}, ${\mathcal R}$ on $S$, again using certain weighted functions on the coefficients of the polynomials $f_i(x_i)$ (see Section \ref{ElimAlg}).
And, in the same way, we have that
\begin{equation*}
\Hord^{(d)}_X(\zeta)=\min_{i}\Hord^{(d)}_{H_i}(\zeta).
\end{equation*}

This approach  will allow  us to work in a situation very similar to    the hypersurface case.   
Details and definitions will be given in Sections \ref{presentaciones}   and \ref{seccion_demos}. The precise statement of Villamayor's result is given in  Theorem \ref{presentaciones_mult},
because it will be used in the proof of our results.

\

\noindent{\bf Results}
\medskip

\noindent
From our previous discussion, the value $\Hord^{(d)}_X(\zeta)$ codifies information from
the coefficients of the polynomials in (\ref{polinomios}) that only depends on the inclusion $S\subset R$.
Observe that the definition of the  function  $\Hord^{(d)}_X$ requires   
the use of  local (\'etale) embeddings, the selection of a sufficiently general finite projection to some smooth scheme, and the construction of a local presentation of the multiplicity as in (\ref{local_presentation}).
Neither of these choices is unique. As a consequence, some work has to be done to show that the values of the function do not depend on any of these different choices.
\medskip
		
In this paper we show that the value $\Hord^{(d)}_X(\zeta)$ can be read  from the arc space of $X$ combined with the use of information provided by the asymptotic Samuel function at the maximal ideal  of the local ring at $\zeta$. 
In particular, no \'etale extensions and no local embeddings into smooth schemes are  needed:
the information is already present in the cotangent space at $\zeta$,
${\mathfrak m}_{\zeta}/{\mathfrak m}^2_{\zeta}$,  and the space of arcs in $X$ with center  at $\zeta$, ${\mathcal L}(X,\zeta)$.
		
More precisely,  on the one hand, the value $\ord^{(d)}_X(\zeta)$ can be read studying the Nash multiplicity sequences of arcs in $X$ with center $\zeta$
(this was studied in \cite{BEP2} by the last two authors in collaboration with B.~Pascual-Escudero).

On the other hand, studying the properties of the asymptotic Samuel function, we came up
with the notion of the {\em Samuel slope of a local ring ${\mathcal O}_{X,\zeta}$},
${\mathcal S}\text{-Sl}({\mathcal O}_{X,\zeta})$ (see Definition \ref{DefSSl}). For a singular point, ${\mathcal S}\text{-Sl}({\mathcal O}_{X,\zeta})\geq 1$, and we will make a distinction 
depending on whether ${\mathcal S}\text{-Sl}({\mathcal O}_{X,\zeta})=1$ (non-extremal case) or ${\mathcal S}\text{-Sl}({\mathcal O}_{X,\zeta})> 1$ (extremal case). Actually, the previous distinction can be made after analyzing properties of the cotangent space ${\mathfrak m}_{\zeta}/{\mathfrak m}_{\zeta}^2$. A combination of these pieces of information gives us enough input to compute $\Hord_X^{(d)}$.
Our results say that
$$\Hord^{(d)}_X(\zeta)=\min \{{\mathcal S}\text{-Sl}({\mathcal O}_{X,\zeta}), \ord^{(d)}_X (\zeta)\},$$
but more precisely, we can say more:
\medskip

\noindent{\bf Theorem \ref{maintheorem}.}
{\em Let $X$ be an equidimensional variety of dimension $d$ defined over a perfect field $k$. Let $\zeta\in X$ be a point of multiplicity $m>1$. Then: 
\begin{itemize}
	\item If ${\mathcal S}\text{-Sl}({\mathcal O}_{X,\zeta})=1$, then
	$$1={\mathcal S}\text{-Sl}({\mathcal O}_{X,\zeta})=\Hord^{(d)}_X(\zeta)\leq \ord^{(d)}_X(\zeta).$$
	In addition, if $\zeta$ is a closed point then also $\ord^{(d)}_X(\zeta)=1$.
	\item If ${\mathcal S}\text{-Sl}({\mathcal O}_{X,\zeta})>1$, then    
	$$\Hord^{(d)}_X(\zeta)=\min \{{\mathcal S}\text{-Sl}({\mathcal O}_{X,\zeta}), \ord^{(d)}_X (\zeta)\}.$$
\end{itemize}
}
\medskip	

We give an idea of the meaning of this result in the following lines. When the characteristic is zero,  the description of the maximum multiplicity locus of $X$ in  (\ref{local_presentation}) goes far beyond that equality.  In fact,  it can be proven that   to lower the maximum multiplicity of $X$ it suffices to work with the elimination algebra ${\mathcal R}$ (which is defined on a smooth scheme of dimension $d$). In other words, a simplification of the multiplicity of the $d$-dimension variety $X$ becomes a problem about {\em finding a resolution} of a Rees algebra defined on a smooth $d$-dimensional scheme (see Sections \ref{Rees_Algebras} and \ref{ElimAlg}). 
If $\ord^{(d)}_X(\zeta)=1$, then  this indicates that, either the multiplicity of $X$ can be lowered with a single blow up at a regular center, or else, a simplification of the multiplicity of $X$ is a problem that can be solved using certain Rees algebra defined in a $(d-1)$-dimensional smooth scheme
(see \S\ref{induccion} for details).
Thus, our original problem is, in principle, simpler to solve. 
And the theorem says that the condition $\ord^{(d)}_X(\zeta)=1$ is already encrypted in ${\mathfrak m}_{\zeta}/{\mathfrak m}_{\zeta}^2$. 
\medskip

The second part of the theorem says that the relevant information from the coefficients of the polynomials in
(\ref{polinomios}), which, in general, only exists in a suitable \'etale neighborhood of the point, can already be read through the Samuel slope of the original local ring at the singular point and the sequences of Nash multiplicities of arcs with center the given point.

\

\noindent{\bf Organization of the paper}
\medskip

\noindent
Facts on the asymptotic Samuel function are given in Section \ref{Samuel_a}, and
in addition, we study the behavior of this function when consider certain finite extension of
rings (Proposition \ref{ExtFinNuBar}).
In section \ref{Seccion3} we define the notion of the Samuel slope of a local ring,
and we study his behavior under \'etale extensions (Propositions \ref{Prop_Slope_etale} and \ref{prop_sucesiones}).
Rees algebras and their use in resolution of singularities are studied in Sections \ref{Rees_Algebras}, \ref{AlgoResol}, and \ref{ElimAlg}.
The function $\Hord^{(d)}_X$ is treated in Section \ref{presentaciones}.
The proof of the main result is addressed in Section \ref{seccion_demos}, here
our results from Section \ref{Seccion3} are needed.

\

\noindent \emph{Acknowledgments:}  We  profited from conversations with O.~E. Villamayor U.,
C.~Abad, B.~Pascual-Escudero,
 C.~del-Buey-de-Andrés, and C. Chiu. In addition, we would like to thank S.~D. Cutkosky for giving us a crucial hint that led us to a proof of  Proposition \ref{Prop_Slope_etale}.
We also want to thank to the anonymous referee for useful suggestions and comments to 
improve the paper.

\section{The asymptotic Samuel function}\label{Samuel_a}

The {\em asymptotic Samuel function} was first introduced by Samuel in \cite{Samuel} and  studied afterwards  by D. Rees in a series of papers (\cite{Rees1955}, \cite{Rees_Local},  \cite{Rees_Ideals}, \cite{Rees_Ideals_II}).
Thorough expositions on this topic can be found in \cite{LejeuneTeissier1974} and \cite{Hu_Sw},
see also \cite{Dale} for a generalization to arbitrary filtrations.
We will use $A$ to denote a commutative ring with 1.
	\begin{definition} \label{DefFuncOrd}
		A function $w:A\to\mathbb{R}\cup\{\infty\}$ is an 
		\emph{order function} if
		\begin{enumerate}
			\item[(i)]  $w(f+g)\geq \min\{w(f), w(g)\}$, for all $f,g\in A$,
			\item[(ii)]  $w(f\cdot g)\geq w(f)+w(g)$, for all $f,g\in A$,
			\item[(iii)] $w(0)=\infty$ and $w(1)=0$.
		\end{enumerate}
	\end{definition}
	
\begin{remark} \cite[Remark 0.3]{LejeuneTeissier1974}
	\label{minimo_suma} If $w$ is an order function then $w(x)=w(-x)$ and if  $w(x) \neq w(y)$ then $w(x+y)=\min\{w(x), w(y)\}$.  
\end{remark}

	\begin{example}
		Let   $I\subset A$ be a proper ideal.
		Then the function $\nu_I:A\to \mathbb{R}\cup\{\infty\}$ defined by
		\begin{equation*}
		\nu_I(f):=\sup\{m\in\mathbb{N}\mid f\in I^m\}
		\end{equation*}
		is an order function.  If $(A,{\mathfrak m})$ is a local regular ring, then $\nu_{\mathfrak m}$ is just the usual order function. 
	\end{example}

In general, for $n\in {\mathbb N}_{>1}$, the inequality $\nu_I(f^n)\geq n\nu_I(f)$ can  be strict. This can be seen for instance by considering the following example. Let $k$ be a field, and let $A=k[x,y]/\langle x^2-y^3\rangle$.  Set ${\mathfrak m}=\langle \overline{x}, \overline{y}\rangle$. Then $\nu_{\mathfrak m}(\overline{x})=1$, but $\nu_{\mathfrak m}(\overline{x}^2)=3$.  The asymptotic Samuel function is a  normalized version of the  previous  order    that gets around  this problem:
	
\begin{definition}
		Let   $I\subset A$ be a proper ideal. The {\em asymptotic Samuel function at $I$}, $\bar{\nu}_I:A\to\mathbb{R}\cup\{\infty\}$, is defined as: 
		\begin{equation}\label{Def_Sam_lim}
		\bar{\nu}_I(f)=\lim_{n\to\infty}\frac{\nu_I(f^n)}{n}, \qquad f\in A.
		\end{equation}
\end{definition}

It can be shown that the limit (\ref{Def_Sam_lim})  exists in  $\mathbb{R}_{\geq 0}\cup\{\infty\}$ for any ideal $I\subset A$  (see   \cite[Lemma 0.11]{LejeuneTeissier1974}). Again, if   $(A,{\mathfrak m})$ is a local regular ring, then $\nub_{\mathfrak m}$ is just the usual order function. As indicated before, this is an order function with   nice properties:  
	
\begin{proposition}\label{PotenciaNuBar}  
\cite[Corollary 0.16, Proposition 0.19]{LejeuneTeissier1974} 
	The function $\bar{\nu}_I$ is an order function. Furthermore,  it satisfies the following properties for each $f\in A$ and each $r\in {\mathbb N}$: 
		\begin{enumerate}
			\item[(i)] $\bar{\nu}_I(f^r)=r\bar{\nu}_I(f)$;
			\item[(ii)] $\bar{\nu}_{I^r}(f)=\dfrac{1}{r}\bar{\nu}_I(f)$.
		\end{enumerate}
		
\end{proposition}

\

\noindent{\bf The asymptotic Samuel function on Noetherian rings}
\medskip

\noindent When $A$ is Noetherian, the number $\nub_I(f)$  measures how deep the element $f$ lies in the integral closure of powers of $I$. In fact, the following results hold:

	\begin{proposition}
	\cite[Corollary 6.9.1]{Hu_Sw}
	Suppose $A$ is Noetherian. Then for a proper ideal    $I\subset A$  and every $a\in\mathbb{N}$, \begin{equation*}
	\overline{I^a}=\{f\in R \mid \bar{\nu}_I(f)\geq a\}.
	\end{equation*}
\end{proposition}

		\begin{corollary} \label{nub_integral}
		Let $A$ be a Noetherian ring and $I\subset A$ a proper ideal. If $f\in A$ then
		\begin{equation*}
		\bar{\nu}_I(f)\geq \frac{a}{b} \Longleftrightarrow f^b\in\overline{I^a}.
		\end{equation*}
	\end{corollary}
	
The previous characterization of $\nub_I$  leads to the following result that give a valuative version of the function. 

\begin{theorem}
	Let $A$ be a Noetherian ring, and let $I\subset A$ be a proper ideal not contained in a minimal prime of $A$.
	Let $v_1,\ldots,v_s$ be a set of Rees valuations of the ideal $I$.
		If $f\in A$ then
		\begin{equation*}
		\bar{\nu}_I(f)=\min\left\{\frac{v_i(f)}{v_i(I)}\mid i=1,\ldots,s\right\}.
		\end{equation*}
\end{theorem}
	
	\begin{proof} See \cite[Lemma 10.1.5, Theorem 10.2.2]{Hu_Sw} and \cite[Proposition 2.2]{Irena}. 
	\end{proof}

	\begin{remark}\label{NuBarBlupNorm} 
		Let $A$ be a Noetherian   reduced   ring,  and let $I\subset A$ be a proper ideal not contained in any minimal prime of $A$. Set $X=\Spec(A)$ and let $\overline{X}$ be the normalized blow up of $X$ at the ideal $I$. 
		Then, the sheaf of ideals $I\mathcal{O}_{\overline{X}}$ is invertible and, since $\overline{X}$ is normal, there is a finite number of reduced  and irreducible hypersurfaces $H_1,\ldots,H_{\ell}$ in $\overline{X}$,   and 
		there exists an open set $U\subset \overline{X}$, such that $\overline{X}\setminus U$ has codimension at least  $2$ such that: 
		\begin{equation*}
		I\mathcal{O}_{U}=I(H_1)^{c_1}\cdots I(H_{\ell})^{c_{\ell}}\vert_U
		\end{equation*}
		for some   integers $c_1,\ldots,c_{\ell}\in {\mathbb Z}_{\geq 1}$. 
		Denote by  $v_i$ the valuation associated to $\mathcal{O}_{\overline{X},h_i}$, where $h_i$ is the generic point of $H_i$. Then note that a subset of $\{v_1,\ldots,v_{\ell}\}$ has to be a Rees valuation set of $I$. Therefore,  if $f\in A$ then 
		\begin{equation*}
		\bar{\nu}_I(f)=\min\left\{\frac{v_i(f)}{v_i(I)}\mid i=1,\ldots,\ell\right\}.
		\end{equation*}
		See \cite[Theorem 10.2.2]{Hu_Sw} and \cite[Theorem 2.1, Proposition 2.2]{Irena}. 
\end{remark} 

As an application of Remark~\ref{NuBarBlupNorm} we can prove the following result about the behavior of the $\nub$ function on products of elements.
This will be used in the proof of Theorem~\ref{maintheorem}.
\begin{proposition} \label{ExtFinNuBar}
Let $A\to C$ be ring homomorphism of Noetherian rings, where $A$ is regular and $C$ is reduced and equidimensional. Let $Q(A)$ be the fraction field of $A$.  Suppose that no non-zero element of $A$ maps to a zero divisor in $C$,  and that  the extension $Q(A)\to Q(A)\otimes_AC$ is finite.  
Let $\mathfrak{q}\in\Spec(C)$ and $\mathfrak{n}=\mathfrak{q}\cap A$.
Assume that ${\mathfrak n}C$ is a reduction of ${\mathfrak q}\subset C$, and
that $A/\mathfrak{n}$ is regular. If $a\in A$ and $f\in C$
then:
\begin{equation*}
\bar{\nu}_{\mathfrak{q}}(a)=\bar{\nu}_{\mathfrak{n}}(a), \ \  \text{ and }  \  \   \bar{\nu}_{{\mathfrak q}}(af)=\bar{\nu}_{{\mathfrak q}}(a)+\bar{\nu}_{{\mathfrak q}}(f).
\end{equation*}
\end{proposition}

\begin{proof}  
Set $X=\Spec(C)$ and $Z=\Spec(A)$.
Let $\overline{X}$ be the normalized blow up of $X$ at the ideal ${\mathfrak q}$ and let
$\overline{Z}$ be the blow up of $Z$ at ${\mathfrak n}$. Then there is  a  commutative diagram
		\begin{equation*}
		\xymatrix@R=1pc{
			X \ar[d] & \ar[l] \overline{X} \ar[d] \\ Z  &  \ar[l] \overline{Z},
		}
		\end{equation*}
		(see \cite[Lemma 4.2]{COA}). 
The exceptional divisor $E$ of the blow up $\overline{Z}\to Z$ defines only a valuation $v_0$ in $A$. The exceptional divisor of $\overline{X}\to X$ defines valuations $v_1,\ldots,v_{\ell}$ 
as in Remark~\ref{NuBarBlupNorm}.
Note that every valuation $v_i$ is an extension of $v_0$ to $C$.
Then if $a\in A$: 
\begin{equation*}
\bar{\nu}_{\mathfrak n}(a)=\frac{v_0(a)}{v_0({\mathfrak n})}=
\frac{v_i(a)}{v_i({\mathfrak n})}=\bar{\nu}_{\mathfrak{q}}(a) \qquad \text{for all } i=1,\ldots,\ell.
\end{equation*}
On the other hand, for each $i\in \{1,\ldots,\ell\}$, 
\begin{equation*}
\frac{v_i(af)}{v_i(\mathfrak{q})}=\frac{v_i(a)}{v_i(\mathfrak{q})}+\frac{v_i(f)}{v_i(\mathfrak{q})}=\bar{\nu}_{\mathfrak{q}}(a)+\frac{v_i(f)}{v_i(\mathfrak{q})}.
\end{equation*}
And, again, by  Remark \ref{NuBarBlupNorm} be have the required equality.
\end{proof}

\begin{parrafo}
\textbf{Notation.}
Along this paper we will interested in computing the function order $\nub$ at points $\zeta$ in a variety $X$ over a field $k$. We will be distinguishing between $\nub_{\zeta}$ and $\nub_{\mathfrak{p}_{\zeta}}$
where $\mathfrak{p}_{\zeta}$ is the prime defining $\zeta$ in an affine open set of $X$.
In the first case, for an element $f\in\calo_{X,\zeta}$, $\nub_{\zeta}(f)$ is computed using the function $\nub$ for the local ring $\calo_{X,\zeta}$ at the maximal ideal
$\mathfrak{m}_{\zeta}=\mathfrak{p}_{\zeta}\calo_{X,\zeta}$.
In the second case, for an element $f\in B$, where $\Spec(B)\subset X$ is an affine open containing $\zeta$, $\nu_{\mathfrak{p}_{\zeta}}$ is computed using the function $\nub$ for the ring $B$ at the prime ideal $\mathfrak{p}_{\zeta}$.
Note that 
$\nub_{\zeta}(f)\geq \nub_{\mathfrak{p}_{\zeta}}(f)$.
If the local ring $\calo_{X,\zeta}$ is regular then we will use the standard notation $\nu_{\zeta}$ for the usual order function, and then
$\nu_{\zeta}=\nub_{\zeta}$.
\end{parrafo}	

\section{The Samuel slope of a local ring} \label{Seccion3}

Let $(A,\mathfrak{m})$ be a local Noetherian ring.
We will focus on some elements in the associated graded ring $\Gr_{\mathfrak{m}}(A)$ which are nilpotent.
They will be used to define the Samuel slope of the local ring.

\begin{parrafo} \label{DefLambdaBig} {\bf Degree one nilpotents in $\Gr_{\mathfrak{m}}(A)$.} \cite[\S 0.7, \S 0.21 and  \S  0.22]{LejeuneTeissier1974}
For a local ring $(A,\mathfrak{m})$, consider
	\begin{equation*}
	\mathfrak{m}^{(\geq 1)}:=\{g\in A\mid \bar{\nu}_{\mathfrak{m}}(g)\geq 1\}, \qquad \text{ and } \qquad 
	\mathfrak{m}^{(> 1)}:=\{g\in A\mid \bar{\nu}_{\mathfrak{m}}(g)> 1\}.
	\end{equation*}
	Note that $\mathfrak{m}^{(\geq 1)}$ and $\mathfrak{m}^{(>1)}$ are ideals in $A$.
There is a natural morphism of $k(\mathfrak{m})$-vector spaces, 
$$	\begin{array}{rccl}
	\lambda_{\mathfrak{m}}: & \mathfrak{m}/\mathfrak{m}^2 & \longrightarrow & 
	\mathfrak{m}^{(\geq 1)}/\mathfrak{m}^{(>1)} \\
 &   f+\mathfrak{m}^{2} & \mapsto  & 	\lambda_{\mathfrak{m}}(f+\mathfrak{m}^{2}):=f+\mathfrak{m}^{(>1)}, 
	\end{array}$$
whose  kernel is the subspace generated by the degree one nilpotents of $\Gr_{\mathfrak{m}}(A)$. 
\end{parrafo}

\begin{remark} \label{RemNilpotent}
If $A$ is a local regular ring, then $\nub_{\mathfrak{m}}=\nu_{\mathfrak{m}}$ is the usual order function and
$\lambda_{\mathfrak{m}}$ is an isomorphism.
If $A$ is not regular, then we have that 
$\dim_{k(\mathfrak{m})}{\mathfrak m}/{\mathfrak m}^2=d+t$, with $t>0$
being the excess of the embedding dimension of $(A,\mathfrak{m})$.
Note that
$d=\dim(A)=\dim(\Gr_{\mathfrak{m}}(A))=\dim(\Gr_{\mathfrak{m}}(A))_{\text{red}}$.
Therefore, if $x_1,\ldots,x_{d+t}\in\mathfrak{m}$ is a minimal set of generators, then there are at least $d$
elements $x_{i_1},\ldots,x_{i_d}$, such that their classes in $\Gr_{\mathfrak{m}}(A)$ are 
not nilpotent.
This means that $\nub_{\mathfrak{m}}(x_{i_j})=1$, for $j=1,\ldots,d$.
\end{remark}

Assume that $\nub_{\mathfrak{m}}(x_{1})=\ldots=\nub_{\mathfrak{m}}(x_{d})=1$.
The minimum of $\nub_{\mathfrak{m}}(x_{d+1}),\ldots,\nub_{\mathfrak{m}}(x_{d+t})$
defines a slope with respect to the chosen generators.
The Samuel slope is the supremum of all these possible coordinate dependent slopes.
\begin{definition}\label{DefSSl}
Let $(A,\mathfrak{m})$ is a Noetherian local ring of dimension $d$ and embedding dimension $d+t$,
with $t>0$.
Let $\mathbf{x}=\{x_1,\ldots,x_{d+t}\}\subset\mathfrak{m}$ be a minimal set of generators of $\mathfrak{m}$.
We define the \emph{slope with respect to $\mathbf{x}$}  as
$$\text{Sl}_{\mathbf{x}}(A):=\min\{\nub_{\mathfrak{m}}(x_{d+1}),\ldots,\nub_{\mathfrak{m}}(x_{d+t})\}.$$
The \emph{Samuel slope of the local ring $A$} is 
\begin{equation*}
{\mathcal S}\text{-Sl}(A):=
\sup\limits_{\mathbf{x}}\text{Sl}_{\mathbf{x}}(A)=
\sup\limits_{\mathbf{x}} \left\{
\min\left\{\bar{\nu}_{\mathfrak{m}}(x_{d+1}),\ldots,\bar{\nu}_{\mathfrak{m}}(x_{d+t})\right\} 
\right\},
\end{equation*}
where the supremum is taken over all possible minimal set of generators $\mathbf{x}$ of $\mathfrak{m}$.
\end{definition}

\begin{example}
Let $R=k[x_1,x_2,x_3]_{\langle x_1,x_2,x_3\rangle}$, set 
$A=R/\langle x_2^2+x_1^5,\, x_3^2+x_1^7\rangle$, and let $\mathfrak{m}\subset A$ be the maximal ideal.
Then 
$\nub_{\mathfrak{m}}(x_1)=1$, $\nub_{\mathfrak{m}}(x_2)=5/2$ and $\nub_{\mathfrak{m}}(x_3)=7/2$.
It can be checked that ${\mathcal S}\text{-Sl}(A)=5/2$.
\end{example}

\begin{remark} \label{cota_dim_ker_b}
Let $\Gamma$ be the set of all possible minimal ordered sets of generators $\mathbf{x}$ of $\mathfrak{m}$.
For $\mathbf{x}=\{x_1,\ldots,x_{d+t}\}\in\Gamma$ let
$\alpha(\mathbf{x}):= \#\{i \mid \nub_{\mathfrak{m}}(x_i)>1\}$.
Note that 
$$r_{\mathfrak{m}}:=\dim_{k(\mathfrak{m})}\ker(\lambda_{\mathfrak{m}})=
\max\left\lbrace \alpha(\mathbf{x}) \mid \mathbf{x}\in\Gamma \right\rbrace.$$

Since, by Remark \ref{RemNilpotent}, in any set of minimal generators there are at
least $d$ elements with $\nub_{\mathfrak{m}}(x_i)=1$, we have that
$$0\leq \dim_{k(\mathfrak{m})}\ker(\lambda_{\mathfrak{m}})\leq t.$$ 
\end{remark}

\begin{definition}\label{Def_lambda_sequence}  Let $(A,\mathfrak{m})$ be a Noetherian local ring.
Suppose that the embedding dimension of $(A,\mathfrak{m})$ is $d+t$ with $t>0$.
We say that $(A,\mathfrak{m})$ is in the \emph{extremal case} if $\dim \ker(\lambda_{\mathfrak{m}})=t$.
Otherwise we say that $(A,\mathfrak{m})$ is in the \emph{non-extremal case}.
If $\dim \ker(\lambda_{\mathfrak{m}})=t$, then we say that a sequence of elements 
$\gamma_1,\ldots,\gamma_t\in {\mathfrak m}$ is a {\em $\lambda_{\mathfrak{m}}$-sequence} if their classes $\bar{\gamma}_i\in\mathfrak{m}/\mathfrak{m}^2$ form a basis of $\ker(\lambda_{\mathfrak{m}})$. 
In other words, $\gamma_1,\ldots,\gamma_t\in {\mathfrak m}$ is a $\lambda_{\mathfrak{m}}$-sequence if
their classes in $\Gr_{\mathfrak{m}}(A)$ are nilpotent and $\gamma_1,\ldots,\gamma_t$  are part of a 
minimal set of generators of $\mathfrak{m}$.
\end{definition} 
	
\begin{remark}\label{Def_Sam_Slope}  
Let $(A,\mathfrak{m})$ be a Noetherian local ring. 
Suppose that the embedding dimension of $(A,\mathfrak{m})$ is $d+t$ with $t>0$.
We can express the Samuel slope in terms of 
$\lambda_{\mathfrak{m}}$-sequences as follows: 
	\begin{itemize}
		\item If $\dim\ker(\lambda_{\mathfrak{m}})<t$ (non-extremal case), then  
		${\mathcal S}\text{-Sl}(A)=1$;
		\item If $\dim\ker(\lambda_{\mathfrak{m}})=t$ (extremal case), then  
		\begin{equation*}
		{\mathcal S}\text{-Sl}(A)=
		\sup\limits_{{\lambda_{\mathfrak{m}}}{\text{-sequence}}} \left\{
		\min\left\{\bar{\nu}_{\mathfrak{m}}(\gamma_1),\ldots,\bar{\nu}_{\mathfrak{m}}(\gamma_t)\right\} 
		\right\},
		\end{equation*}
where the supremum is taken over all the $\lambda_{\mathfrak{m}}$-sequences in the local ring $(A,\mathfrak{m})$.
\end{itemize}
\end{remark}

\begin{remark}
Suppose that $X$ is an equidimensional variety of dimension $d$ defined over a perfect field $k$, and $\zeta\in X$ a (non-necessarily closed) point of multiplicity $m>1$.
Set $d_{\zeta}=\dim(\calo_{X,\zeta})$ and
$d_{\zeta}+t_{\zeta}=\dim_{k(\zeta)}(\mathfrak{m}_{\zeta}/\mathfrak{m}_{\zeta}^2)$ be the embedding dimension at $\zeta$, where $k(\zeta)$ denotes the residue field of $\calo_{X,\zeta}$.
The Samuel slope of $X$ at $\zeta$ is the Samuel slope of the local ring $\calo_{X,\zeta}$,
and a $\lambda_{\zeta}$-sequence will be a
$\lambda_{\mathfrak{m}_{\zeta}}$-sequence. 
\end{remark}
\medskip

\noindent
{\bf The Samuel slope and \'etale extensions}.
\medskip

\noindent
To prove Theorem \ref{maintheorem} we will have to work in an
\'etale neighborhood of a given point.
To be able to use \'etale extensions in our arguments, we will  first prove that the dimension of  $\ker(\lambda_{\zeta})$ is an invariant under  such  extensions. 
From here, it follows that if $X'\to X$ is an \'etale morphism mapping $\zeta'\in X'$ to $\zeta\in X$ then 
$\mathcal{S}\text{-Sl}(\mathcal{O}_{X,\xi})\leq \mathcal{S}\text{-Sl}(\mathcal{O}_{X',\xi'})$.
We do not know if the equality holds in general.
However we can prove it for some special cases, which will be enough for our purposes.
\medskip

\begin{lemma}\label{nucleo_etale}
Let $\varphi:(A,\mathfrak{m})\to (A',\mathfrak{m}')$ be an \'etale homomorphism of Noetherian local rings.
Then
\begin{equation*}
r_{\mathfrak{m}}=\dim_{k(\mathfrak{m})}\ker(\lambda_{\mathfrak{m}})=r_{\mathfrak{m}'}=\dim_{k(\mathfrak{m}')}\ker(\lambda_{\mathfrak{m}'}).
\end{equation*}
\end{lemma}

\begin{proof}
Let $\mathcal{N}$ (resp. $\mathcal{N}'$) be the nilradical of $\Gr_{\mathfrak{m}}(A)$
(resp. of $\Gr_{\mathfrak{m}'}(A')$).
Note that $\Gr_{\mathfrak{m}'}(A')=k(\mathfrak{m}')\otimes\Gr_{\mathfrak{m}}(A)$ 
is an \'etale extension of $\Gr_{\mathfrak{m}}(A)$.
Therefore we have that $\mathcal{N}'=\mathcal{N}\Gr_{\mathfrak{m}'}(A')$. 
The lemma follows since $\ker(\lambda_{\mathfrak{m}})=(\mathcal{N}+\mathfrak{m}^2)/\mathfrak{m}^2$.
\end{proof}

\begin{proposition} \label{Prop_Slope_etale} 
Let $\varphi:(A,\mathfrak{m})\to (A',\mathfrak{m}')$ be an \'etale homomorphism of Noetherian local rings.
If $k(\mathfrak{m})=k(\mathfrak{m}')$, then
\begin{equation*}
\mathcal{S}\text{-Sl}(A)=\mathcal{S}\text{-Sl}(A').
\end{equation*}
\end{proposition}

\begin{proof} 
Let $d$ be the Krull dimension of $A$.
Suppose that $\dim_{k(\mathfrak{m})}{\mathfrak m}/{\mathfrak m}^2=d+t$, with $t>0$.
By Lemma \ref{nucleo_etale}, the result is immediate if $\dim\ker(\lambda_{\mathfrak{m}})<t$, and in fact, in this case, the hypothesis  $k(\mathfrak{m})=k(\mathfrak{m}')$   is not needed. 

Suppose now that $\dim\ker(\lambda_{\mathfrak{m}})=t$.
Since $k(\mathfrak{m})=k(\mathfrak{m}')$, it follows that  
$\Gr_{\mathfrak{m}}(A)=\Gr_{\mathfrak{m}'}(A')$. 
Observe that 
if $\theta'\in\mathfrak{m}'$ then, for each $n\in\mathbb{N}$, there exists some $\rho_n\in\mathfrak{m}$ 
such that $\theta'-\rho_n\in(\mathfrak{m}')^n$.
This means that there is some  $n\gg 0$ such that $\bar{\nu}_{\mathfrak{m}}(\rho_n)=\bar{\nu}_{\mathfrak{m}'}(\theta')$.
From here we can conclude  that given a $\lambda_{\mathfrak{m}'}$-sequence 
$\theta'_1,\ldots,\theta'_t\in\mathfrak{m}'$ 
we can always find  $\theta_1,\ldots,\theta_t\in\mathfrak{m}$ such that :
\begin{itemize}
	\item $\theta_1,\ldots,\theta_t$ is a $\lambda_{\mathfrak{m}}$-sequence of $(A,\mathfrak{m})$ and
	\item $\bar{\nu}_{\mathfrak{m}}(\theta_i)=\bar{\nu}_{\mathfrak{m}'}(\theta'_i)$ for $i=1,\ldots,t$.
\end{itemize}
The result now follows by the definition of the Samuel slope and Remark \ref{Def_Sam_Slope}.
\end{proof}

The following result will allow us to compare the Samuel slope of a local ring, at a non closed point of a variety, before and after an \'etale extension (at least under some special assumptions). This will be used in the proof of Theorem \ref{maintheorem}.

\begin{proposition}\label{prop_sucesiones}
Let $(A,\mathfrak{m})$ be a formally $d$-equidimensional local Noetherian ring.
Let $\mathfrak{p}\subset A$ be a prime ideal such that the quotient ring $A/\mathfrak{p}$ is a $(d-r)$-dimensional regular ring, with $r>0$, and $\mult_{\mathfrak{m}}(A)=\mult_{\mathfrak{p}A_{\mathfrak{p}}}(A_{\mathfrak{p}})=m>1$.
Suppose that:
    \begin{itemize}
        \item the excess of embedding dimension of $(A,\mathfrak{m})$ is $t$ and  coincides with the excess of embedding dimension of $(A_{\mathfrak{p}},\mathfrak{p}A_{\mathfrak{p}})$;
        \item both $(A,\mathfrak{m})$ and  $(A_{\mathfrak{p}},\mathfrak{p}A_{\mathfrak{p}})$ are in the extremal case.
    \end{itemize}
Let $\varphi:(A,\mathfrak{m})\to (A',\mathfrak{m}')$ be an \'etale homomorphism of local rings, and $\mathfrak{p}'\subset A'$ be a prime ideal such that $\mathfrak{p}'\cap A=\mathfrak{p}$.
Assume that:
\begin{itemize}
    \item $k(\mathfrak{m})=k(\mathfrak{m}')$;
    \item there is $\lambda_{\mathfrak{m}'}$-sequence at $A'$, $\gamma_1',\ldots,\gamma_t'$, that is also a $\lambda_{\mathfrak{p}'A'_{\mathfrak{p}'}}$-sequence.
\end{itemize}
Then there is a $\lambda_{\mathfrak{p}A_{\mathfrak{p}}}$-sequence at $A$, $\gamma_1,\ldots, \gamma_t$, such that:
$$\min_i\{\nub_{\mathfrak{p}}(\gamma_i)\}\geq
 \min_i\{\nub_{\mathfrak{p}'}(\gamma_i')\},
\qquad \text{and}\qquad
\min_i\{\nub_{\mathfrak{p}A_{\mathfrak{p}}}(\gamma_i)\}\geq
 \min_i\{\nub_{\mathfrak{p}'A'_{\mathfrak{p}'}}(\gamma_i')\}.$$
In particular $\SSl(A_{\mathfrak{p}})\geq \min_i\{\nub_{\mathfrak{p}'A'_{\mathfrak{p}'}}(\gamma_i')\}$.
\end{proposition}

\begin{proof}
We divide the proof in three steps:

\

{\bf {\em Step 1.}}  We claim that there are elements $y_1, \ldots, y_{r}, y_{r+1},\ldots, y_{d}\in A'$ such that
$$\mathfrak{m}'=\langle y_1,\ldots, y_d, \gamma_1',\ldots,\gamma_t'\rangle \  \  \text{ and } \ \ \mathfrak{p}'=\langle y_1,\ldots, y_r, \gamma_1',\ldots,\gamma_t'\rangle.$$

To prove the claim observe first that $\overline{A'}:=A'/\mathfrak{p}'$ is a regular local ring of dimension $(d-r)$. Therefore, we have that $\overline{\mathfrak{m}}':=\mathfrak{m}'/\mathfrak{p}'=\langle  \overline{y}_{r+1},\ldots, \overline{y}_d\rangle$ for some $ \overline{y}_{r+1},\ldots, \overline{y}_d\in \overline{A'}$. Thus
$$\mathfrak{m}'=\mathfrak{p}'+\langle y_{r+1},\ldots, y_d\rangle,$$
where  $y_{r+1},\ldots, y_d\in A'$ are liftings of $\overline{y}_{r+1},\ldots,\overline{y}_{d}$.
Notice that $\nub_{\mathfrak{m}'}(y_i)=1$ for $i=r+1,\ldots, d$ (because this is so at $\overline{A'}$). Since $\gamma_i'\in\mathfrak{p}'$ and  $\nub_{\mathfrak{m}'}(\gamma_i')>1$, we should be able to find $r$ elements, $y_1,\ldots, y_r$
in $\mathfrak{p}'$,  with $\nub_{\mathfrak{m}'}(y_i)=1$ and so  that,
$$\mathfrak{m}'=\langle y_1,\ldots, y_d\rangle+\langle \gamma_1',\ldots,\gamma_t'\rangle.$$
Now we have that,
$$ \langle y_1,\ldots, y_r\rangle+\langle \gamma_1',\ldots,\gamma_t'\rangle \subset\mathfrak{p}'.$$
To see that the last containment is an equality it suffices to prove that  ${\mathfrak q}:=\langle y_1,\ldots, y_r\rangle+\langle \gamma_1',\ldots,\gamma_t'\rangle$ is prime and that  it defines a $(d-r)$-dimensional closed subscheme at $\Spec(B')$.
But this is immediate since 
$$d-r=\dim(A'/\mathfrak{p}')\leq \dim(A'/\mathfrak{q})\leq d-r,$$
where that last inequality follows because $\mathfrak{m}'/\mathfrak{q}$ is generated by classes of
$y_{r+1},\ldots,y_d$.

\medskip

{\bf {\em Step 2.}} Consider the surjective morphism of graded $k(\mathfrak{m}')$-algebras:
 $$\xymatrix{  {\mathfrak D}':=\mathop{\oplus}\limits_{n\geq 0} {\mathfrak{p}'}^n/{\mathfrak{p}'}^n{\mathfrak m}' [T_{r+1},\ldots, T_{d}] \ar[rr]^{\psi'}   &  & 
 {\mathfrak C}':=\mathop{\oplus}\limits_{n\geq 0} {\mathfrak{m}'}^n/{\mathfrak{m}'}^{n+1} \ar[r]  & 0.
}$$
where the $T_i$ are variables mapping to the class of $y_i$ in ${\mathfrak m}'/ {\mathfrak{m}'}^{2}$, for $i=r+1,\ldots,d$.
We claim that
\begin{equation}
\label{nil_A_bis}
\text{Nil}({\mathfrak D'})=\langle  [{\gamma_1'}]_{{\mathfrak D}'},\ldots, [{\gamma_t'}]_{{\mathfrak D}'}\rangle,
\end{equation}
where $[{\gamma_i'}]_{{\mathfrak D}'}$ denotes the class of $\gamma_i'$ in
$\mathfrak{p}'/\mathfrak{p}'\mathfrak{m}'$ for $i=1,\ldots, t$, and that
\begin{equation}
\label{nil_B_bis}\text{Nil}({\mathfrak C}')=\langle [{\gamma_1'}]_{{\mathfrak C}'},\ldots, [{\gamma_t'}]_{{\mathfrak C}'}\rangle,
\end{equation}
where $[{\gamma_i'}]_{{\mathfrak C}'}$ denotes the class of $\gamma'_i$ in
$\mathfrak{m}'/{\mathfrak{m}'}^2$ for $i=1,\ldots, t$.
\medskip

To prove the claim, consider the ring of polynomials in $d$ variables over $k(\mathfrak{m}')$ localized at the origin, $T:=k(\mathfrak{m}')[x_1,\ldots,x_d]_{\langle x_1,\ldots,x_d\rangle}$,  and the morphism of $k(\mathfrak{m}')$-algebras,
\begin{equation}
\label{morfismo_T_B}
\xymatrix@R=1pt{  T=k(\mathfrak{m}')[x_1,\ldots,x_d]_{\langle x_1,\ldots,x_d\rangle}  \ar[rr] & & A' \\
	x_i \ar@{|->}[rr] &   & y_i, }
\end{equation}
(here  we are using the notation from step 1). Setting ${\mathfrak n}:=\langle x_1,\ldots, x_d\rangle\subset T$, the previous  morphism induces another morphism of $k(\mathfrak{m}')$-algebras between the graded rings, $\text{Gr}_{\mathfrak n}(T)$ and $\text{Gr}_{\mathfrak{m}'}(A')$,
\begin{equation}
\label{morfismo_D_C_prima}
\xymatrix@R=1pt{  \mathfrak{T}:=\text{Gr}_{\mathfrak n}(T) \ar[rr]^{\rho_{\xi'}} & & \text{Gr}_{\mathfrak{m}'}(A')={\mathfrak C}' \\
	[x_i]_1 \ar@{|->}[rr] &   & [y_i]_1,}
\end{equation}
where $[x_i]_1$ (resp. $[y_i]_1$) denotes the class of $x_i$ at ${\mathfrak n}/{\mathfrak n}^2$ (resp.  $\mathfrak{m}'/{\mathfrak{m}'}^2$) for $i=1,\ldots, d$. Via this morphism, $\text{Gr}_{\mathfrak{m}'}(A')$  is a finite extension of $\text{Gr}_{\mathfrak n}(T)$   (here we use the fact that the $\gamma_i'$ define nilpotents at $\text{Gr}_{\mathfrak{m}'}(A'))$.

Now  set  $\mathfrak{b}:=\langle x_1,\ldots, x_r\rangle\subset T$.
Then  we have the following commutative diagram of graded rings:
$$\xymatrix{  \mathfrak{D}'=\mathop{\oplus}\limits_{n\geq 0} {\mathfrak{p}'}^n/{\mathfrak{p}'}^n{\mathfrak m}' [T_{r+1},\ldots, T_{d}] \ar[rr]^{\qquad \psi'}   &  & 
\mathfrak{C}'=\mathop{\oplus}\limits_{n\geq 0} {\mathfrak{m}'}^n/ {\mathfrak{m}'}^{n+1} \ar[r]  & 0 \\
    \mathfrak{F}:=\mathop{\oplus}\limits_{n\geq 0} \mathfrak{b}^n/\mathfrak{b}^n\mathfrak{n} [T_{r+1},\ldots, T_{d}] \ar[rr]^{\qquad \phi} \ar[u]_{\rho_{\eta'}}  & &  
    \mathfrak{T}=\mathop{\oplus}\limits_{n\geq 0} {\mathfrak n}^n/ {\mathfrak n}^{n+1}    \ar[u]_{\rho_{\xi'}} \ar[r]  & 0.
}$$
By \cite[\S 5, Theorem 5]{Lipman}, $\ker(\psi')$ is nilpotent. Observe that $\phi$ is an isomorphism and  that ${\mathfrak D}'$ is a finite extension of ${\mathfrak F}$ (here we use the fact that each  $[\gamma_i']_{{\mathfrak C}'}$ is nilpotent at ${\mathfrak C}'$ and that $\ker(\psi')$ is nilpotent: hence  each $[\gamma_i']_{{\mathfrak D}'}$ is nilpotent at ${\mathfrak D}'$). Thus
$$
\langle  [{\gamma_1'}]_{{\mathfrak D}'},\ldots, [{\gamma_t'}]_{{\mathfrak D}'}\rangle\subset \text{Nil}({\mathfrak D'}) \ \ \text{ and } \ \
\langle [{\gamma_1'}]_{{\mathfrak C}'},\ldots, [{\gamma_t'}]_{{\mathfrak C}'}\rangle\subset \text{Nil}({\mathfrak C}').$$
To check that the containments are equalities it suffices to observe  that
$${\mathfrak D}'/\langle  [{\gamma_1'}]_{{\mathfrak D}'},\ldots, [{\gamma_t'}]_{{\mathfrak D}'}\rangle\simeq {\mathfrak F} \ \ \text{ and } \ \
{\mathfrak C}'/\langle  [{\gamma_1'}]_{{\mathfrak D}'},\ldots, [{\gamma_t'}]_{{\mathfrak D}'}\rangle\simeq {\mathfrak T}.$$
\medskip

{\bf {\em Step 3.}} Consider the commutative diagram,
$$\xymatrix{   {\mathfrak D}'=\mathop{\oplus}\limits_{n\geq 0} {\mathfrak{p}'}^n/{\mathfrak{p}'}^n \mathfrak{m}' [T_{r+1}',\ldots, T_{d}'] \ar[rr]^{\psi'}   &  &
 {\mathfrak C}'=\mathop{\oplus}\limits_{n\geq 0} {\mathfrak{m}'}^n/ {\mathfrak{m}'}^{n+1} \ar[r]  & 0 \\
 {\mathfrak D}:= \mathop{\oplus}\limits_{n\geq 0} {\mathfrak{p}}^n/{\mathfrak{p}}^n\mathfrak{m} [T_{r+1},\ldots, T_{d}] \ar[rr]^{\psi} \ar[u]   & & 
 {\mathfrak C}:=\mathop{\oplus}\limits_{n\geq 0} {\mathfrak{m}}^n/{\mathfrak{m}}^{n+1} \ar[r] \ar@{=}[u]  & 0,
}$$
paying attention to the  sequence for the $n$-th degree  part from  ${\mathfrak D}'$ and ${\mathfrak D}$:
$$ \xymatrix{
[\mathfrak{D}']_n= {\mathfrak{p}'}^n/{\mathfrak{p}'}^n \mathfrak{m}'
   \ar[r]^-{\epsilon_n'} \ar@/^2pc/[rr]^{[\psi']_n} &
    {\mathfrak{m}'}^n/{\mathfrak{p}'}^n{\mathfrak m}'   \ar[r]^-{\pi_n'} &
      {\mathfrak{m}'}^n/{\mathfrak{m}'}^{n+1} \ar[r] & 0 \\
[\mathfrak{D}]_n= {\mathfrak{p}}^n/{\mathfrak{p}}^n \mathfrak{m}
\ar[r]^-{\epsilon_n} \ar@/_2pc/[rr]_{[\psi]_n} \ar[u] &
 {\mathfrak m}^n/{\mathfrak p}^n{\mathfrak m} \ar[u]_{\iota_n} \ar[r]^-{\pi_n} &
   {\mathfrak m}^n/{\mathfrak m}^{n+1} \ar[r] \ar@{=}[u] & 0 }     $$

Now, for each $i\in \{1,\ldots, t\}$, choose $[\kappa_{i,1}]_1 \in {\mathfrak m}/{\mathfrak p}{\mathfrak m}$ so that $\pi_1([{\kappa_{i,1}}]_1)= [{\gamma_i'}]_{{\mathfrak C}'}\in {\mathfrak m}/ {\mathfrak m}^2={\mathfrak{m}'}/ {\mathfrak{m}'}^2$.
Then if  $[\gamma_i']_1$ denotes the class of $\gamma_i'$ in $ {\mathfrak{m}'}/{\mathfrak{p}'}{\mathfrak{m}'} $, we have that  $[\gamma_i']_1-\iota_1([{\kappa_{i,1}}]_1)\in \ker(\pi_1')=\epsilon'_1(\ker(\psi_1'))$.
Notice that from Step 2 and \cite[\S 5, Theorem 5]{Lipman}, it follows that $\ker(\pi')\subset \langle [\gamma_1']_1,\ldots, [\gamma_t']_1\rangle$.
\medskip

Thus, selecting  $\kappa_{i,1}\in A$ as some lifting of $[{\kappa_{i,1}}]_1$ we have that
$$\gamma_i'-\kappa_{i,1}\in \langle \gamma_1', \ldots, \gamma_t'\rangle +\alpha_{i,2}$$
for some  $\alpha_{i,2}\in {\mathfrak{p}'}{\mathfrak{m}'}$. Notice that it follows from here that $\kappa_{i,1}\in {\mathfrak p}$.
\medskip

Since
$\alpha_{i,2}\in {\mathfrak{p}'}{\mathfrak{m}'}\subset {\mathfrak{m}'}^2$ we now choose $[\kappa_{i,2}]_2\in
{\mathfrak{m}}^2/{\mathfrak{p}}^2{\mathfrak{m}}$
so that
$$\pi_2([\kappa_{i,2}]_2)=[\alpha_{i,2}]_{\mathfrak C}\in  {\mathfrak{m}}^2/ {\mathfrak{m}}^3={\mathfrak{m}'}^2/ {\mathfrak{m}'}^3.$$
Then $$[\alpha_{i,2}]_{2}-\iota_2([\kappa_{i,2}]_2)\in \ker(\pi'_2)=\epsilon'_2(\ker\psi'_2).$$
And, selecting some lifting $\kappa_{i,2}\in A$ of $[\kappa_{i,2}]_2$, we have that
$$\alpha_{i,2}-\kappa_{i,2}\in \langle \gamma_1',\ldots, \gamma_t' \rangle + \alpha_{i,3} $$
with $\alpha_{i,3}\in {\mathfrak{p}'}^2{\mathfrak{m}'}$.
From here it follows that $\kappa_{i,2}\in\mathfrak{p}$.
Iterating  this procedure, we find that
$$\gamma_i'-(\kappa_{i,1}+\kappa_{i,2}+\ldots+ \kappa_{i,n})\in \langle \gamma_1',\ldots, \gamma_t'\rangle +\alpha_{i,n}$$
with $\alpha_{i,n}\in  {\mathfrak{p}'}^n{\mathfrak{m}'}$, and $\kappa_{i,j}\in\mathfrak{p}$.
Taking $n\gg 0$, and setting
$$\gamma_i:=\kappa_{i,1}+\ldots+\kappa_{i,n}$$
we have that $\gamma_i\in {\mathfrak{p}}$ for $i=1,\ldots,t$,   that:
$$\min_{i=1,\ldots,t}\left\{ \nub_{\mathfrak{p}}(\gamma_i)\right\}=
\min_{i=1,\ldots,t}\left\{ \nub_{\mathfrak{p}'}(\gamma_i)\right\}\geq
\min_{i=1,\ldots,t}\left\{ \nub_{\mathfrak{p}'}(\gamma_i')\right\}>1, $$
$$\min_{i=1,\ldots,t}\left\{ \nub_{\mathfrak{p}A_{\mathfrak{p}}}(\gamma_i)\right\}=
\min_{i=1,\ldots,t}\left\{ \nub_{\mathfrak{p}'A'_{\mathfrak{p}'}}(\gamma_i)\right\}\geq
\min_{i=1,\ldots,t}\left\{ \nub_{\mathfrak{p}'A'_{\mathfrak{p}'}}(\gamma_i')\right\}>1, $$
and that
$$\min_{i=1,\ldots,t}\left\{ \nub_{\mathfrak{m}}(\gamma_i)\right\}=
\min_{i=1,\ldots,t}\left\{ \nub_{\mathfrak{m}'}(\gamma_i)\right\}\geq 
\min_{i=1,\ldots,t}\left\{ \nub_{\mathfrak{m}'}(\gamma_i')\right\}>1. $$
In particular, $\nub_{\mathfrak{p}A_{\mathfrak{p}}}(\gamma_i)>1$ and $\nub_{\mathfrak{m}}(\gamma_i)>1$
for $i=1,\ldots, t$. By construction,
$$[\gamma_i]_{\mathfrak C}= [\kappa_{i,1}]_{\mathfrak C}=[\gamma_i']_{{\mathfrak C}'},$$ from where it follows that
$\gamma_1,\ldots, \gamma_t\in\mathfrak{m}$ form both a $\lambda_{\mathfrak{m}}$-sequence and $\lambda_{\mathfrak{m}'}$-sequence.
We also have that
$${\mathfrak{m}'}=\langle y_1,\ldots, y_d\rangle +\langle \gamma_1,\ldots, \gamma_t\rangle, \  \  \text{ and that } \ \ \langle y_1,\ldots, y_r\rangle +\langle \gamma_1,\ldots, \gamma_t\rangle \subset {\mathfrak{p}'}. $$
To show that the last inclusion is an equality
we can argue as in Step 1, to check  that
$$A'/\left(\langle y_1,\ldots, y_r\rangle +\langle \gamma_1,\ldots, \gamma_t\rangle\right)$$ is a $(d-r)$-dimensional regular local ring.
Thus   $\{y_1,\ldots, y_r,\gamma_1,\ldots, \gamma_t\}$ form a minimal set of generators for
${\mathfrak{p}'}A'_{\mathfrak{p}'}\subset A'_{\mathfrak{p}'}$, 
hence $\gamma_1,\ldots, \gamma_t\in  {\mathfrak p}$ form a
 $\lambda_{\mathfrak{p}'A'_{\mathfrak{p}'}}$-sequence and therefore a
  $\lambda_{\mathfrak{p}A_{\mathfrak{p}}}$-sequence.
\end{proof}

\section{Rees algebras and their use in resolution}\label{Rees_Algebras} 

The stratum  defined by the  maximum value of the multiplicity function of a variety can be described using equations and weights (\cite{V}); and 
the same occurs with the Hilbert-Samuel function (\cite{Hir1}).
As we will see, Rees algebras happen to be a  a suitable tool to work in this setting,
opening the possibility to using different algebraic techniques. 
We refer to \cite{V3} and \cite{E_V} for further details.

\begin{definition}
Let $A$ be a Noetherian ring. A \textit{Rees algebra $\mathcal{G}$ over $A$} is a finitely generated graded $A$-algebra, $\mathcal{G}=\bigoplus _{l\in \mathbb{N}}I_{l}W^l\subset A[W]$, 
for some ideals $I_l\in A$, $l\in \mathbb{N}$ such that $I_0=A$ and $I_lI_j\subset I_{l+j}\mbox{,\; }$ for all $ l,j\in \mathbb{N}$. Here, $W$ is just a variable to keep track of the degree of the ideals $I_l$. Since $\mathcal{G}$ is finitely generated, there exist some $f_1,\ldots ,f_r\in A$ and positive integers (weights) $n_1,\ldots ,n_r\in \mathbb{N}$ such that
$\mathcal{G}=A[f_1W^{n_1},\ldots ,f_rW^{n_r}]$. The previous definition extends to Noetherian schemes in the obvious manner.
\end{definition}

In the following lines, we assume that $\mathcal{G}=\oplus _{l\geq 0}I_lW^l$  is a Rees algebra defined on a scheme $V$ that is smooth  over a perfect field $k$
(whenever the conditions on $V$ are relaxed it will be explicitly indicated).
If we assume $V$ to be affine, then we will write $V=\Spec(R)$. 

The  \textit{singular locus} of $\mathcal{G}$, Sing$(\mathcal{G})$,  is the closed set given by all the points $\zeta \in V$ such that $\nu _{\zeta }(I_l)\geq l$, $\forall l\in \mathbb{N}$,
where $\nu _{\zeta}(I)$ denotes the order of the ideal $I$ in the regular local ring $\mathcal{O}_{V,\zeta }$. If, locally, $\mathcal{G}=R[f_1W^{n_1},\ldots ,f_rW^{n_r}]$, then  $
\Sing(\mathcal{G})=\left\{ \zeta \in \mathrm{Spec}(R)\ |\, \nu _{\zeta }(f_i)\geq n_i,\;  i=1,\ldots ,r\right\} \subset V$ (see \cite[Proposition 1.4]{E_V}).  

\begin{example} \label{Ex:Multiplicity}
Suppose that
$X\subset\Spec(R)=V$ is  a hypersurface with $I(X)=(f)$.  Let $m>1$ be the maximum   multiplicity  at the points of $X$. Then the singular locus of $\mathcal{G}=R[fW^m]$   is the set of points of $X$ having maximum multiplicity $m$. This idea can be generalized as follows. 
Suppose $X$ is a $d$-dimensional variety over a  perfect field, and let $\mm$ be the maximum value of the multiplicity at points of $X$,   $\text{Mult}_X$.
Then as, explained in the Introduction, using the polynomials in (\ref{polinomios}) we
have that  
if $\G:={\mathcal O}_V[f_1W^{m_1}, \ldots, f_eW^{m_e}]$, then $\Sing(\G)= \Mm=\bigcap_{j=1}^r \mathrm{\underline{Max}\; mult}_{\{f_j=0\}}$. 	
The precise statement of this result will be given in Section \ref{seccion_demos}, since it will play a central role in the proof of Theorem \ref{maintheorem}. 
\end{example}

In the previous example, the link between  the closed set of points of {\em worst singularities} of   $X$ and the singular loci of  the corresponding  Rees algebras   is much stronger than just an equality of closed sets  $\Sing(\G)= \Mm$.  In particular,  
by defining  a suitable law of transformations of Rees algebras  after  a blow up, we can establish the same  link  between the  closed set of points of {\em worst singularities} of  the strict transform of $X$,  and the singular locus of the {\em transform} of the corresponding Rees algebra (at least if  the singularities of $X$ have not improved).  
This motivates the following definitions.

\begin{definition}\label{def:transf_law}
	Let $\mathcal{G}$ be a Rees algebra on a smooth scheme $V$. A \textit{$\mathcal{G}$-permissible blow up}, $V\stackrel{\pi }{\longleftarrow} V_1$, 
	is the blow up of $V$ at a smooth closed subset $Y\subset V$ contained in $\mathrm{Sing}(\mathcal{G})$ (a {\em permissible center for $\mathcal{G}$}). We use    $\mathcal{G}_1$ to denote  the (weighted) transform of $\mathcal{G}$ by $\pi $, which is defined as
	$\mathcal{G}_1:=\bigoplus _{l\in \mathbb{N}}I_{l,1}W^l\mbox{,}$
	where $
	I_{l,1}=I_l\mathcal{O}_{V_1}\cdot I(E)^{-l}
$, 
	for $l\in \mathbb{N}$ and $E$ the exceptional divisor of the blow up $V\stackrel{\pi }{\leftarrow} V_1$.
\end{definition}

\begin{definition}\label{def:res_RA}
	Let $\mathcal{G}$ be a Rees algebra over a smooth scheme $V$. A \textit{resolution of $\mathcal{G}$} is a finite sequence of blow ups
	\begin{equation}\label{diag:res_Rees_algebra}
	\xymatrix@R=0pt@C=30pt{
		V=V_0 & V_1 \ar[l]_>>>>>{\pi _1} & \ldots \ar[l]_{\pi _2} & V_L \ar[l]_{\pi _L}\\
		\mathcal{G}=\mathcal{G}_0 & \mathcal{G}_1 \ar[l] & \ldots \ar[l] & \mathcal{G}_L \ar[l]
	}\end{equation}
	at permissible centers $Y_i\subset \text{Sing} ({\mathcal G}_i)$, $i=0,\ldots, L-1$, such that $\mathrm{Sing}(\mathcal{G}_L)=\emptyset$, and such that the exceptional divisor of the composition $V_0\longleftarrow V_L$ is a union of hypersurfaces with normal crossings.
\end{definition}

\begin{remark}\label{ejemplo_hipersuperficie}
The Rees algebras of Example   \ref{Ex:Multiplicity}  are  defined so that   a resolution of the corresponding Rees algebra, $\G$ (\ref{diag:res_Rees_algebra}), induces a sequence of blow ups on $X$,  that ultimately leads to a simplification of the multiplicity of $X$ as in (\ref{SimplificaMult}).
Notice that for these sequences $\Sing(\G_i)=\Max\mult_{X_i}$, for $i=0,1,\ldots,L$.

Resolution of Rees algebras is known to exists when $V$  is a smooth scheme defined over a field of characteristic zero  (\cite{Hir}, \cite{Hir1}). In \cite{V1} and \cite{B-M} different algorithms of resolution of Rees algebras are presented (see also \cite{E_V97}, \cite{E_Hau}). More details will be given in the next section. 
 \end{remark}

\begin{parrafo} \label{RepreMult}
{\bf On the representation of the multiplicity by Rees algebras.}
In addition to permissible blow ups, there are other morphisms that play a role in resolution. These are involved in the arguments of  {\em Hironaka's trick}, and they are used to justify that the   {\em resolution invariants}   are {\em well defined} (\cite[\S 21]{Br_E_V}).
Some of these invariants will be treated in the following sections.
Apart from permissible blow ups, 
these morphisms are  multiplications by an affine line or restrictions to open subsets. A concatenation of any of these three kinds of morphisms is what we call a {\em local sequence}.   Therefore, for a given Rees algebra $\G$ defined on a smooth scheme $V$,  a 
	{\em ${\mathcal G}$-local sequence over $V$} is a   sequence of transformations over $V$, 
	\begin{equation}\label{Def:GLocSeq}
	\xymatrix @C=30pt{
		(V=V_0,\mathcal{G}=\mathcal{G}_0) & \ar[l]_-{\pi_{0}}  (V_1,\mathcal{G}_1) & \ar[l]_-{\pi_{1}}  \cdots & \ar[l]_-{\pi_{L-1}} (V_L,\mathcal{G}_L),
	}
	\end{equation}
where each  $\pi_i$ is either a permissible blow up  for ${\mathcal G}_i\subset {\mathcal O}_{V_i}[W]$ (and  ${\mathcal G}_{i+1}$ is the transform of ${\mathcal G}_{i}$ in the sense of  Definition \ref{def:transf_law}), or a multiplication by a line or a restriction to some open subset of $V_i$    (and then ${\mathcal G}_{i+1}$ is   the pull-back of ${\mathcal G}_i$ in $V_{i+1}$).  
If we assume that sequence (\ref{diag:res_Rees_algebra}) is a $\G$-local sequence over $V$ (instead of just a sequence of permissible blow ups), with $\G$ as in Example \ref{ejemplo_hipersuperficie}, then the equality 
	$\text{\underline{Max} mult}_{X_i} =\Sing(\G_i)$ still holds for each $i=1,\ldots,L-1$. Because of this fact we say that the {\em pair $(V, \G)$ represents the closed set $\Mm$}, since there is such a strong link between the two closed sets $\Sing (\G_i)$ and $\text{\underline{Max} mult}_{X_i}$ along the sequence.
The same  can be said about the representation of the Hilbert-Samuel function in \cite{Hir1}. See \cite{Br_V2} for precise definitions and results on local presentations.
\end{parrafo}

\begin{parrafo} \label{integral_diff} {\bf Uniqueness of the representations of the multiplicity.} The Rees algebra of Example \ref{Ex:Multiplicity} is not the unique representing $\Mm$.
To see this, we consider two operations: 
	
{\bf (i) Rees algebras and integral closure.} Two Rees algebras over a (not necessarily regular) Noetherian ring $R$  are \textit{integrally equivalent} if their integral closure in $\mathrm{Quot}(R)[W]$ coincide.
We use $\overline{\mathcal{G}}$ for the integral closure of $\mathcal{G}$, which can be shown to also be a Rees algebra over $R$ (\cite[\S 1.1]{Br_G-E_V}).  It is worth noticing that for a given  Rees algebra $\G=\oplus_lI_lW^l$ there is always some integer $N$ such that $\G$ is finite over $R[I_NW^N]$    (see \cite[Remark 1.3]{E_V}).   
	
{\bf (ii) Rees algebras and  saturation by differential operators.}   Let $\beta: V\to V'$ be a smooth morphism of smooth schemes defined over a perfect field $k$ with $\dim V > \dim V'$. Then, for any integer $s$, the sheaf of relative differential operators of order at most $s$,   $\Diff_{V/V'}^s$,
is locally free over $V$  (\cite[(4)~\S~16.11]{EGA_4}). We will say that a sheaf of ${\mathcal O}_V$-Rees algebras ${\mathcal G}=\oplus_l I_lW^l$ is a
{\em $\beta$-differential Rees algebra} if there is an affine covering $\{U_i\}$ of $V$, such that for  every homogeneous element $fW^N\in \G$ and every $\Delta\in \Diff_{V/V'}^s(U_i)$ with $s<N$,  we have that $\Delta(f)W^{N-s}\in {\mathcal G}$ (in particular, $I_{i+1}\subset I_i$  since  $ \Diff_{V/V'}^0\subset \Diff_{V/V'}^1$).
Given an arbitrary Rees algebra $\G$ over $V$  there is a natural way to construct a $\beta$-relative differential algebra with the property of being the smallest containing $\G$,  and we will denote it by $\Diff_{V/V'}(\G)$   (see \cite[Theorem 2.7]{V07}). Relative differential Rees algebras will play a role in the definition  of the so called {\em elimination algebras}, see Section \ref{ElimAlg}.  
	
We say that $\mathcal{G}$ is \textit{differentially closed} if it is closed by the action of the sheaf of (absolute) differential operators  $\Diff_{V/k}$. We use  $\mathrm{Diff}(\mathcal{G})$ to denote the smallest differential Rees algebra containing $\mathcal{G}$ (its \textit{differential closure}). See \cite[Theorem 3.4]{V07} for the existence and construction. 
	
It can be shown that
$\Sing(\G) =\Sing (\overline{\G})= \Sing (\mathrm{Diff}(\mathcal{G})),$
(see \cite[Proposition 4.4 (1), (3)]{V3}). In addition, it can be checked that if $\G$ represents $\Mm$ as in Example \ref{Ex:Multiplicity}, then
the integral closure of $\mathrm{Diff}(\mathcal{G})$  is the largest algebra in $V$ with this property. 
The previous discussion motivates the following definition: two Rees algebras on $V$, $\G$ and $\H$, are said to be {\em weakly equivalent}  if: (i) they share the same singular locus;  (ii) any $\G$-local sequence is an $\H$-local sequence, and vice versa, and they share the same singular locus after any $\G$-(respectively $\H$-)local sequence. It can be proven that two Rees algebras $\G$ and $\H$ are weakly equivalent if and only if $\overline{\mathrm{Diff}(\mathcal{G})}=\overline{\mathrm{Diff}(\mathcal{H})}$ (see \cite{Br_G-E_V} and \cite{HirThree}),
and, in particular, a resolution of one of them induces a resolution of the other and vice versa. 
\end{parrafo} 

\section{Algorithmic resolution and resolution invariants} \label{AlgoResol}

In characteristic zero, an algorithmic  resolution of Rees algebras  requires the definition of  {\em resolution invariants}.
These  are used to assign a string of numbers to each point $\zeta \in \Mm= \Sing (\G)$. In this way  one  can  define  an upper semi-continuous function $g:\Sing (\G) \to (\Gamma, \geq)$, where $\Gamma$ is some well ordered set, and whose maximum value determines the first center to blow up. This function is constructed so that its maximum value drops after each blow up.  As a consequence, a resolution of $\G$ is achieved after a finite number of steps.

The most important resolution invariant   is {\em Hironaka's order function at a point  $\zeta \in \mathrm{Sing}(\mathcal{G})$}  which we also refer  as   the \textit{order of the Rees algebra $\mathcal{G}$ at $\zeta$}, and it is defined as 
$\mathrm{ord}_{\zeta}(\mathcal{G}):=\inf _{l\geq 0}
\left\{\nu_{\zeta}(I_l)/l\right\}$. 
If $\mathcal{G}=R[f_1W^{m_1},\ldots ,f_rW^{m_r}]$ and $\zeta\in \mathrm{Sing}(\mathcal{G})$ then by   \cite[Proposition 6.4.1]{E_V}),  
$\mathrm{ord}_{\zeta}(\mathcal{G})=\min _{i=1,\ldots,r}\left\{\nu_{\zeta}(f_i)/{m_i}\right\}$. 
Any other invariant involved in the algorithmic resolution of a Rees algebra $\G$ derives from Hironaka's order function. Finally, it can be proved that for any point $\zeta\in\Sing(\mathcal{G})$ we have
$\ord_{\zeta}(\mathcal{G})=\ord_{\zeta}(\overline{\mathcal{G}})=\ord_{\zeta}(\Diff(\mathcal{G}))$  
(see   \cite[Remark 3.5, Proposition 6.4 (2)]{E_V}).

It can be shown that two Rees algebras that are weakly equivalent share the same resolution invariants and therefore a resolution of one induces a resolution of the other. In particular, this is the case for     $\G$, $\overline{\G}$ and $\Diff (\G)$   (\cite[Proposition 3.4, Theorem 4.1,  Theorem 7.18]{E_V}, \cite{Villamayor2005_2}).

\begin{parrafo} \label{induccion}
{\bf The role of Hironaka's order in resolution and the use of induction in the dimension.}
Suppose $\G$ is defined on a smooth scheme $V$ of dimension $n$, and assume  that $\mathrm{ord}_{\xi }(\mathcal{G})=1$ for some closed point $\xi \in \Sing (\G)$. Then, there are two possibilities: 
	\begin{itemize}
		\item[(i)] Either the point $\xi$ is contained in some codimension-one  component  $Y$  of $\Sing (\G)$; in such case   it can be proven  that $Y$ is smooth, and the blow up at $Y$ induces a resolution of $\G$, locally at $\xi$ (\cite[Lemma 13.2]{Br_V});
		\item[(ii)] Otherwise,  it can be shown that, locally, in an \'etale neighborhood of $\xi$,  there is a smooth projection from $V$ to some smooth $(n-1)$-dimensional scheme $Z$, together with a new Rees algebra ${\mathcal R}$ on $Z$ such that a resolution of ${\mathcal R}$ induces a resolution of $\G$ and vice versa, at least if the characteristic is zero. This is what we call an {\em elimination algebra of $\G$}  and details on its construction will be given in the next section. 
	\end{itemize}

Case (ii) indicates that resolution of Rees algebras can be addressed by induction on the dimension when the characteristic is zero. 

It is worthwhile mentioning that if the maximum order at the points of $\Sing(\G)$ is larger than one,  then one can attach a new Rees algebra $\H$  to the closed points of maximum order,  $\underline{\text{Max}} \ord (\G)$, so that $\Sing (\H)=\underline{\text{Max}} \ord (\G)$, and so that the equality is preserved by $\H$-local sequences. Thus $\H$ is unique up to weak equivalence.  This new Rees algebra $\H$ is constructed so that  its maximum order equal to one,  and the arguments in (i) and (ii) can be applied to it.  
\end{parrafo}

\section{Elimination algebras}\label{ElimAlg}

Along this and the following sections, $V^{(n)}$ denotes an $n$-dimensional smooth scheme over a perfect field $k$, and $\G^{(n)}=\oplus_l I_l W^l$ a Rees algebra over $V^{(n)}$. Our purpose is to   search for smooth morphisms  from $V^{(n)}$ to some $(n-e)$-dimensional smooth scheme, for some $e\geq 1$, so that $\Sing (\Gn)$ is homeomorphic to its image via $\beta$, and so that this condition is preserved by permissible blow ups in some sense that will be specified below. One way to find such smooth morphisms  is by considering morphisms from $V^{(n)}$ which are somehow {\em transversal}  to  $\Gn$. Transversality is expressed in terms of the {\em tangent cone of $\Gn$} at a given point of its singular locus (see Definition \ref{def83} below).   

\medskip

Let    $\xi \in \Sing (\G^{(n)})$ be a closed point, and let $\text{Gr}_{{\mathfrak m}_{\xi}}({\mathcal O}_{V^{(n)},\xi}) \cong k'[Y_1,\ldots, Y_n]$ be  the graded ring of $\mathcal{O}_{V^{(n)},\xi}$, where $k'$ is  the residue field at $\xi$.  Observe  that $\text{Spec} (\text{Gr}_{{\mathfrak m}_{\xi}}({\mathcal O}_{V^{(n)},\xi}))={\mathbb T}_{{V^{(n)}},\xi}$, the tangent space of $V^{(n)}$ at $\xi$. 

\begin{definition}\label{tangent_cone}  {\rm Suppose $\xi\in \Sing(\Gn)$ is a closed point with $\ord_{\xi}(\Gn)=1$. The {\em initial ideal} or  {\em tangent ideal of $\G^{(n)}$ at $\xi$}, $\text{In}_{\xi}\G^{(n)}$, is   the homogeneous ideal of $\text{Gr}_{{\mathfrak m}_{\xi}}({\mathcal O}_{V^{(n)},\xi})$ generated by
		$
		\text{In}_{\xi}(I_l) :=  (I_l+{\mathfrak m}^{l+1}_{\xi})/{\mathfrak m}^{l+1}_{\xi}, 
		$
		for all $l \geq 1$. The {\em tangent cone of    $\Gn$ at $\xi$}, ${\mathcal C}_{\Gn,\xi}$,  is the closed subset of ${\mathbb T}_{{V^{(n)}},\xi}$ defined by the initial ideal   of    $\Gn$ at $\xi$. }
\end{definition}

\begin{definition}\label{tau}  \cite[4.2]{V07} 
The {\em $\tau$-invariant of ${\Gn}$ at the closed point ${\xi}$} is the minimum number of variables  in $\text{Gr}_{{\mathfrak m}_{\xi}}({\mathcal O}_{V^{(n)},\xi})$   needed to generate 
	$\operatorname{In}_{{\xi}}({\Gn})$. This in turn is the codimension of the largest linear
	subspace  ${\mathcal L}_{{\Gn},{\xi}}\subset {\mathcal
		C}_{{\Gn},{\xi}}$ such that $u+v\in {\mathcal C}_{{\Gn},{\xi}}$ for all $u\in {\mathcal C}_{{\Gn},{\xi}}$ and $v\in
	{\mathcal L}_{{\Gn},{\xi}}$. The $\tau$-invariant of ${\Gn}$ at ${\xi}$ is denoted by  $\tau_{{\Gn},{\xi}}$. 
\end{definition} 

\begin{definition}\label{transversalidad}
Let $\xi\in\Sing({\Gn})$ be a  closed point with $\tau_{{\Gn},{\xi}}\geq e\geq 1$. A local smooth projection to a  $(n-e)$-dimensional (smooth) scheme  $V^{(n-e)}$,   $\beta : V^{(n)}   \to  V^{(n-e)}$, is {\em  $\Gn$-transversal at $\xi$} if
$\ker (d_{\xi}\beta)\cap   {\mathcal C}_{{\Gn},\xi}=\{0\}\subset {\mathbb T}_{{V^{(n)}},\xi}$,  
where $d_{\xi}\beta$ denotes the differential of $\beta$ at the point $\xi$. 
\end{definition}

\begin{definition}\label{def83}
Let $\xi\in\mbox{Sing }{\Gn}$ be a  closed point with $\tau_{{\Gn},{\xi}}\geq e\geq 1$. A local smooth projection to an  $(n-e)$-dimensional (smooth) scheme  $V^{(n-e)}$,   $\beta : V^{(n)}   \to  V^{(n-e)}$, 
is {\em  $\Gn$-admissible locally at $\xi$} if the following conditions hold:
\begin{enumerate}
\item The  point $\xi$ is not contained in any codimension-$e$-component  of    $\mbox{Sing }{\Gn}$;
\item The Rees algebra ${\Gn}$ is a $\beta$-relative differential algebra (see \S  \ref{integral_diff} (ii));  
\item The morphism $\beta$  is $\Gn$-transversal at $\xi$. 
\end{enumerate}
\end{definition}
Regarding condition (1), if $\xi$ is contained in a codimension-$e$-component  of  $\mbox{Sing }{\Gn}$ then this component is a permissible center, see \S\ref{induccion}.
Under the previous conditions, it is always possible to construct a $\Gn$-admissible morphism in an (\'etale) neighborhood of $\xi$ (see \cite{V07} and also \cite[\S 8.3]{Br_V}).

\begin{definition} \label{eliminacion}
\cite{V07,Br_V}
Let $\beta  : V^{(n)}    \to   V^{(n-e)}$ be a $\Gn$-admissible projection in an (\'etale)  neighborhood of the closed point $\xi$. Then the ${\Ovne}$-Rees algebra 
 $\Gne:=\Gn\cap {\Ovne}[W]$, 
and any other with the same integral closure in ${\Ovne}[W]$,  is an {\em elimination algebra of $\Gn$ in $V^{(n-e)}$} (see \cite[Theorem 4.11]{V07}). 
\end{definition}

\begin{example} \label{Emininacion_Hiper}
	Let $S$ be a  smooth  $d$-dimensional $k$-algebra of finite type, with $d>0$.  Let $V^{(d+1)}=\mathrm{Spec}(S[x])$. Then the natural inclusion $S\stackrel{\beta ^*}{\longrightarrow } S[x]$, induces a   smooth projection
$V^{(d+1)}\stackrel{\beta}{\longrightarrow }V^{(d)}=\mathrm{Spec}(S)$.  Let $f(x)\in S[x]$ be a polynomial of degree $m>1$, defining a hypersurface $X$  in $V^{(n)}$. Set $X=\mathrm{Spec}(S[x])/\langle f(x) \rangle$. Suppose that  $\xi \in X$ is  a point  of multiplicity $m$.  
Then,
$$\mathcal{G}^{(d+1)}=\mathrm{Diff}(S[x][fW^{m}])\subset S[x][W]$$ represents the multiplicity function on $X$ locally at $\xi$. If the characteristic is zero and if  we assume that  $f$ has the form of Tschirnhausen (there is always a change of coordinates that leads us to this form):
\begin{equation}\label{ec:f_Tsch}
f(x)=x^m+a_{2}x^{m-2}+\ldots + a_{m-i}x^i+\ldots +a_m\in S[x]\mbox{,}
\end{equation}
where $a_i\in S$ for $i=0,\ldots ,m-2$,  then it can be shown that, up to integral closure,  
$$\mathcal{G}^{(d)}=\mathrm{Diff}(S[x][a_{2}W^2,\ldots ,a_{m-i}W^{m-i},\ldots ,a_mW^m])\mbox{,}$$
is an elimination algebra of $\G^{(d+1)}$. If the characteristic is positive, the elimination algebra is also defined. In either case,  it can be shown that it is generated by a finite set of some symmetric (weighted homogeneous) functions evaluated   on the  coefficients of $f(x)$ (cf.  \cite{V00}, \cite[\S1, Definition 4.10]{V07}).  It is worthwhile noticing that  the elimination algebra $\Gd$ is invariant under changes of the form $x'=x+\alpha$ with $\alpha\in S$ \cite[\S 1.5]{V07}. Finally, we will see that, to understand elimination algebras in a more general setting, it suffices to treat the hypersurface case, at least for the purposes of this paper (see \S \ref{p_presentations_suitable}, specially (\ref{amalgama}) and (\ref{amalgama_eliminacion_1})). 
\end{example}

\begin{parrafo}\label{EliminationProperties} {\bf Properties of elimination algebras.} {\rm 
		Let $\beta : V^{(n)}    \to   V^{(n-e)}$ be a $\Gn$-admissible projection in an (\'etale)  neighborhood of a closed $\xi\in \Sing (\Gn)$, and let $\Gne\subset \Ovne[W]$ be an elimination algebra.  Then $\Sing({\mathcal G}^{(n)})$ maps injectively into $\Sing({\mathcal G}^{(n-e)})$, in particular 
		$\beta  (\mbox{Sing}({\mathcal G}^{(n)}))\subset \mbox{Sing}({\mathcal G}^{(n-e)})$  with  equality if the characteristic is zero, or if ${\mathcal G}^{(n)}$ is a differential Rees algebra  (see \cite[\S 8.4]{Br_V}).   Moreover, If ${\mathcal G}^{(n)}$ is a differential Rees algebra, then so is    ${\mathcal G}^{(n-e)}$ (see \cite[Corollary 4.14]{V07}). And if ${\mathcal G}^{(n)}\subset {\mathcal G}'^{(n)}$ is a finite extension, then ${\mathcal G}^{(n-e)}\subset {\mathcal G}'^{(n-e)}$ is a finite extension (see \cite[Theorem 4.11]{V07}). Finally, for a point $\zeta\in\Sing(\Gn)$, 
		the order of $\Gne$ at  $\beta (\zeta)$ does not depend on the choice of the  projection $\beta$  (see \cite[Theorem 5.5]{V07} and \cite[Theorem 10.1]{Br_V}).
	}
\end{parrafo}

\begin{parrafo} \label{DefOrdHironaka}
\textbf{Hironaka's order of an algebraic variety.} 
 Let $X$ be an equidimensional variety of dimension $d$ over a perfect field $k$ and let $\zeta\in X$ be a point of maximum  multiplicity $m>1$. We can assume that $X=\Spec(B)$ is affine. 
Let $\xi\in\overline{\{\zeta\}}$ be a closed of multiplicity $m$. 
 Then, as indicated in Example \ref{Ex:Multiplicity}, there is an 
\'etale neighborhood of $\Spec(B)$, $X'=\Spec(B')$,   an embedding in some  smooth $(d+e)$-dimensional scheme $V^{(d+e)}$,
and a differential Rees algebra $\G^{(d+e)}$ representing the top multiplicity locus of $X'$.  In \S\ref{p_presentations_suitable} we will see that under these assumptions,  $\tau_{\G,\xi'} \geq e$, and   there is a $\G^{(d+e)}$-admissible projection to some $d$-dimensional smooth  scheme where an elimination algebra $\Gd$ can be defined. Let $\zeta'\in X'$ be a point mapping to $\zeta$.  Then by \S\ref{EliminationProperties}, 
$$\ord^{(d)}_X(\zeta):=\ord^{(d)}_{\G^{(d+e)}}({\zeta'}).$$
does not depend on the selection of the \'etale neighborhood, nor on the choice of  Rees algebra representing the top multiplicity locus,  nor on   the  admissible projection. We refer to  this rational number  as  {\em Hironaka's order function of $X$ at $\zeta$ in dimension $d$}. 
\end{parrafo}

\section{The function $\Hord$} \label{presentaciones}

When facing an algorithmic resolution of the  variety $X$ in characteristic zero, the number $\ord^{(d)}_{X}({\zeta})$ is  the most important invariant at the  point $\zeta$  (after  the multiplicity), and there is a strong link between the resolutions of $\G^{(d+e)}$ and $\G^{(d)}$:
in particular, a resolution of the first induces a resolution of the second and vice versa.
When the characteristic is positive, this link between ${\mathcal G}^{(d+e)}$ and ${\mathcal G}^{(d)}$ is weaker, as illustrated in the following example.

\begin{example} \label{EjemCuspide}
Let $X=\Spec\left({\mathbb F}_{2}[z,y]/\langle z^2-y^3\rangle \right)$.
Set  $V^{(2)}=  \Spec\left({\mathbb F}_{2}[z,y]\right)$,   define the  ${\mathbb F}_{2}[z,y]$-Rees algebra
$\G^{(2)}:=\Diff\left({\mathbb F}_{2}[z,y][(z^2-y^3)W^2] \right)= {\mathbb F}_{2}[z,y][y^2W, (z^2-y^3)W^2]$,
and let $\xi$ be the singular point of $X$. 
The inclusion ${\mathbb F}_{2}[y] \subset {\mathbb F}_{2}[z,y]$ induces a $\G^{(2)}$-transversal projection $\beta:V^{(2)}\to V^{(1)}=\Spec({\mathbb F}_{2}[y])$.
The elimination algebra is 
$\G^{(1)}={\mathbb F}_{2}[y][y^2W]$, and $\beta(\Sing(\G^{(2)})=\Sing(\G^{(1)})$.
However, after the blow up at $\xi$, $\Sing(\G_1^{(2)}=\emptyset$ but 
$\Sing(\G^{(1)}_1)\neq\emptyset$.
\end{example}

Thus, when the characteristic is positive, what we consider the first  relevant invariant in characteristic zero,    $\ord^{(d)}_{\G^{(d+e)}}(\zeta)=\ord_{\beta(\zeta)}\Gd$, needs to be refined. 
This leads us to talk about the function $\Hord^{(d)}_{X}$, introduced and studied in \cite{BVComp} and \cite{BVIndiana}. 
We will start with the definition for hypersurfaces, and then we will see that the general case reduces to that of hypersurfaces.

\begin{parrafo}  \label{Pres_Hiper}  
\textbf{The hypersurface setting.}
Let $V^{(d+1)}$ be $(d+1)$-dimensional   smooth scheme over a perfect field $k$, 
let $X\subset V^{(d+1)}$ be a hypersurface of dimension $d$,
and let $\xi\in X$ be a closed point of maximum multiplicity $m>1$.
Choose a local generator $f\in\mathcal{O}_{V^{(d+1)},\xi}$ defining $X$ in an open affine  neighborhood   $U\subset V^{(d+1)}$ of $\xi$, which we denote by $V^{(d+1)}$ for simplicity. 
Define the Rees algebra $\mathcal{G}^{(d+1)}=\Diff(\mathcal{O}_{V^{(d+1)}}[fW^m])$,
see Example \ref{Ex:Multiplicity}.
After applying Weierstrass Preparation Theorem, we can assume that in an \'etale neighborhood of
$\xi\in V^{(d+1)}$,
which we again denote by $V^{(d+1)}$, we have the following situation.
There is an affine smooth scheme of dimension $d$, $V^{(d)}=\Spec(S)$, such that
$V^{(d+1)}=\Spec(S[z])$, where is $z$ is a variable, and 
$X$ is defined by
\begin{equation} \label{Weierstrass}
f=z^m+a_1z^{m-1}+\cdots+a_{m-1}z+a_m, \qquad a_i\in S, \  i=1,2,\ldots,m.
\end{equation}
It can be checked that the morphism $\beta:V^{(d+1)}\to V^{(d)}$ is $\mathcal{G}^{(d+1)}$-transversal at $\xi$ (Definition \ref{transversalidad}).
We say that $f$ is written in \emph{Weierstrass form with respect to the projection $\beta$}.
\end{parrafo}

\begin{remark} \label{SimpPresHiper}
\cite[\S 2.15]{BVIndiana}
With the same notation as in \S \ref{Pres_Hiper},
it can be proved  that, in a neighborhood of $\xi$, $\G^{(d+1)}$ has the same integral closure as
\begin{equation} \label{EqPresHiper}
S[z][fW^m, \Delta_z^{\alpha}(f)W^{m-\alpha}]_{1\leq \alpha \leq m-1}\odot \Gd,
\end{equation}
where $\Gd$ is an elimination algebra of $\mathcal{G}^{(d+1)}$, the $\Delta^{i}_{z} $ are the Taylor differential operators, and we use "$\odot$" to denote the smallest Rees algebra containing the
two that are involved in the expression.
Recall that $\{\Delta^0_z,\ldots, \Delta^r_z\}$ is a basis of the free module of $S$-differential operators of $S[z]$ of order $r$ (see \cite[Proposition 2.12]{BVComp}; see also Example \ref{Emininacion_Hiper}).
We will say that (\ref{EqPresHiper}) is a \emph{simplified presentation} of $\mathcal{G}^{(d+1)}$ at $\xi$.
The presentation depends on the choice of the smooth morphism $\beta$, the variable $z$ and the monic generator $fW^m$. We will use $\mathcal{P}(\beta,z,fW^m)$ to denote this simplified presentation.
\end{remark}

\begin{definition} \cite[\S 5.5]{BVIndiana} \label{DefSlope1}
Let $\mathcal{P}(\beta,z,fW^m)$ be a  simplified presentation of $\mathcal{G}^{(d+1)}$ as in Remark \ref{SimpPresHiper}, and $f$ as in (\ref{Weierstrass}).
The \emph{slope} of $\mathcal{P}(\beta,z,fW^m)$ at a point $\zeta\in\Sing(\mathcal{G}^{(d+1)})\subset V^{(d+1)}$ is defined as: 
\begin{equation} \label{EqDefSlope1}
Sl(\mathcal{P})(\zeta):=\min\left\{
\nu_{\beta(\zeta)}(a_1),\ldots,\frac{\nu_{\beta(\zeta)}(a_j)}{j},\ldots,\frac{\nu_{\beta(\zeta)}(a_m)}{m}, \ \ord_{\beta(\zeta)}(\mathcal{G}^{(d)})
\right\}.
\end{equation}
\end{definition}

\begin{remark}
The value $Sl(\mathcal{P})(\zeta)$ depends on the chosen data, that is, on the morphism $\beta$, the generator $fW^m$ and the global section $z$.
Translations of the form $z+s$, with $s\in \mathcal{O}_{V^{(d)}}$, give new simplified presentations
$\mathcal{P}(\beta,z+s,fW^m)$ which
may lead to different values of the slope.
The value	
\begin{equation} \label{EqBetaOrd}
\sup_{z'}\left\{Sl(\mathcal{P}(\beta,z',fW^m))(\zeta) \right\}
\end{equation}
does not depend on the choice of the transversal morphism $\beta$,
nor on the choice of the order-one-element $fW^m\in\mathcal{G}^{(d+1)}$ ($fW^m$ can be replaced  by any other order-one-element $gW^{m_1}\in\mathcal{G}^{(d+1)}$ non necessarily defining the hypersurface $X$).
Moreover, the supremum in (\ref{EqBetaOrd}) is a maximum for a suitable selection of $z'$.
See \cite[\S 5.2 and Theorem 7.2]{BVComp}.

\end{remark}

\begin{definition}\label{DefHordHyper}
\cite[\S 5, Definition 5.12]{BVIndiana}
Let $\zeta\in X$ be a point of a hypersurface $X$ of multiplicity $m>1$,
and consider an \'etale neighborhood
$X'\to X$ of a closed point of multiplicity $m$, $\xi\in\overline{\{\zeta\}}$,
such that the setting of \S\ref{Pres_Hiper} holds,
and let $\zeta'\in X'$ be a point mapping to $\zeta$.
Then we define
\begin{equation*}
\Hord^{(d)}_X(\zeta):=\Hord^{(d)}_{X'}(\zeta'):=\max_{z'}\left\{Sl(\mathcal{P}(\beta,z',gW^{N}))(\zeta') \right\}.
\end{equation*} 
\end{definition}

\begin{remark}
When the characteristic of the base field $k$ is zero, then it can be shown that  for all $\zeta\in \Sing (\G^{(d+1)})$, 
$\Hord^{(d)}_X(\zeta)=\ord_{\beta(\zeta)}(\mathcal{G}^{(d)})$
(see \cite[\S 2.13]{BVIndiana} and Example \ref{Emininacion_Hiper}).
Thus, this invariant provides new information only when the characteristic of $k$ is positive.
For example, if $X$ is as in Example \ref{EjemCuspide}, it can be checked that  
$\Hord^{(d)}_X(\xi)=3/2< \ord_{\beta(\xi)}(\mathcal{G}^{(1)})=2$.   
\end{remark}

\begin{parrafo} \textbf{$\mathbf{p}$-Presentations.} \label{Def_p_PresHiper} 
 Suppose $\text{char}(k)=p>0$. 
Continuing with the notation introduced in \S \ref{Pres_Hiper}, since $\mathcal{G}^{(d+1)}$ is a differential algebra, in order to compute the value
$\beta\text{-}\ord(\xi)$, 
it is always possible to find an order-one-element of the form $hW^{p^{\ell}}\in\mathcal{G}^{(d+1)}$, where
$h$ is a monic polynomial of degree $p^{\ell}$ for some $\ell\in\mathbb{Z}_{\geq 1}$, 
and in Weierstrass form with respect to $\beta$.
This can be done as follows. Assume that $g(z)W^N\in\mathcal{G}^{(d+1)}$ and that
\begin{equation*}
g(z)=z^{N}+b_1z^{N-1}+\cdots+b_{N-1}z + b_{N}, \quad b_i\in S, \ i=1,\ldots,N.
\end{equation*}
Write  $N=N'p^{\ell}$ with  $p$ not dividing   $N'$. Set $r=(N'-1)p^{\ell}$ and  $h(z)=\dfrac{1}{N'}\Delta_{z}^r(g(z))$. Note that
\begin{equation}
h(z)=z^{p^{\ell}}+\tilde{b}_1z^{p^{\ell}-1}+\cdots + \tilde{b}_{p^{\ell}}
\end{equation}
where, for $j=1,\ldots,p^{\ell}-1$, $\tilde{b}_j=\dfrac{c_j}{N'}b_j$ for some integer $c_j$, 
and $\tilde{b}_{p^{\ell}}=\dfrac{1}{N'}b_{p^{\ell}}$.
Then  $h(z)W^{p^{\ell}}\in\mathcal{G}^{(d+1)}$ and
$\mathcal{P}(\beta,z,h(z)W^{p^{\ell}})$ is a special type of simplified presentation of $\mathcal{G}^{(d+1)}$.
Presentations of the form $\mathcal{P}(\beta,z,hW^{p^{\ell}})$ will be called \emph{$p$-presentations}
(\cite[Definition 2.14]{BVComp}).
Compared to general simplified presentations, $p$-presentations have the advantage that the computation of the slope (\ref{EqDefSlope1}) becomes simpler.
\end{parrafo}

\begin{theorem} \label{ThBenVilla}
\cite[Theorem 4.4]{BVComp}
Let $\mathcal{P}(\beta,z,hW^{p^{\ell}})$ be a $p$-presentation of $\mathcal{G}^{(d+1)}$, where
\begin{equation}\label{EqThBenVilla}
h(z)=z^{p^{\ell}}+\tilde{b}_1z^{{p^{\ell}}-1}+\cdots+\tilde{b}_{{p^{\ell}}-1}z+\tilde{b}_{p^{\ell}} \in\mathcal{O}_{V^{(d)}}[z].
\end{equation}
Let $\zeta\in\Sing(\mathcal{G}^{(d+1)})$.
Then
\begin{equation*}
Sl(\mathcal{P})(\zeta)=\min\left\{
\frac{\nu_{\beta(\zeta)}(\tilde{b}_{p^{\ell}})}{p^{\ell}}, \ \ord_{\beta(\zeta)}(\mathcal{G}^{(d)})
\right\}.
\end{equation*}
\end{theorem}

\begin{remark}\label{remark_ordenes_intermedios} 
Using the arguments as in the proof of \cite[Theorem 4.4]{BVComp},
it follows that
	\begin{equation}
	\label{ordenes_intermedios}
	\frac{\nu_{\beta(\zeta)}(\tilde{b}_{j})}{j}\geq \ord_{\beta(\zeta)}(\mathcal{G}^{(d})),
	\end{equation}
whenever $1\leq j\leq p^{\ell}-1$. 	
\end{remark}

\begin{parrafo}\label{cleaning}{\bf Cleaning process} \cite[\S 5.1, \S 5.2, and  Proposition 5.3]{BVComp} 
Here we sketch the main ideas to find a $p$-presentation that maximizes $Sl(\mathcal{P})(\zeta)$, since we will be using them in Section \ref{seccion_demos}.
For a given $p$-presentation, and a point $\zeta\in \Sing(\G^{d+1})$, there are different possibilities: 

(A) $Sl(\mathcal{P})(\zeta)=\ord_{\beta(\zeta)}(\mathcal{G}^{(d)})$; 

\noindent (B) $Sl(\mathcal{P})(\zeta)=\frac{\nu_{\beta(\zeta)}(\tilde{b}_{p^{\ell}})}{p^{\ell}}<\ord_{\beta(\zeta)}(\mathcal{G}^{(d)})$, and then:

(B1) $\frac{\nu_{\beta(\zeta)}(\tilde{b}_{p^{\ell}})}{p^{\ell}}\notin {\mathbb Z}_{>0}$; 

(B2) $\frac{\nu_{\beta(\zeta)}(\tilde{b}_{p^{\ell}})}{p^{\ell}}\in {\mathbb Z}_{>0}$ and the initial part of $\tilde{b}_{p^{\ell}}$ at $\zeta$, $\text{In}_{\zeta}(\tilde{b}_{p^{\ell}})\in \text{Gr}_{\beta(\zeta)}({\mathcal O}_{V^{(d)},\zeta})$ is not a $p^{e}$-th power at $\text{Gr}_{\beta(\zeta)}({\mathcal O}_{V^{(d)},\zeta})$; 

(B3) $\frac{\nu_{\beta(\zeta)}(\tilde{b}_{p^{\ell}})}{p^{\ell}}\in {\mathbb Z}_{>0}$ and  $\text{In}_{\zeta}(\tilde{b}_{p^{\ell}})$ is  a $p^{e}$-th power at $\text{Gr}_{\beta(\zeta)}({\mathcal O}_{V^{(d)},\zeta})$.

It can be proven that changes of the form $uz+s$ produce a new $p$-presentation ${\mathcal P}'$ with $Sl(\mathcal{P'})(\zeta)>Sl(\mathcal{P})(\zeta)$ only in case (B3).  In such case, only  changes of the section of the form: 
$z':=z+s$
with $s\in {\mathcal O}_{V^{(d)},\beta(\zeta)}$, and $\nu_{\beta(\eta)}(s)\geq \nu_{\beta(\zeta)}(\tilde{b}_{p^{\ell}})/{p^e}$ lead to new $p$-presentations ${\mathcal P}'$  with  $Sl(\mathcal{P'})(\zeta)\geq Sl(\mathcal{P})(\zeta)$. Moreover, if $\xi\in\overline{\{\zeta\}}$, and $\zeta$ defines a regular closed subscheme at $\xi$, then to maximize the slope it suffices to consider changes of the form $z':=z+s$ with $s\in {\mathcal O}_{V^{(d)},\xi}$, see \cite[proof of Propositions 5.7 and 5.8]{BVComp}. 
\end{parrafo}

\begin{definition} \cite[Definition 5.4]{BVComp} \label{DefHiperNormalForm}
A $p$-presentation $\mathcal{P}(\beta,z,hW^{p^{\ell}})$ with  $h$ as   in (\ref{EqThBenVilla})  is in 
\emph{normal form}\footnote{This is called \emph{well-adapted presentation} in \cite{BVComp}.}
at a point $\zeta\in \Sing(\mathcal{G}^{(d+1)})$, if condition (A), (B1) or (B2) holds in \S \ref{cleaning}.
\end{definition}

Hence to maximize the value $Sl(\mathcal{P})(\zeta)$
for a given $p$-presentation $\mathcal{P}(\beta,z,hW^{p^{\ell}})$,
one can work with presentations in \emph{normal form}.
For simplicity we restrict the notion of normal form to $p$-presentations, but a similar concept can be defined for any presentation, see  \cite[\S 5.7]{BVIndiana}.

\begin{remark}
Given  a hypersurface $X$  and $\mathcal{G}^{(d+1)}$  as in \S \ref{Pres_Hiper}, for a point $\zeta\in\Sing(\mathcal{G}^{(d+1)})$, and a  $p$-presentation $\mathcal{P}(\beta,z,hW^{p^{\ell}})$ in normal form at $\zeta$, 
it can be shown that
\begin{equation} \label{EqHordpNormal}
\Hord^{(d)}_X(\zeta)=Sl(\mathcal{P}(\beta,z,hW^{p^{\ell}}))(\zeta).
\end{equation}
See \cite[Theorem 7.2, Corollary 7.3 and \S 5]{BVComp}.
\end{remark}
\bigskip

\noindent\textbf{The general case}
\medskip

\noindent Given an equidimensional variety $X$ of dimension $d$ over a perfect field $k$, and a singular point $\zeta\in X$, we would like to emulate  the previous statements,  which were valid for a hypersurface.
To this end, we will use the following result, which can be understood as a generalization of Weierstrass preparation theorem.

\begin{theorem}\label{presentacion_previa} 
	\cite[Theorem 6.5]{BVIndiana} Let $\Gn$ be a Rees algebra on a smooth scheme $\Vn$ over $k$ and let $\xi\in \Sing(\Gn)$ be a closed point with $\tau_{\Gn, \xi}\geq e\geq 1$. Then, at a suitable \'etale neighborhood of $\xi$, a $\Gn$-transversal morphism, $\beta: \Vn\to \Vne$, can be defined so that the following conditions hold: 
	\begin{enumerate}
		\item[(i)] There are global functions $z_1, \ldots, z_e$ in ${\mathcal O}_{\Vn}$  such that $\{dz_1,\ldots, dz_e\}$ forms a basis of $\Omega_{\beta}^1$, the module of $\beta$-relative differentials;
		\item[(ii)] There are positive integers $m_1,\ldots, m_e$; 
		\item[(iii)] There are elements $f_{1}W^{m_1},\ldots, f_{e}W^{m_e}\in \Gn$, such that:
		\begin{equation}
		\begin{array}{c}
		f_{1}(z_1)=z_1^{m_1}+a_1^{(1)}z_1^{m_1-1}+\ldots + a_{m_1}^{(1)}\in {\mathcal O}_{\Vne}[z_1],\\
		\vdots \\
		f_{e}(z_e)=z_e^{m_e}+a_1^{(e)}z_1^{m_e-1}+\ldots + a_{m_e}^{(e)}\in {\mathcal O}_{\Vne}[z_e],
		\end{array}
		\end{equation}
		for some global functions $a_i^{(j)}\in {\mathcal O}_{\Vne}$;
		\item[(iv)] The Rees algebra $\Gn$ has the same integral closure as:
\begin{equation}
{\mathcal O}_{\Vn}[f_{i}W^{m_i}, \Delta_{z_i}^{j_i}(f_{i})W^{m_i-j_i}]_{\substack{
1\leq j_i\leq m_i-1\\ i=1,\ldots, e}}\odot \beta^*(\Gne),
\end{equation}
where $\Gne$ is an elimination algebra of $\Gn$ on $\Vne$, and the set $\left\{\Delta_{z_i}^{j_i}\right\}_{\substack{1\leq j_i\leq m_i-1\\ i=1,\ldots,e}}$ consists of the relative differential operators described in by the Taylor operators.
\end{enumerate}
\end{theorem}

\begin{remark} \label{RemRelDiffEtale}
Observe that since $\beta:V^{(n)}\to V^{(n-e)}$ is a smooth morphism of relative dimension~$e$, 
locally, $\mathcal{O}_{V^{(n)}}$ is \'etale over the polynomial ring $\mathcal{O}_{V^{(n-e)}}[z_1,\ldots,z_e]$.
The differential operators $\Delta_{z_i}^{j_i}$ are defined to be the Taylor differential operators. 
\end{remark}

\begin{definition} \label{DefSimpPres}
\cite[Definition 6.6]{BVIndiana}
With the setting and the notation of Theorem \ref{presentacion_previa}, the data,
	\begin{equation}
	{\mathcal P}(\beta, z_1,\ldots, z_e, f_{1}W^{m_1}, \ldots, f_{e}W^{m_e})
	\end{equation}
	that fulfills conditions (i)-(iv) in Theorem \ref{presentacion_previa} is a {\em simplified presentation of $\Gn$}.

Let $X_i$ be the hypersurface defined by $f_i(z_i)\in\calo_{V^{(n-e)}}[z_i]$.
Then we can also define
$$\Hord^{(n-e)}_{\G^{(n)}}:=\min_{i=1\ldots,e}\Hord^{(n-e)}_{X_i}.$$  
\end{definition}

\begin{remark} \label{presentacion_simp_mult} 
Now we go back to  Example \ref{Ex:Multiplicity},  where we consider a representation of the multiplicity of a variety $X\subset V$ at a closed point $\xi\in X$,
given by a Rees Algebra $\G=\calo_{V}[f_1W^{m_1},\ldots,f_eW^{m_e}]$.
We will see in \S \ref{p_presentations_suitable} that $\Diff(\G)$ satisfies conditions
(i)-(iv) in Theorem \ref{presentacion_previa}.
This leads us to define
$$\Hord^{(d)}_{X}(\zeta):=\Hord^{(d)}_{\Diff(\G)}=\min\{\Hord^{(d)}_{X_i}(\zeta)\},$$
where $X_i$ is the hypersurface defined by $f_i$, $i=1,\ldots,e$, and
$\zeta\in\Max\mult_X$. 
\end{remark}

\section{Main results}\label{seccion_demos}

In this section we will  address the proof  of  Theorem\ref{maintheorem}.  For
a given  point $\zeta\in X$ of maximum multiplicity $m>0$, we will want to compute the value $\Hord_{X}^{(d)}(\zeta)$ following the constructions given in    Section \ref{presentaciones}.
To this end, we will use 
Villamayor's presentations of the multiplicity in the \'etale topology, Theorem  
\ref{presentaciones_mult} below.
Finally, since we want to show that $\Hord_{X}^{(d)}(\zeta)$ can  actually be  computed at ${\mathcal O}_{X,\zeta}$,  without the need of \'etale topology, and using the Samuel slope of the local ring,  we will be using our results from Section \ref{Seccion3}.

\begin{theorem}\label{presentaciones_mult} \cite[Lemma 5.2, \S6, Theorem 6.8]{V}  (Presentations for the Multiplicity  function)
 	Let $X=\text{Spec}(B)$ be an affine equidimensional algebraic variety of dimension $d$  defined over a perfect field $k$, and let $\xi\in  \underline{\text{Max}}\text{\,Mult}_X$ be a closed  point of   multiplicity $m>1$. Then,  there is an \'etale neighborhood $B'$ of $B$,   mapping $\xi'\in \Spec(B')$ to $\xi$, so that there is   a smooth $k$-algebra $S$ together with a
 	finite  morphism     $\alpha: \Spec(B')\to \Spec(S)$ of generic rank $m$, i.e., if $K(S)$ is the quotient field of $S$, then $[K(S)\otimes_S B: K(S)]=m$. Write $B'=S[\theta_1,\ldots, \theta_e]$.  Then:  
 		\begin{itemize}
 		\item[(i)] If $f_i(x_i)\in K(S)[x_i]$ denotes the minimum polynomial of $\theta_i$ over $K(S)$ for $i=1,\ldots,e$, then $f_i(x_i)\in S[x_i]$  and there is a commutative diagram: 
 	\begin{equation}
 			\begin{aligned}
 				\label{diagrama_presentacion}
 		\xymatrix{R=S[x_1,\ldots, x_e] \ar[r] & S[x_1,\ldots, x_e]/\langle f_1(x_1),\ldots, f_e(x_e)\rangle \ar[r]   & B'  \\
 			& S \ar[u] \ar[ur]_{\alpha^*} \ar[ul]^{\beta^*} &  }
 			\end{aligned}
 	\end{equation}

 \item[(ii)]  Let $V^{(d+e)}=\Spec(R)$, and let ${\mathcal I}(X')$ be the defining ideal of $X'$ at $V^{(d+e)}$. Then 
 		$$\langle f_1,\ldots, f_e\rangle \subset {\mathcal I}(X');$$ 
 		\item[(iii)] Denoting by $m_i$ the maximum order of the hypersurface $H_i=\{f_i=0\}\subset V^{(d+e)}$,  
 		the differential Rees algebra
 		\begin{equation}
 			\label{G_Representa}
 			\mathcal{G}^{(d+e)}=\Diff(R[f_1(x_1)W^{m_1}, \ldots, f_{e}(x_e)W^{m_{e}}])
 		\end{equation}
 		represents  the top  multiplicity locus of $X$, $\underline{\text{Max}}\text{\,Mult}_X$,  at   $\xi$ in $V^{(d+e)}$.

 	\end{itemize}
 \end{theorem}
     
 \begin{parrafo} \label{setting_prueba} {\bf The setting and the notation  for the proof of Theorem \ref{maintheorem}.}
Let $\xi \in X$ be a closed point of multiplicity $m>1$, and let  $(B,{\mathfrak m}, k(\xi))$ the local ring at the point. Applying Theorem \ref{presentaciones_mult}     there is an \'etale extension $ (B, {\mathfrak m}, k(\xi)) \to (B', {\mathfrak m}', k')$ for which we can find a smooth  $k'$-algebra $S$   and a finite inclusion of generic rank $m$, 
$$S\to B'=S[\theta_1,\ldots, \theta_e].$$
Thus,  statements (i), (ii) and (iii) of  Theorem \ref{presentaciones_mult} hold for $S\subset B'$. In particular,  we have a commutative diagram like  (\ref{diagrama_presentacion}). With this notation, which we fix  for  the rest of the section, we will be simultaneously using $\alpha(\zeta')$ and  $\beta (\zeta')$ to denote  the image in  $\Spec(S)$ of a point  $\zeta'\in \Spec(B')$. We will  choose the first notation if  we want to use   the properties of the finite projection from $\Spec(B')$. The second notation will be  convenient  to emphasize the fact that $\zeta'$ is also a point in the smooth scheme $\Spec(R)$. Sometimes we   will use $V^{(d+e)}$ to refer to $\Spec(R)$. This will help us   recall  the dimension of  the smooth ambient space where $\Spec(B')$ is embedded,  and the space where the Rees algebra $\G^{(d+e)}$ is defined.
And for similar  reasons we occasionally will write   $V^{(d)}$ for  $\Spec(S)$, specially if the elimination algebra $\Gd$ of $\G^{(d+e)}$ is involved (see Section \ref{ElimAlg}).
 
 \medskip
 
Theorem \ref{presentaciones_mult} provides three  pieces of information that will be specially relevant in our arguments: 
\begin{enumerate}
	\item[{\bf (I)}] {\em The existence of the  \'etale neighborhood of $B$, $B'$ together with the  finite extension   $S\subset B'$.} To be able to compare the Samuel slope of $B$ and $B'$ (in the extremal case) we will need to know that $B'$ can be constructed having the same residue field as $B$. This issue is addressed in \S\ref{etale_residual}. 
	\item[{\bf (II)}] {\em The Rees algebra ${\mathcal G}^{(d+e)}$  representing the top multiplicity locus of $X'=\Spec(B')$}. We will see in \S \ref{p_presentations_suitable} below how to use this Rees algebra  to compute the function $\Hord_{X'}^{(d)}$ using the results from Section \ref{presentaciones}.   
	\item[{\bf (III)}] An {\em algebraic presentation of $B'$ as an algebra over $S$, $S[\theta_1, \ldots, \theta_e]$}. We will see in \S \ref{suitable_presentations} below how to find suitable presentations that will help us computing the Samuel slope in the extremal case. 
	
\end{enumerate}

After addressing (I), (II), (III), and after establishing some technical results, we will give the proof of Theorem \ref{maintheorem}.  

\begin{parrafo} \label{etale_residual} {\bf (I) On the \'etale extension of Theorem \ref{presentaciones_mult}.}
\end{parrafo} 		

\noindent We start by stating a giving an idea of the proof of Proposition \ref{suma_directa} below. This result was  sketched in \cite[\S 6.11]{V} and a complete proof  can be found in  \cite[Appendix  A]{Br_V2}. Here we will focus  on    the three main steps of  the argument  that require considering \'etale extensions.   Remark \ref{lema_lema} and Proposition \ref{Caso_tau_e_reformulado_b} below will be relevant to treat the proof of Theorem \ref{maintheorem} in the extremal case. 
\begin{proposition} \label{suma_directa} 
	\cite[\S 6.11]{V}, \cite[Appendix A]{Br_V2}  Let $X$ be an   equidimensional variety defined over a perfect field $k$ and let $\xi\in X$ be a closed point of multiplicity $m>1$. Let $(B,{\mathfrak m}, k(\xi))$ be the local ring at the point.  Then   there is a local   \'etale extension $(B, {\mathfrak m}, k(\xi)) \to (B', {\mathfrak m}', k')$ such that: 
	\begin{enumerate}
		\item[(i)] There is a smooth $k'$-algebra $S$ and a finite morphism $S\to B'$ of generic rank equal to $m$; 
		\item[(ii)] If $\alpha: \Spec(B')\to \Spec(S)$, then    the morphism  $\Gr_{\mathfrak{m}_{\alpha(\xi')}}(S)\to\Gr_{\mathfrak{m}_{\xi'}}(B')$
		is injective, and if,  in addition, $B$ is in the extremal case, then
		$${\mathfrak m}_{\alpha(\xi')}/{\mathfrak m}_{\alpha(\xi')}^2 \oplus \ker(\lambda_{\xi'})={\mathfrak m}_{\xi'}/{\mathfrak m}_{\xi'}^2.$$  
	\end{enumerate} 
\end{proposition}  
\noindent{\em Sketch of the proof. } {\bf Step 1:} If $k(\xi)$ is the residue field at $\xi$, then, after considering the extension
$B_1=\calo_{X,\xi}\otimes_k k(\xi)$
it can be assumed that the point of interest is rational.
Let ${\mathfrak m}_1$ be a maximal ideal of $B_1$ dominating $\mathfrak{m}_{\xi}$.
Then if  $ k_1:=B_1/{\mathfrak m}_1$, we have that $k_1=k(\xi)$.

{\bf Step 2:} After  a finite extension of  the base field $k_1$, $k_2$,  considering  the base change $B_2=B_1\otimes_{k_1}k_2$, there is a maximal ideal ${\mathfrak m}_2\subset B_2$,
dominating ${\mathfrak m}_1$,
such that ${\mathfrak m}_2$ contains a reduction generated by $d$ elements, $\kappa_1,\ldots, \kappa_d$.
To achieve this step, a graded version of Noether's Normalization Lemma is used at the graded ring $\text{Gr}_{{\mathfrak m}_2}(B_2)$.
Letting $k_2=B_2/{\mathfrak m}_2$ we get a
$k_2$-morphism from a polynomial ring in $d$ variables with coefficients in $k_2$ to some localization of $B_2$:
\begin{equation}
	\label{primera_aproximacion}
	\begin{array}{rclr}
		S_2:=k_2[Y_1,\ldots,Y_d] & \longrightarrow & (B_2)_f & \\
		Y_i & \mapsto & \kappa_i  & \text{ for } i=1,\ldots, d.
	\end{array}
\end{equation}

To ease the notation set $B_2:=(B_2)_f$.

{\bf Step 3:}  Finally, after considering an \'etale extension $S_3$ of $S_2$ (inside the henselization of the local ring $(S_2)_{\langle Y_1,\ldots, Y_d\rangle }$; the strict henselization is not needed in this step),
$$\xymatrix@R=20pt@C=30pt{B_2\ar[r]  & B_3:=B_2\otimes_{S_2}S_3 \\
	S_2 \ar[u]\ar[r]  & S_3\ar[u],}$$
it can be assumed that the extension $S_3\to B_3$ is finite of generic rank equal to $m$.
Let $\mathfrak{n}_3\subset S_3$ be the  maximal ideal dominating $\langle Y_1,\ldots, Y_d\rangle$.
Notice that the residue field of $S_3$ at $\mathfrak{n}_3$ is again $k_2$.
There is a maximal ideal $\mathfrak{m}_3\subset B_3$ dominating $\mathfrak{m}_2$ and
if $k_3=B_3/\mathfrak{m}_3$ then $k_3=k_2$.
To conclude, set $B'=(B_3)_{\mathfrak{m}_3}$ and   $S=S_3$. 

Regarding to (ii),  it suffices to observe that   that from the way the finite projection $S\to B'$ is constructed (see  step 2), the morphism $\Gr_{\mathfrak{m}_{\alpha(\xi')}}(S)\to\Gr_{\mathfrak{m}_{\xi'}}(B')$ 
is injective.  Note that the elements $\kappa_1,\ldots, \kappa_d$ are analytically irreducible over $k_2$.  \qed

\begin{remark} \label{lema_lema}  In the proof of Proposition \ref{suma_directa}  we have a sequence of \'etale local extensions:
	\begin{equation*}
		\begin{array}{ccccccc}
			(\calo_{X,\xi}, \mathfrak{m})& \to & ((B_1)_{\mathfrak{m}_1}, \mathfrak{m}_1) & \to &
			((B_2)_{\mathfrak{m}_2},\mathfrak{m}_2) & \to & ((B_3)_{\mathfrak{m}_3}, \mathfrak{m}_3)=(B', {\mathfrak m}'),
		\end{array}
	\end{equation*}
	leading to the  (\'etale) extensions  of graded rings:
	\begin{equation}\label{graduados_etales}
		\text{Gr}_{{\mathfrak m}_{\xi}}(\calo_{X,\xi})=\text{Gr}_{{\mathfrak m_1}}(B_1)\longrightarrow \text{Gr}_{{\mathfrak m_1}}(B_1)\otimes_{k_1} k_2= \text{Gr}_{{\mathfrak m_2}}(B_2)=\text{Gr}_{{\mathfrak m'}}(B'). 
	\end{equation}
Proposition \ref{Caso_tau_e_reformulado_b} below guarantees that the field extension in Step 2 of the proof is not needed if $(B,{\mathfrak m})$ is in the extremal case. Under this assumption all the graded rings in (\ref{graduados_etales})   are isomorphic.
\end{remark}

\begin{proposition} \label{Caso_tau_e_reformulado_b}
	Let $X$ be an equidimensional algebraic variety of dimension $d$ defined over a perfect field $k$,  and  let $\xi\in X$ be a  closed point of multiplicity $m>1$  with local ring $(\mathcal O_{X,\xi}, {\mathfrak m}_{\xi}, k(\xi))$. Assume that the embedding dimension at $\xi$ is $(d+t)$ for some $t\geq 1$.
	If $\xi$ is in the extremal case, 
	then ${\mathfrak m}_{\xi}$ has a reduction ${\mathfrak a}\subset {\mathfrak m}_{\xi}$ generated by $d$-elements.
\end{proposition}

\begin{proof}
	To prove the statement it is enough to show that there are $d$-elements $\kappa_1,\ldots, \kappa_d\in {\mathfrak m}_{\xi}\setminus {\mathfrak m_{\xi}}^2$ such that if $\overline{\kappa_1},\ldots, \overline{\kappa_d}$ denote  their images  in $ {\mathfrak m}_{\xi} /{\mathfrak m}_{\xi}^2 $, then $\text{Gr}_{{\mathfrak m}_{\xi}}({\mathcal O}_{X,\xi})/\langle \overline{\kappa_1},\ldots, \overline{\kappa_d} \rangle$  is a graded ring of dimension 0 (see \cite[Theorem 10.14]{H_I_O}).
	
	Since $\dim_{k(\xi)}{\mathfrak m}_{\xi}/{\mathfrak m}_{\xi}^2=d+t$ and by hypothesis $\dim_{k(\xi)}\ker(\lambda_{\xi})=t$,  we can find generators  of ${\mathfrak m}_{\xi}$,
	\begin{equation}
		\label{completar_bases}
		\kappa_1,\ldots, \kappa_d, \delta_1, \ldots, \delta_t
	\end{equation}
	such that  $\overline{\delta_1}, \ldots, \overline{\delta_t}$ form a basis of $\ker(\lambda_{\xi})$. Notice that
	the elements   $\overline{\delta_1}, \ldots, \overline{\delta_e}$ are nilpotent in   $\text{Gr}_{{\mathfrak m}_{\xi}}({\mathcal O}_{X,\xi})/\langle \overline{\kappa_1},\ldots, \overline{\kappa_d} \rangle$
	(see \S \ref{DefLambdaBig}).
	Since  the graded ring $\text{Gr}_{{\mathfrak m}_{\xi}}({\mathcal O}_{X,\xi})$  is generated   in degree one by $\{\overline{\kappa_1},\ldots, \overline{\kappa_d}, \overline{\delta_1}, \ldots, \overline{\delta_t} \}$ it follows that  the quotient
	$\text{Gr}_{{\mathfrak m}_{\xi}}({\mathcal O}_{X,\xi})/\langle \overline{\kappa_1},\ldots, \overline{\kappa_e}\rangle $ is a graded ring of dimension zero and hence      $\langle \kappa_1,\ldots, \kappa_d\rangle$ is a reduction of ${\mathfrak m}_{\xi}$.
\end{proof}

Observe that the previous proposition holds for any local Noetherian ring in the extremal case. 

\begin{parrafo} \label{p_presentations_suitable}
{\bf (II) $p$-presentations and the computation of $\Hord_{X'}^{(d)}$.}
\end{parrafo}

\noindent Theorem \ref{presentaciones_mult} says that  the  ${\mathcal O}_{V^{(d+e)}}$-Rees algebra  $\mathcal{G}^{(d+e)}$    in (\ref{G_Representa})
represents the maximum multiplicity locus of $\Spec(B')$ in $V^{(d+e)}$ (see \S \ref{RepreMult}). We can assume that the order $m_i$ of each $f_i(x_i)\in S[x_i]$ is greater than $1$.   Notice also  that 
$\G_i^{(d+1)}:=\Diff(S[x_i][f_i(x_i)W^{m_i}])$ represents the maximum multiplicity of the hypersurface defined by $f_i(x_i)$ in $V_i^{(d+1)}=\text{Spec}(S[x_i])$, for $i=1,\ldots, e$.  By identifying $\G_i^{(d+1)}$ with its pull-back in $V^{(d+e)}$, we have that: 
\begin{equation}
	\label{amalgama}
	\mathcal{G}^{(d+e)}=\Diff(\G_1^{(d+1)})\odot \ldots \odot \Diff(\G_{e}^{(d+1)}).
\end{equation}

The natural inclusion $S\subset R=S[x_1,\ldots, x_{e}]$ induces  smooth projections, $\beta:  V^{(d+e)}\to V^{(d)}=\text{Spec}(S)$, and $\beta_i: V_i^{(d+1)}=\Spec(S[x_i])\to V^{(d)}=\text{Spec}(S)$ for $i=1,\ldots, e$. Also, observe   that $\tau_{{\mathcal G}^{(d+e)},\xi'}\geq e$.  This follows from the fact that the initial forms at $\xi'$ of the polynomials $f_i(x_i)\in S[x_i]$ depend on different variables 
(see \cite[\S 4.2]{BEP2} for further details). Hence,  $\beta$ is $\G^{(d+e)}$-admissible, and   each $\beta_i$ is $\G^{(d+1)}_i$-admissible. Thus   $\mathcal{G}^{(d)} =\mathcal{G}^{(d+e)}\cap S[W] $ is an elimination algebra of $\G^{(d+e)}$, and, moreover, up to integral closure,    
\begin{equation}\label{amalgama_eliminacion_1}
	\G^{(d)}= \G _1^{(d)} \odot \ldots \odot \G _{e}^{(d)} \subset S[W],
\end{equation}
where $\G _i^{(d)}$ is an elimination algebra of $\G^{(d+1)}_i$ on  $\Vd$ (see \cite[\S 3.8]{BEP2}).
\medskip

As indicated in Remark  \ref{presentacion_simp_mult}, $\mathcal{G}^{(d+e)}$ has the same integral closure as
\begin{equation}
	\label{primera_presentacion}
	R[f_i(x_i)W^{m_i}, \Delta_{x_i}^{j_i}(f_i(x_i))W^{m_i-j_i}]_{1\leq j_i\leq m_i-1} \odot \beta^*(\Gd),
\end{equation}
which in turns is a simplified presentation of $\mathcal{G}^{(d+e)}$ (see Theorem \ref{presentacion_previa}). We will write:
\begin{equation}
	\label{desc_f}
	f_i(x_i)=x_i^{m_i}+a_{1}^{(i)}x_i^{m_i-1}+\ldots+a_{m_i}^{(i)},
\end{equation}
with $a_j^{(i)}\in S$, for $j=1,\ldots, m_i$, and $i=1,\ldots, e$. 
\medskip

\noindent{\bf (A) The slope of a $p$-presentation at the closed point of $\Spec(B')$.}
\medskip

    \noindent     Suppose that $\xi'\in X'=\Spec(B')$ maps to $\xi$, and let ${\mathfrak m}_{\xi'}\subset B'$ be the corresponding maximal ideal.   Since the generic rank of $S\to B'$ equals the multiplicity at $\xi'$, by Zariski's multiplicity formula for finite projections (\cite[Chapter 8, \S 10, Theorem 24]{Z-SII}) we have that:
        \begin{enumerate}
            \item[(i)] The point $\xi'$ is the only one mapping to $\alpha(\xi')\in \Spec(S)$;
            \item[(ii)] The residue fields  $k(\xi')$  and $k(\alpha(\xi'))$ are equal;
            \item[(iii)] The expansion of the maximal ideal of $\alpha(\xi')$, ${\mathfrak m}_{\alpha(\xi')}B'$,  is a reduction of ${\mathfrak m}_{\xi'}$.
        \end{enumerate}

        From (ii) it follows that, after a translation of the form $\theta_i+s_i$, for some $s_i\in S$, we can also assume that $\theta_i\in {\mathfrak m}_{\xi'}$ for $i=1,\ldots, e$, and that in addition,
        ${\mathfrak m}_{\xi'}={\mathfrak m}_{\alpha(\xi')}B'+\langle \theta_1,\ldots, \theta_e\rangle$.

    Since   $\theta_i\in {\mathfrak m}_{\xi'}$, we have  that $\nu_{\alpha(\xi')}(a_j^{(i)})\geq 1$, for    $j=1,\ldots, m_i$ and $i=1,\ldots, e$ in (\ref{desc_f}).
Moreover, since $\xi'\in\Sing(\mathcal{G}^{(d+e)})$, necessarily 
$\nu_{\alpha(\xi')}(a_j^{(i)})=\nu_{\beta(\xi')}(a_j^{(i)})\geq j$.

        By \S \ref{Def_p_PresHiper} and Remark \ref{presentacion_simp_mult}, after applying suitable Taylor operators to the elements $f_i(x_i)\in R$, we get  that $\mathcal{G}^{(d+e)}$ is weakly equivalent to:
        \begin{equation}
        \label{segunda_presentacion}
        R[h_i(x_i)W^{p^{\ell_i}}, \Delta_{x_i}^{j_i}(h_i(x_i))W^{p^{\ell_i}-j_i}]_{1\leq j_i\leq \ell_i-1} \odot \beta^*(\Gd),
        \end{equation}
        where for each $i=1,\ldots, e$,  $h_i(x_i)\in S[x_i]\subset R$ is a monic polynomial of order $p^{\ell_i}$ for some $\ell_i\geq 1$,
        \begin{equation}
        \label{des_h}
h_i(x_i)=x_i^{p^{\ell_i}}+\tilde{a}_{1}^{(i)}x_i^{p^{\ell_i}-1}+
        \ldots+\tilde{a}_{p^{\ell_i}}^{(i)},
        \end{equation}
        with $\tilde{a}_j^{(i)}\in S$, for $j=1,\ldots, p^{\ell_i}$. Observe that
        $\nu_{\alpha(\xi')}(\tilde{a}_j^{(i)})\geq j$  for $j=1,\ldots, p^{\ell_i}$ and $i=1,\ldots, e$. Expression (\ref{segunda_presentacion})  is a  $p$-presentation ${\mathcal P}$ of $\mathcal{G}^{(d+e)}$ at $\xi$  (see \S \ref{Def_p_PresHiper} and Remark \ref{presentacion_simp_mult}). Notice that the differential operators in (\ref{segunda_presentacion}) are elements in $\Diff_{V^{(d+e)}/\Vd}$. 
\medskip

With the previous notation,  the slope of the $p$-presentation ${\mathcal P}$  at $\xi'$  (\ref{segunda_presentacion}) is 
        \begin{equation}\label{slope_p}
       Sl(\mathcal{P})(\xi')= \min_{i=1,\ldots,e} \left\{\frac{\nu_{\alpha(\xi')}(\tilde{a}^{(i)}_{p^{\ell_i}}) }{p^{\ell_i}}, \ord_{\alpha(\xi')}(\Gd)\right\}.
        \end{equation}
 From the exposition in \S\ref{cleaning}, it follows that a $p$-presentation ${\mathcal P}'$ with $Sl(\mathcal{P'})(\xi')=\Hord_{X'}^{(d)}(\xi')$ can be found  starting from the presentation ${\mathcal P}$   after considering translations of the form $\theta_i':=\theta_i+s_i$ with $s_i\in S$, and so that for each translation
\begin{equation}
\label{pendiente_hord_cerrado}
\nub_{{\mathfrak m}_{\alpha(\xi')}}(s_i)\geq \frac{{\nu_{{\mathfrak m}_{\alpha(\xi')}}(\widetilde{a}^{(i)}_{p^{\ell_i}}) }}{p^{\ell_i}}.
\end{equation} 
        \medskip
        
Finally, the restriction of $\mathcal{G}^{(d+e)}$ to $B'$, $\G_{B'}$, is finite over the expansion of $\Gd$ in $B'$, $\Gd B'$ (see \cite[Theorem 4.11]{V} and \cite[Corollary 7.7]{COA}). Write   $\G_{B'}=\oplus_nJ_nW^n$ and  define 
$$\overline{\ord}_{\xi'}(\G_{B'}):=\min\left\{\frac{\nub_{\xi'} (J_n)}{n}: n\in {\mathbb N}\right\}.$$
Then, by  
  Proposition \ref{ExtFinNuBar}, and using the fact that 
  ${\mathfrak m}_{\alpha(\xi')}B'$ is  a reduction of ${\mathfrak m}'$,     it can be checked that
\begin{equation}
\label{orden_G_J}
\ord_{\alpha(\xi')}(\Gd) = \overline{\ord}_{\xi'}(\G_{B'}), 
\end{equation}
(here it suffices to use   arguments  similar to those in the proof of  \cite[Proposition 0.20]{LejeuneTeissier1974}).

\

\noindent{\bf (B)  The slope of a $p$-presentation at non-closed points of  $\Spec(B')$.}
\medskip

\noindent With the same setting and notation as before, now let $\eta\in X$ be a non-closed point of multiplicity $m$  with $\xi\in \overline{\{\eta\}}$.  Let
 $\eta'\in \Spec(B')$ be a point mapping to $\eta$,   let ${\mathfrak p}_{\eta'}\subset B'$ be the corresponding prime and set ${\mathfrak p}_{\alpha(\eta')}:={\mathfrak p_{\eta'}}\cap S$. Again, by Zariski's multiplicity formula for finite projections we have that:
    \begin{enumerate}
    \item[(i')] The point $\eta'$ is the only one mapping to $\alpha(\eta')\in \Spec(S)$;
    \item[(ii')] The residue fields  $k(\eta')$  and $k(\alpha(\eta'))$ are equal;
    \item[(iii')] The expansion of the maximal ideal ${\mathfrak p}_{\alpha(\eta')}S_{{\mathfrak p}_{\alpha(\eta')}}$, ${\mathfrak m}_{\alpha(\eta')}B_{{\mathfrak p}_{\eta'}}$,  is a reduction of ${\mathfrak m}_{\eta'}:={\mathfrak p}_{\eta'}B'{\mathfrak p}_{\eta'}$.
\end{enumerate}
From  (i') it follows that $B'\otimes_SS_{{\mathfrak p}_{\alpha(\eta')}}$ is local (thus
$B'_{{\mathfrak p}_{\eta'}}=S_{{\mathfrak p}_{\alpha(\eta')}}[\theta_1,\ldots,\theta_e]$). By (ii'),   after translating $\theta_i$ by elements of $S_{{\mathfrak p}_{\alpha(\eta')}}$, we can assume that $\theta_i\in {\mathfrak m}_{\eta'}$.
The localization at $\eta'$  of the  p-presentation ${\mathcal P}$ at $\xi'$ (\ref{segunda_presentacion}) can be used to compute   $\Hord_{X'}^{(d)}(\eta')$.   
Interpreting $\eta'$ as a point in $V^{(d+e)}$, and using the fact that $\eta'\in \Sing(\G^{(d+e)})$, i.e., $\eta'$ is a point of multiplicity $m$ in $X'$,  it follows that
$\nu_{\alpha(\eta')}(a_{j}^{(i)})\geq j$ for $i=1,\ldots, e$, and $j=1,\ldots, m_e$
(see \cite[Propositions 5.4 and 5.7]{V}).
\medskip

\noindent{\bf  (C)  The slope of a $p$-presentation at non-closed points defining regular subschemes of  $\Spec(B')$.}
\medskip

\noindent Now suppose that   $\eta'$ is the generic point of a regular closed subscheme at $\xi'$. In such case, it  can be shown that 
        ${\mathfrak p}_{\alpha(\eta')}$ also defines a regular closed subscheme at $\alpha(\xi')$ (cf. \cite[Proposition 6.3]{V}). In addition,   after translating the elements $\theta_i$ by elements in $S$, it can be assumed that
        $B'=S[\theta_1,\ldots,\theta_e]$ with $\theta_i\in {\mathfrak p}_{\eta'}$, and that moreover, ${\mathfrak p}_{\alpha(\eta')}B$ is a reduction of ${\mathfrak p}_{\eta'}$ (without localizing at ${\mathfrak p}_{\eta'}$, see \cite[Lemma 3.6]{COA}).

    As  we argued above, again, interpreting $\eta'$ as a point in $V^{(d+e)}$, and using the fact that $\eta'\in \Sing(\G^{(d+e)})$, i.e., $\eta'$ is a point of multiplicity $m$ in $X'$,  it follows that
     $\nu_{\alpha(\eta')}(a_{j}^{(i)})\geq j$ for $i=1,\ldots, e$, and $j=1,\ldots, m_e$ in (\ref{desc_f}) 
(see \cite[Propositions 5.4 and 5.7]{V}).  But now, because  ${\mathfrak p}_{\alpha(\eta')}$ determines a closed  regular subscheme  at $\alpha(\xi')$, its ordinary powers and symbolic powers coincide on $S$. Therefore also
     $\nu_{{\mathfrak p}_{\alpha(\eta')}}(a_{j}^{(i)})\geq j$  for $i=1,\ldots, e$, and $j=1,\ldots, m_e$.  Hence  it follows that for the coefficients in (\ref{des_h}),
     \begin{equation}
     \label{des_h_p}
     \nu_{{\mathfrak p}_{\alpha(\eta')}}(\tilde{a}_j^{(i)})\geq j
     \end{equation}
for    $j=1,\ldots, p^{\ell_i}$ and $i=1,\ldots, e$.

        \

With the previous notation,  the slope of the $p$-presentation ${\mathcal P}$ at $\eta'$ (\ref{segunda_presentacion})  equals to:
        \begin{equation}\label{slope_p_no_cerrado}
         Sl(\mathcal{P})(\eta')=\min_{i=1,\ldots,e} \left\{\frac{\nu_{\alpha(\eta')}(\widetilde{a}^{(i)}_{p^{\ell_i}}) }{p^{\ell_i}}, \ord_{\alpha(\eta')}(\Gd)\right\}=
        \min_{i=1,\ldots,e}\left\{\frac{\nu_{{\mathfrak p}_{\alpha(\eta')}}(\widetilde{a}^{(i)}_{p^{\ell_i}}) }{p^{\ell_i}}, \ord_{{\mathfrak p}_{\alpha(\eta')}}(\Gd)\right\},
\end{equation}
see \cite[Definition 6.7]{BVIndiana}.   Going back to the discussion in \S\ref{cleaning},  recall that a $p$-presentations ${\mathcal P}'$ with $Sl(\mathcal{P'})(\eta')=\Hord_{X'}^{(d)}(\eta')$ can be found after considering translations of the form $\theta_i':=\theta_i+s_i$ with $s_i\in S$ and so that for each translation, 
\begin{equation}
\label{pendiente_hord_no_cerrado}
\nub_{{\mathfrak p}_{\alpha(\eta')}}(s_i)\geq \frac{{\nu_{{\mathfrak p}_{\alpha(\eta')}}(\widetilde{a}^{(i)}_{p^{\ell_i}}) }}{p^{\ell_i}}.
\end{equation}
We emphasize here that there is no need to consider translations with
$s_i\in S_{\mathfrak{p}_{\eta'}}$.
\medskip  

To conclude, considering  $\G_{B'}$ as before, recall that, 
  $\overline{\ord}_{\eta'}(\G_{B'})=\inf\left\{\frac{\nub_{\eta'} (J_n)}{n}: n\in {\mathbb N}\right\}$. 
Then, on the one hand,
$$\ord_{\alpha(\eta')}(\Gd) = \ord_{{\mathfrak p}_{\alpha(\eta')}}(\Gd).$$
On the other, since ${\mathfrak p}_{\alpha(\eta')}B'$ is a reduction of ${\mathfrak p}_{\eta'}$, and  $\Gd B'\subset \G_{B'}$ is a finite extension of Rees algebras,  by Proposition \ref{ExtFinNuBar}, and following similar arguments as in \cite[Proposition 0.20]{LejeuneTeissier1974},
$$\ord_{{\mathfrak p}_{\alpha(\eta')}}(\Gd) = \overline{\ord}_{{\mathfrak p}_{\eta'}}(\G'_B).$$
For similar reasons,
$$\ord_{\alpha(\eta')}(\Gd) = \overline{\ord}_{{\eta'}}(\G'_B).$$
Thus it follows that,
\begin{equation}
\label{orden_G_J_no_cerrado}
\overline{\ord}_{{\eta'}}(\G_{B'})=\ord_{\alpha(\eta')}(\Gd) = \ord_{{\mathfrak p}_{\alpha(\eta')}}(\Gd) = \overline{\ord}_{{\mathfrak p}_{\eta'}}(\G_{B'})=\min\left\{\frac{\nub_{{\mathfrak p}_{\eta'}} (J_n)}{n}: n\in {\mathbb N}\right\}.
\end{equation}
\end{parrafo}

\begin{parrafo}\label{suitable_presentations}
	{\bf  (III) Finding suitable algebraic presentations for $B'$} (for the extremal case).
	
\end{parrafo}

\noindent {\bf Closed points}

\begin{lemma} \label{base_ker} Let $B'=S[\theta_1,\ldots, \theta_e]$ be as in \S \ref{setting_prueba}, suppose that  the embedding dimension of $\xi'\in X'$ is $d+t$,  and  that $\xi'$ is in the extremal case.  Write  ${\mathfrak m}_{\alpha(\xi')}=  \langle y_1,\ldots, y_d\rangle$.  Then, after reordering the elements $\theta_i$  and  after considering translations of the form    $\theta_i'=\theta_i+s_i$ with $s_i\in S$,   it can be assumed that:
	\begin{itemize}
		\item[(i)]  $B'=S[\theta'_1,\ldots, \theta'_e]$, and
		\item[(ii)] $\{y_1,\ldots, y_d, \theta'_1, \ldots, \theta'_t\}$  is a minimal set of generators of ${\mathfrak m}_{\xi'}$ with $t\leq e$. 
	\end{itemize}
	Furthermore, 
	\begin{itemize} 
		\item[(iii)] For a given  a $\lambda_{\xi'}$-sequence,  $\{\delta_1,\ldots,\delta_t\}$,  after translating again the elements  $\theta'_i:=\theta_i+s_i$ for suitably chosen elements  $s_i\in S$,  we can assume that  $B'=S[\theta_1',\ldots, \theta_e']$, that
		$$\min \{\nub_{\xi'}(\theta'_i): i=1,\ldots, t, \ldots, e\}=\min\{\nub_{\xi'}(\theta'_i): i=1,\ldots, t\}\geq\min \{\nub_{\xi'}(\delta_i): i=1,\ldots, t\},$$
		and that $\{\theta_1',\ldots, \theta_t'\}$ is a $\lambda_{\xi'}$-sequence.
	\end{itemize}
	
\end{lemma}

\begin{proof} Recall that by \S \ref{p_presentations_suitable} (A),  maybe  after  translating    the $\theta_i$ by elements in $S$, it can be assumed that ${\mathfrak m}_{\xi'}= \langle y_1,\ldots, y_d, \theta_1, \ldots, \theta_e\rangle$ (here    we will identify $y_i$ with its image at $B'$).
Note that $\nub_{\xi'}(\theta_i)\geq 1$ for $i=1,\ldots,e$.
We can extract a minimal set of generators for ${\mathfrak m}_{\xi'}$ from the previous set,  and we can always assume that such a minimal set contains  $\{y_1,\ldots, y_d\}$ (see Proposition  \ref{suma_directa} (ii) and Remark \ref{RemNilpotent}).  After reordering the elements $\theta_i$, we can think  that such a minimal   set is of the form $\{y_1,\ldots, y_d, \theta_1,\ldots, \theta_t\}$. Thus conditions (i) and (ii) hold.
	
For condition (iii),
given a $\lambda_{\xi'}$-sequence, $\delta_1,\ldots,\delta_t$, by Proposition \ref{suma_directa} (ii),
we have that
$${\mathfrak m}_{\xi'}=\langle y_1,\ldots, y_{d}, \delta_1,\ldots, \delta_t\rangle,$$ 
and since $\theta_i\in {\mathfrak m}_{\xi'}$,
for   $i=1,\ldots,t$,  we can write,
$$\theta_{i}=p_{i,1}y_1+\ldots+p_{i,d}y_{d}+q_{i,1}\delta_1+\ldots+q_{i,t}\delta_{t},$$
where $p_{i,j}, q_{i,k}\in B'=S [\theta_1,\ldots, \theta_{t},\ldots, \theta_e]$ for $i=1\ldots, t$,  $j=1,\ldots,d$, and  $k=1,\ldots, t$.
For $i=1,\ldots, t$, and  $j=1,\ldots, d$, we can write
$$p_{i,j}=s_{i,j,0}+\sum_{i_1,\ldots,i_e}s_{i,j, i_1,\ldots, i_e}\theta_1^{i_1}\cdots\theta_e^{i_e},$$
with $s_{i,j,0}, s_{i,j, i_1,\ldots, i_e}\in S$ and $i_1+\ldots+i_e\geq 1$.
For $i=1,\ldots, t$, set
$$\theta_{i}':=\theta_{i}-s_{i,1,0}y_1-\ldots-s_{i,d,0}y_{d}.$$
Note that $B'=S[\theta_1',\ldots, \theta_t',\theta_{t+1},\ldots, \theta_e]$. In addition, since
$$\theta_{i}'=(p_{i,1}-s_{i,1,0})y_1+\ldots+(p_{i,d}-s_{i,d,0})y_d+q_{i,1}\delta_1+\ldots+q_{i,t}\delta_{t},$$
$\nub_{\xi'}((p_{i,j}-s_{i,j,0})y_j)\geq 2$ for $j=1,\ldots,d$, and    $\nub_{\xi'}(\delta_j)>1$ for $j=1,\ldots,t$, we have that
$\nub_{\zeta'}(\theta'_i)>1$ and that $\bar{\theta}'_i\in \ker(\lambda_{\zeta'})$.
Since
$$\langle y_1,\ldots, y_d,\theta_1,\ldots, \theta_t\rangle=\langle y_1,\ldots, y_d,\theta_1',\ldots, \theta_t'\rangle$$
it follows that
$\bar{\theta}'_1,\ldots,\bar{\theta}'_t\in {\mathfrak m}_{\xi'}/{\mathfrak m}_{\xi'}^2$ form a  basis of $\ker(\lambda_{\xi'})$.
Moreover by construction,
$$\nub_{\xi'}(\theta'_i)\geq\min\{1+\nub_{\xi'}(\theta_1),\ldots,1+\nub_{\xi'}(\theta_e),
	\nub_{\xi'}(\delta_1),\ldots,\nub_{\xi'}(\delta_t)\}.$$
Iterating this process we can assume that
$$\min\{\nub_{\xi'}(\theta'_i): i=1,\ldots, t\}\geq\min \{\nub_{\xi'}(\delta_i): i=1,\ldots, t\}.$$
	
Now suppose that there is some $j>t$ such that $\nub_{\xi'}(\theta_j)< \nub_{\xi'}(\theta'_i)$, for $i=1,\ldots, t$.
After reordering again, we can assume that $j=t+1$.
	
Repeating the previous argument,
$$\theta_{t+1}=p_1y_1+\ldots+p_dy_d+q_1\theta'_1+\ldots+q_{t}\theta'_{t},$$
where $p_i, q_j\in B'=S[\theta_1,\ldots, \theta_{t},\ldots, \theta_e]$ for $i=1\ldots, d$,  and $j=1,\ldots, t$.
	Now for $i=1,\ldots, d$, write
	$$p_i=s_{i,0}+\sum_{i_1,\ldots,i_e}s_{i, i_1,\ldots, i_e}\theta_1^{i_1}\cdots\theta_e^{i_e},$$
	with $s_{i,0}, s_{i, i_1,\ldots, i_e}\in S$ and $i_1+\ldots+i_e\geq 1$.
	Set
	$$\theta_{t+1}':=\theta_{t+1}-s_{1,0}y_1-\ldots-s_{d,0}y_d.$$
	Then
	$$\nub_{\xi'}(\theta_{t +1}')\geq \min\left\{\rule{0cm}{0.4cm} \nub_{\xi'}\left((p_1-s_{1,0})y_1+\ldots+ (p_d-s_{d,0})y_d\right), \nub_{\xi'}(q_1\theta'_1+\ldots+q_{t}\theta'_{t}) \right\}. $$
	Now,  it can be checked that either
	$$\nub_{\xi'}(\theta_{t+1}')\geq \min\{\nub_{\xi'}(\theta_i)+1: i=1,\ldots, e\}, $$
	or
	$$\nub_{\xi'}(\theta_{t+1}')\geq \min\{\nub_{\xi'}(\theta'_1),\ldots, \nub_{\xi'}(\theta'_{t})\}.$$
	Since  $B'=S[\theta'_1,\ldots, \theta'_{t},\theta_{t+1}',\theta_{t+2},\ldots,\theta_e]$, the claims in (iii)  follow   after a finite number of  translations  of the elements $\theta_i$ ($i=t+1,\ldots,e$) by elements in $S$. \end{proof}

\noindent {\bf Non-closed points} 

\medskip

\noindent To find suitable presentations of $B'$ that help us computing the Samuel slope at non-closed points, first we need a technical result, Lemma \ref{lucky_lemma} below. Then, a similar argument as the one exhibited in the proof of Lemma \ref{base_ker} will lead us to a similar statement (see Remark \ref{remark_lema_lema}). 

\begin{lemma}
	\label{lucky_lemma}
	Let $B'=S[\theta_1,\ldots, \theta_e]$ be as in \S \ref{setting_prueba}. Let $\xi'$ be the closed point of $\Spec(B')$ with multiplicity $m$, and  assume that $\eta'$ is a point of multiplicity $m$ defining  a regular subscheme in $\Spec(B')$. If $${\mathfrak m}_{\xi'}={\mathfrak m}_{\alpha(\xi')}B'+\langle \gamma_1,\ldots, \gamma_s\rangle$$
	with $\gamma_i\in {\mathfrak p}_{\eta'}$ for $i=1,\ldots, s$,         then
	$${\mathfrak p}_{\eta'}={\mathfrak p}_{\alpha(\eta')}B'+\langle \gamma_1,\ldots, \gamma_s\rangle.$$
\end{lemma}
\begin{proof}
	By the assumptions, there is a regular system of parameters at $S$, $y_1,\ldots, y_d$, such that ${\mathfrak p}_{\alpha(\eta')}=\langle y_1,\ldots, y_r\rangle$ for some $r<d$ and ${\mathfrak m}_{\alpha(\xi')}=\langle y_1,\ldots, y_d\rangle$. Since $S\to B'$ is an inclusion, we will identify $y_i$ with its image at $B'$.  
	We have that,
	\begin{equation}
		\label{cadena_incluiones}
		\langle y_1,\ldots, y_r,  \gamma_1,\ldots,  \gamma_s\rangle \subset
		{\mathfrak p}_{\eta'}.
	\end{equation}
	Let $\overline{B'}=B'/\langle y_1,\ldots, y_r,  \gamma_1,\ldots, \gamma_s\rangle$.
	Notice now that
	$$d-r=\dim(B'/\mathfrak{p}_{\eta'})\leq 
	\dim(\overline{B'}) \leq d-r,$$
	where the last inequality follows because
	$\mathfrak{m}_{\xi'}/\langle y_1,\ldots, y_r,  \gamma_1,\ldots, \gamma_s\rangle$
	can be generated by $d-r$ elements.
	Therefore $\overline{B'}$ is a regular local ring of dimension $d-r$ and the inclusion 
	(\ref{cadena_incluiones}) is an equality.
\end{proof}

\begin{remark} \label{remark_lema_lema}  With the same assumptions as in Lemma \ref{base_ker}, assume now that  $\eta'\in  X'$ is a point of multiplicity $m$ defining a regular closed subscheme at $\xi'$.  Let ${\mathfrak p}_{\eta'}\subset {\mathfrak m}_{\xi'}$ be the corresponding  prime,  and  suppose that 
	$${\mathfrak p}_{\eta'}={\mathfrak p}_{\alpha(\eta')}+\langle \gamma_1,\ldots, \gamma_s\rangle,$$
	for some $\gamma_1,\ldots, \gamma_s\in B'$. 
	Then, using a similar argument as the one given in the proof of Lemma \ref{base_ker} (iii),  it can be proven that,  after reordering the elements $\theta_i$,  and  after considering translations of the form    $\theta_i'=\theta_i+s_i$ with $s_i\in S$,   it can be assumed that 
	$B'=S[\theta_1',\ldots, \theta_e']$ and
	\begin{equation} \label{Ineq_lema_lema}
		\min\{\nub_{{\mathfrak p}_{\eta'}}(\theta'_i): i=1,\ldots, e\}\geq\min \{\nub_{{\mathfrak p}_{\eta'}}(\gamma_i): i=1,\ldots, s\}.
	\end{equation}	  
	 To see this it suffices to observe that since 
	${\mathfrak p}_{\eta'}$ defines a regular prime at ${\mathfrak m}_{\xi'}$,  after translating the elements $\theta_i$ if needed, we may assume that $\theta_i\in {\mathfrak p}_{\eta'}$ for $i=1,\ldots, e$ (see \S\ref{setting_prueba}~(C)).   Then we can select a regular system of parameters at $S$, $y_1,\ldots, y_r, y_{r+1},\ldots, y_d$,  so that
	$${\mathfrak p}_{\alpha(\eta')}=\langle y_1,\ldots, y_r\rangle.$$
	Now  for $i=1,\ldots, e$,
	$$\theta_{i}=p_{i,1}y_1+\ldots+p_{i,r}y_{r}+q_{i,1}\gamma_1+\ldots+
	q_{i,s}\gamma_{s},$$
	where $p_{i,j}, q_{i,k}\in B'=S [\theta_1,\ldots, \theta_{t},\ldots, \theta_e]$ for $i=1\ldots, e$,  $j=1,\ldots,r$, and  $k=1,\ldots, s$.
	For $i=1,\ldots, e$, and  $j=1,\ldots, r$, we can write
	$$p_{i,j}=s_{i,j,0}+\sum_{i_1,\ldots,i_e}s_{i,j, i_1,\ldots, i_e}\theta_1^{i_1}\cdots\theta_e^{i_e},$$
	with $s_{i,j,0}, s_{i,j, i_1,\ldots, i_e}\in S$ and $i_1+\ldots+i_e\geq 1$.
	Set
	$$\theta_{i}':=\theta_{i}-s_{i,1,0}y_1-\ldots-s_{i,d,0}y_{r}\in\mathfrak{p}_{\eta'} .$$
	Finally (\ref{Ineq_lema_lema}) follows using the same argument as in Lemma \ref{base_ker} (iii). 
\end{remark}

Now we are ready to address the proof of our main theorem: 

\begin{theorem} \label{maintheorem}
 	Let $X$ be an equidimensional variety of dimension $d$ defined over a perfect field $k$. Let $\zeta\in X$ be a point of multiplicity $m>1$. Then: 
 	\begin{itemize}
 		\item If ${\mathcal S}\text{-Sl}({\mathcal O}_{X,\zeta})=1$, then
 		$$1={\mathcal S}\text{-Sl}({\mathcal O}_{X,\zeta})=\Hord^{(d)}_X(\zeta)\leq \ord^{(d)}_X(\zeta).$$
 		In addition, if $\zeta$ is a closed point then also $\ord^{(d)}_X(\zeta)=1$.
 		\item If ${\mathcal S}\text{-Sl}({\mathcal O}_{X,\zeta})>1$, then    
 		$$\Hord^{(d)}_X(\zeta)=\min \{{\mathcal S}\text{-Sl}({\mathcal O}_{X,\zeta}), \ord^{(d)}_X (\zeta)\}.$$
 	\end{itemize}
\end{theorem} 
\begin{proof} {\bf Closed points}. 
Assume that $\zeta$ is a closed point and denote it  by $\xi\in X$. Let  $t=t_{\xi}$ be the excess of embedding dimension.
After an \'etale extension $(B', {\mathfrak m}_{\xi'}, k(\xi'))$ of   $({\mathcal O}_{X,\xi}, {\mathfrak m}_{\xi}, k(\xi))$    we can assume the setting and the notation described  in \S \ref{setting_prueba}, where  $B'=S[\theta_1,\ldots, \theta_e]$.  After translating the $\theta_i$ if needed, we  have that  
		\begin{equation}\label{generadores_maximal} 
		{\mathfrak m}_{\xi'}={\mathfrak m}_{\alpha(\xi')}+\langle \theta_1,\ldots.  \theta_e\rangle.
		\end{equation} 
Recall  that $\Gr_{\alpha(\xi')}(S)\to \Gr_{\xi'}(B')$ is a finite extension which induces an inclusion 
in degree one (see Proposition \ref{suma_directa} (ii)).
Therefore, any regular system of parameters generating ${\mathfrak m}_{\alpha(\xi')}$, $y_1,\ldots, y_d$, 
can be considered as part of a minimal set of generators of ${\mathfrak m}_{\xi'}$. Recall in addition that $\nub_{\xi'}(y_i)=1$, for $i=1,\ldots, d$. 

Continuing with the setting in \S\ref{setting_prueba},     the Rees algebra ${\mathcal G}^{(d+e)}$ is  weakly equivalent to the Rees algebra in (\ref{segunda_presentacion}), which in turn is a p-presentation of $\G^{(d+e)}$ (see   \S \ref{p_presentations_suitable} (A)). Since $h_i(x_i)W^{p^{\ell_i}}\in\mathcal{G}^{(d+e)}$, we have that $\nu_{\xi'}(h_i(x_i))\geq p^{\ell_i}$ in $V^{(d+e)}$, and hence 
$\bar{\nu}_{\xi'}(h_i(\theta_i))\geq p^{\ell_i}$ in $\Spec(B')$ for $i=1,\ldots,p^{\ell_i}$ (see (\ref{orden_G_J})).
Note here that if $h_i(\theta_i)=0\in B'$, then $\bar{\nu}_{\xi'}(h_i(\theta_i))=\infty$, but the arguments below go through even in this case.  

\medskip

\noindent{\bf Closed points in the non-extremal case.} 
If   $\dim_{k(\xi)}\ker(\lambda_{\xi})< t$,  then   $\dim_{k(\xi')}\ker(\lambda_{\xi'})< t$ 
(see Lemma  \ref{nucleo_etale}), and  hence, necessarily, $\nub_{\xi'}(\theta_i)=1$ for some $i\in \{1,\ldots, e\}$.   Without loss of generality we can assume that
$\nub_{\xi'}(\theta_1)=1,\ldots,\nub_{\xi'}(\theta_c)=1$ and
$\nub_{\xi'}(\theta_{c+1})>1,\ldots,\nub_{\xi'}(\theta_e)>1$ for some $c\in \{1,\ldots, e\}$.

Since the assumption is that $\nub_{\xi'}(\theta_i)=1$, for $i=1,\ldots,c$,  we will pay special attention to $h_i(\theta_i)W^{p^{\ell_i}}$ for $i=1,\ldots, c$. 
To start with, by Definition \ref{DefFuncOrd} and the properties in Proposition \ref{PotenciaNuBar}, we have that
		\begin{equation}\label{orden_suma_b}
		\bar{\nu}_{\xi'}(h_i(\theta_i)-\theta_i^{p^{\ell_i}})\geq 
		\min_{j=1,\ldots,p^{\ell_1}}\left\{\bar{\nu}_{\xi'}(\tilde{a}_j^{(i)}
		\theta_i^{p^{\ell_1}-j})\right\}\geq 
		\min_{j=1,\ldots,p^{\ell_i}}\left\{\bar{\nu}_{\xi'}(\tilde{a}_j^{(i)})+(p^{\ell_i}-j))\right\}.
		\end{equation}
Next,  we will distinguish different cases depending on the values $\nub_{\xi'}(\tilde{a}_j^{(i)})/j$.
Recall that $\bar{\nu}_{\xi'}(\tilde{a}_j^{(i)})=\nu_{\alpha(\xi')}(\tilde{a}_j^{(i)})\geq j$ for $j=1,\ldots, p^{\ell_i}$ and $i=1,\ldots, e$ (see Proposition~\ref{ExtFinNuBar}).
		
\

{\bf  {\em Case (a).}} There exists some $i\in\{1,\ldots,c\}$ such that $\bar{\nu}_{\xi'}(\tilde{a}_j^{(i)})>j$  for all $j=1,\ldots,p^{\ell_1}$. Then by    Remark  \ref{minimo_suma},  and by  (\ref{orden_suma_b}), for that index $i$, 
		\begin{equation*}
		\bar{\nu}_{\xi'}(h_i(\theta_i))=
		\min\left\{\bar{\nu}_{\xi'}(\theta_i^{p^{\ell_i}}),\bar{\nu}_{\xi'}(h_i(\theta_i)-\theta_i^{p^{\ell_i}})\right\}=
		p^{\ell_i}
		\end{equation*}
		from where it follows  that $\overline{\ord}_{\xi'}(\mathcal{G}_{B'})=\ord_{\alpha(\xi')}(\mathcal{G}^{(d)})=1$.
		Here we use that $h_i(\theta_i)W^{p^{\ell_i}}\in \G_{B'}$ and (\ref{orden_G_J}). 
		
		\

{\bf {\em Case (b).}} For each $i\in\{1,\ldots,c\}$ there exist some $j\in \{1,\ldots,p^{\ell_1}\}$ such that $\nub_{\xi'}(\tilde{a}_{j}^{(i)})=j$.
Here we distinguish two cases: 

{\bf {\em Case (b.1).}} If $j\in\{1,\ldots,p^{\ell_1}-1\}$,    then by Remark \ref{remark_ordenes_intermedios},    
		\begin{equation*}
		1=\min_{j=1,\ldots,p^{\ell_i}-1}\left\{
		\frac{\nu_{\alpha(\xi')}(\tilde{a}_{j}^{(i)})}{j}\right\}\geq  \ord_{\alpha(\xi')}(\mathcal{G}^{(d)})\geq 1,
		\end{equation*}
		hence $\ord_{\alpha(\xi')}(\mathcal{G}^{(d)})=1$.

{\bf {\em Case (b.2).}} Assume that  for all $i=1,\ldots,c$   we have $\nub_{\xi'}(\tilde{a}_j^{(i)})>j$ for  $j=1,\ldots,p^{\ell_1}-1$ and 
$\nub_{\xi'}(\tilde{a}_{p^{\ell_i}}^{(i)})=p^{\ell_i}$.
After replacing $\theta_i$ by $\theta_i+s_i$, for some $s_i\in S$, we may  assume  that the initial part of $\tilde{a}_{p^{\ell_i}}^{(i)}$ is not  a $p^{\ell_i}$-th power 
(here we consider  $\In_{\alpha(\xi')}(\tilde{a}_{p^{\ell_i}}^{(i)})=H(Y_1,\ldots,Y_d)\in\Gr_{\alpha(\xi')}(S)$ as a homogeneous polynomial of degree $p^{\ell_i}$, see \S\ref{cleaning} and Definition~\ref{DefHiperNormalForm}).
Note that the elimination algebra is invariant by the change $\theta_i\to\theta_i+s_i$ (see Example~\ref{Emininacion_Hiper}).
Observe that that now $\nub_{\xi'}(\theta_i+s_i)\geq 1$ but from our hypothesis there must be at least one $\theta_i+s_i$ such that
$\nub_{\xi'}(\theta_i+s_i)=1$.
Setting $\theta'_i=\theta_i+s_i$ after relabeling if needed we can assume $\nub_{\xi'}(\theta'_1)=\cdots=\nub_{\xi'}(\theta'_{c'})=1$,
for some $c'\leq c$.
If some $\theta'_i$ falls into cases (a) or (b.1) we are done, and $\ord^{(d)_{X}}(\xi)=1$.

Otherwise if all $\theta'_i$, $i=1,\ldots,c'$, with $c'\geq 1$ are in case (b.2), then
it follows that $\Hord_X^{(d)}(\xi')=1$.
In such case, moreover, since $\xi$ is a closed point and
the initial part of $\tilde{a}_{p^{\ell_1}}^{(1)}$ has some term  which is not an $p^{\ell_1}$-th power,
there is a differential operator $D$ in $S$ of order $b<p^{\ell_1}$ such that $\nu_{\alpha(\xi')}(D(\tilde{a}_{p^{\ell_1}}^{(1)}))=p^{\ell_1}-b$.
Now,  $D$ is also  a differential operator in $S[x_1,\ldots,x_e]$, thus      we have that
$D(h_1(x_1))W^{p^{\ell_1-b}}\in\mathcal{G}^{(d+e)}$, since $\mathcal{G}^{(d+e)}$ is differentially saturated. Finally, observe that
		\begin{equation*}
		D(h_1(x_1))=D(\tilde{a}_{p^{\ell_1}}^{(1)})+\tilde{\tilde{a}}_1x_1^{p^{\ell_1}-1}+\cdots +\tilde{\tilde{a}}_{p^{\ell_1}-1}x_1.
		\end{equation*}
		Using the same argument as in the proof of Theorem 4.4 in \cite{BVComp} (page 1286) it follows that the norm
		of $D(h_1(x_1))$ is  an element of order one in $\mathcal{G}^{(d)}$,  
		and hence $\ord_{\alpha(\xi')}(\mathcal{G}^{(d)})=1$.
		
		To conclude, for all the cases $\ord_{\alpha(\xi')}(\mathcal{G}^{(d)})=1$, and by Theorem \ref{ThBenVilla} and Remark \ref{presentacion_simp_mult},
		\begin{equation*}
		\min_{j=1,\ldots,p^{\ell_1}}\left\{
		\frac{\nu_{\alpha(\xi')}(\tilde{a}_j^{(1)})}{j}, \ord_{\alpha(\xi')}(\mathcal{G}^{(d)})
		\right\}=
		\min\left\{
		\frac{\nu_{\alpha(\xi')}(\tilde{a}_{p^{\ell_1}}^{(1)})}{p^{\ell_1}}, \ord_{\alpha(\xi')}(\mathcal{G}^{(d)})
		\right\}.
		\end{equation*}
		Hence $\Hord_{X}^{(d)}(\xi)=\Hord_{X'}^{(d)}(\xi')=\ord_{\alpha(\xi')}(\mathcal{G}^{(d)})=1$.
	
	\medskip	
		 
		\noindent{\bf Closed points in the extremal case.}  
		By Lemma \ref{base_ker}, we can assume that ${\mathfrak m}_{\xi'}={\mathfrak m}_{\alpha(\xi')}+\langle \theta_1,\ldots, \theta_t\rangle$
		with $t\leq e$, that  $$\min\{\nub_{\xi'}(\theta_i): i=1,\ldots, t\}=\min \{\nub_{\xi'}(\theta_i): i=1,\ldots, t, \ldots, e\},$$ and that $\nub_{\xi'}(\theta_i)>1$ for $i=1,\ldots, t$. Thus
		$\{{\theta_1},\ldots, {\theta_t}\}$ is $\lambda_{\xi}$-sequence.

		Recall that by Remark \ref{remark_ordenes_intermedios}, for every  $i\in {1,\ldots,e}$, and each $j=1,\ldots, {p^{\ell_i}}-1$,
		\begin{equation}
			\label{desigualdad_intermedia}
			\ord_{\alpha(\xi')}(\Gd) \leq \frac{\nu_{{\alpha(\xi')}}(\tilde{a}_j^{(i)})}{j},
		\end{equation}
		
		Since   $h_i(\theta_i)W^{p^{\ell_i}}\in \G_{B'}$, we have that
		\begin{equation}
			\label{desigualdad_orden}
			\overline{\ord}_{\xi'}(h_i(\theta_i)W^{p^{\ell_i}})\geq \overline{\ord}_{\xi'} ({\G}_{B'})=\ord_{\alpha(\xi')}\Gd,
		\end{equation}
		(see (\ref{orden_G_J})).
		We will   distinguish two cases:
		
		\
		
		{\bf {\em Case (a')}}  Suppose that $\nub_{{\xi'}}(\theta_i) \geq \ord_{\alpha(\xi)}\Gd$ for all $i\in \{1,\ldots, t\}$. Then
		\begin{equation}
			\label{comparacion}
			{\mathcal S}\text{-Sl}({\mathcal O}_{X',\xi'}) \geq \ord_{\alpha(\xi)}\Gd.
		\end{equation}
		In addition, for $i=1,\ldots,t,\ldots, e$ we have also
		$\nub_{{\xi'}}(\theta_i) \geq \ord_{\alpha(\xi)}\Gd$, and  by (\ref{desigualdad_intermedia}),
		$$\frac{\nub_{{\xi'}}\left(\theta_i^{p^{\ell_i}}+\tilde{a}_{1}^{(i)}\theta_i^{p^{\ell_i}-1}+\ldots+\tilde{a}_{p^{\ell_i}-1}^{(i)}\theta_i\right)} {p^{\ell_i}} \geq \ord_{\alpha(\xi')}\Gd, $$
		for  $i=1,\ldots, e$. As a consequence, by (\ref{desigualdad_orden}) and Remark \ref{minimo_suma},
		$$\frac{\nub_{{\xi'}}(\tilde{a}_{p^{\ell_i}}^{(i)})}{p^{\ell_i}}=\frac{\nu_{\alpha({\xi'})}(\tilde{a}_{p^{\ell_i}}^{(i)})}{p^{\ell_i}}\geq \ord_{\alpha(\xi')}\Gd.$$
		Therefore,
		$$Sl(\mathcal{P})(\xi')
		=\min\left\lbrace \frac{\nu_{\alpha(\xi')}(\tilde{a}_{p^{\ell_i}}^{(i)})}{p^{\ell_i}}, \ord_{\alpha(\xi')}(\Gd) \right\rbrace= \ord_{\alpha(\xi')}(\Gd)=\Hord_{X'}^{(d)}(\xi').$$
		Thus, by (\ref{comparacion}),
		$$\Hord^{(d)}_{X'}(\xi')=\min \{{\mathcal S}\text{-Sl}({\mathcal O}_{X',\xi'}), \ord^{(d)}_X (\xi')\}.$$

		\
		
		{\bf {\em Case (b')}} Suppose that $\nub_{{\xi'}}(\theta_i) < \ord_{\alpha(\xi')}\Gd$ for some $i\in \{1,\ldots,t\}$.
		We will  prove  that in this case:
		\begin{equation}\label{igualdad_d_e}
			\min_{i=1,\ldots,t}\left\{\nub_{{\xi'}}(\theta_i), \ord_{\alpha(\xi')}(\Gd)\right\} =Sl(\mathcal{P})(\xi')
			=\min_{i=1,\ldots,e}\left\lbrace \frac{\nu_{\alpha(\xi')}(\tilde{a}_{p^{\ell_i}}^{(i)})}{p^{\ell_i}}, \ord_{\alpha(\xi')}(\Gd) \right\rbrace.
		\end{equation}
		By (\ref{desigualdad_orden}) and Remark \ref{minimo_suma},   either $\nub_{\xi'}(\theta_i^{p^{\ell_i}})=\nub_{\xi'} (\tilde{a}_j^{(i)}\theta_i^{p^{\ell_i}-j})$ for some $j\in \{1,\ldots, p^{\ell_i}-1\}$, or else $\nub_{\xi'}(\theta_i^{p^{\ell_i}})=\nub_{\xi'} (\tilde{a}_{p^{\ell_i}}^{(i)})$. In the first case, we would have that $\nub_{\xi'} (\theta_i^{p^{\ell_i}})=\nub_{\xi'} \left(\tilde{a}_j^{(i)}\theta_i^{p^{\ell_i}-j}\right)$
		which by Proposition \ref{ExtFinNuBar}
		implies that
		$$p^{\ell_i}\nub_{\xi'}(\theta_i)=\nub_{\xi'}(\tilde{a}_{j}^{(i)})+(p^{\ell_i}-j)\nub_{\xi'}(\theta_i),$$
		and therefore, $\nub_{\xi'}(\theta_i)= \nub_{\xi'}(\tilde{a}_{j}^{(i)})/j= \nu_{\alpha(\xi')}(\tilde{a}_{j}^{(i)})/j\geq \ord_{\alpha(\xi')} \Gd$ (by Remark \ref{remark_ordenes_intermedios}) which is a contradiction.
		Thus, necessarily, $\nub_{\xi'}(\theta_i)=\nub_{\xi'} (\tilde{a}_{p^{\ell_i}}^{(i)})/{p^{\ell_i}}=\nu_{\alpha(\xi')} (\tilde{a}_{p^{\ell_i}}^{(i)})/{p^{\ell_i}}< \ord_{\alpha(\xi')}\Gd$   (since by assumption $\nub_{{\xi'}}(\theta_i) < \ord_{\alpha(\xi')}\Gd$). 
		
		Conversely, if for some $i=1,\ldots,e$, ${\nu_{\alpha(\xi')}(\tilde{a}_{p^{\ell_i}}^{(i)})}/{p^{\ell_i}}<\ord_{\alpha(\xi')}(\Gd)$, then this leads
		to $\nub_{\xi'}(\theta_i)= \nub_{\xi'}(\tilde{a}_{p^{\ell_i}}^{(i)})/p^{\ell_i}= \nu_{\alpha(\xi')}(\tilde{a}_{p^{\ell_i}}^{(i)})/p^{\ell_i}$.     
		Hence equality (\ref{igualdad_d_e}) holds.
		\medskip
		
		Now we check that  the theorem follows from here for $\xi'\in X'$.   On the one hand, by Lemma \ref{base_ker}, for each $\lambda_{\xi'}$-sequence, $\delta_1,\ldots, \delta_t$, we can find suitable elements  $\theta_1,\ldots, \theta_t, \ldots, \theta_e$, so that $B'=S[\theta_1,\ldots, \theta_e]$
		and
		$$\min_{i=1,\ldots,t}\{\nub_{\xi'}(\theta_i)\} \geq \min_{i=1,\ldots,t}\{\nub_{\xi'}(\delta_i)\},$$
		for which we either  fall in case (a'), or else we fall in  case (b') and then equality (\ref{igualdad_d_e}) holds.
		
		On the other hand, higher values of $Sl(\mathcal{P})(\xi')$ are only obtained in case (b') after translations on the coefficients $\tilde{a}_{p^{\ell_i}}^{(i)}$ by elements on $S$. These in turn induce changes of the form $\theta_i':=\theta_i+s_i$, with $s_i\in {\mathfrak m}_{\alpha(\xi')}$ and with the additional property pointed out in (\ref{pendiente_hord_cerrado}), thus  
		$\nub_{\xi'}(\theta_i')\geq \nub_{\xi'}(\theta_i)$ for $i=1,\ldots, e$. Observing that also $B'=S[\theta_1',\ldots,\theta_e']$,  again by Lemma \ref{base_ker}, we can extract a $\lambda_{\xi'}$-sequence among $\theta_1',\ldots, \theta_e'$ for which we either fall in case (a') or else we fall in case (b') and equality (\ref{igualdad_d_e}) holds.
		\medskip
		
To conclude, to check that the theorem    holds at $({\mathcal O}_{X,\xi}, {\mathfrak m}, k(\xi))$ it suffices to observe that by Proposition \ref{Caso_tau_e_reformulado_b} and Remark \ref{lema_lema}, $\text{Gr}_{{\mathfrak m}_{\xi}}({\mathcal O}_{X,\xi})=\text{Gr}_{{\mathfrak m'}}(B')$. Therefore, the theorem  follows from Proposition \ref{Prop_Slope_etale}.

\

\noindent{\bf Non-closed points}. Let $\zeta=\eta\in X$ be a non-closed point of multiplicity $m\geq 1$.
Denote by $\mathfrak{p}_\eta$ the prime defined by $\eta$ in some affine open set $U\subset X$.
Choose a closed point 
\begin{equation}
	\label{non_closed_conditions} 
	\xi\in\overline{\{\eta\}}\subseteq X
\end{equation}
with the following conditions: 
\begin{enumerate}
		\item $\xi$ and $\eta$ has the same multiplicity $m$.
			\item $\mathcal{O}_{X,\xi}/\mathfrak{p}_{\eta}$ is a regular local ring of dimension $d-r$ for $r\geq 1$.
		\end{enumerate}
Let $B=\mathcal{O}_{X,\xi}$, and let $B\longrightarrow B'$ an \'etale extension, and $S\longrightarrow B'$ a finite morphism as in \S\ref{setting_prueba}. Denote by and $\mathfrak{p}_{\eta'}$ the prime dominating $\mathfrak{p}_\eta$, $\xi'$ the closed point dominating $\xi$, and $\p_{\alpha(\eta')}$ the prime $\p_{\eta'}\cap S$. By \cite[Corollary 3.2]{COA}, $\p_{\alpha(\eta')}$ determines a regular prime.  Under these assumptions, using \cite[Lemma 3.6]{COA}, we can assume that $B'=S[\theta_1,\ldots,\theta_e]$ with $\theta_i\in \p_{\eta'}$  (see \S \ref{p_presentations_suitable} (C)). Note that $\nub_{\p_{\eta'}}(\theta_i)\geq 1$  and $\nub_{{\xi'}}(\theta_i)\geq 1$, for $i=1,\ldots,e$. Since $\theta_i\in \p_{\eta'}$, it follows that $\p_{\eta'}=\p_{\alpha(\eta')}+\langle \theta_1,\ldots,\theta_e\rangle$.

 \medskip

\noindent {\bf Non-closed points in the non-extremal case.} 
If  $\eta'$ is not in the extremal case, necessarily $\nub_{\mathfrak{p}_{\eta'}}(\theta_i)=1$ for some $i$.
After reordering we may assume that
$\nub_{{\eta'}}(\theta_1)=1,\ldots,\nub_{{\eta'}}(\theta_c)=1$ and
$\nub_{{\eta'}}(\theta_{c+1})>1,\ldots,\nub_{{\eta'}}(\theta_e)>1$.
Note that, in particular, $\nub_{\mathfrak{p}_{\eta'}}(\theta_i)=1$ for $i=1,\ldots,c$.

Now using the fact that 
$$\nub_{\eta'}(\tilde{a}^{(i)}_j)=\nu_{\alpha(\eta')}(\tilde{a}^{(i)}_j)=\nu_{\mathfrak{p}_{\alpha(\eta')}}(\tilde{a}^{(i)}_j)=
\nub_{\mathfrak{p}_{\eta'}}(\tilde{a}^{(i)}_j)$$
cases (a) and (b.1) follow using the same argument as in the closed point case.
Observe that in case (b.2)  if $\nub_{\mathfrak{p}_{\eta'}}(\tilde{a}_{p^{\ell_i}}^{(i)})=p^{\ell_i}$,
after replacing $\theta_i$ by $\theta_i+s_i$, for some $s_i\in S$, we may  assume  that the initial part of $\tilde{a}_{p^{\ell_i}}^{(i)}$ is not  a $p^{\ell_i}$-th power 
(here we consider  $\In_{\mathfrak{p}_{\alpha(\eta')}}(\tilde{a}_{p^{\ell_i}}^{(i)})=H(Y_1,\ldots,Y_r)\in\Gr_{\mathfrak{p}_{\alpha(\eta')}}(S)$ as a homogeneous polynomial of degree $p^{\ell_i}$).
Here there is no need to localize as it is shown in the proof of \cite[Proposition 5.8]{BVComp}.

After the translations $\theta_i+s_i$ we may fall into cases (a), (b.1) or (b.2).
From here it follows that $\Hord^{(d)}_X(\eta)=1$.

\

\noindent {\bf Non-closed points in the extremal case.}  Here  we can  repeat the arguments in cases (a') or  (b')   for $B'_{{\mathfrak p}_{\eta'}}=S_{{\mathfrak p}_{\alpha(\eta')}}[\theta_1,\ldots,\theta_e]$ where the arguments are valid for a local ring  (see \ref{p_presentations_suitable} (B)). Thus:
\begin{equation}
	\label{igualdad_h_puntos_no_cerrados}
	\Hord^{(d)}_{X'}(\eta')= \min \{\SSl({\mathcal O}_{X',\eta'}), \ord^{(d)}_{X'} (\eta')\}.
\end{equation}
We have that  $\SSl({\mathcal O}_{X,\eta})\leq\SSl({\mathcal O}_{X',\eta'})$. To
prove the theorem for $\eta\in X$ we will want to use Proposition \ref{prop_sucesiones}. But to do so, among other things, we need to show that
there  is some  $\lambda_{\eta'}$-sequence in $B'$ (without localizing at ${\mathfrak p}_{\eta'}$),
that is also a $\lambda_{\xi'}$-sequence, $\gamma'_1,\ldots, \gamma'_{t_{\eta'}}\in {\mathfrak p}_{\eta'}$,   for which  the following equality holds: 
\begin{equation}
	\label{igualdad_h_puntos_no_cerrados_sin_localizar}
	\Hord^{(d)}_{X'}(\eta')= \min \{\nub_{\eta'}(\gamma'_1),\ldots, \nub_{\eta'}(\gamma'_{t_\eta'}), \ord^{(d)}_{X'} (\eta')\}.
\end{equation} 
From here the theorem will follow for $\calo_{X,\eta}$ because  
\begin{itemize}
	\item either $\Hord^{(d)}_{X}(\eta)=\Hord^{(d)}_{X'}(\eta')=\ord^{(d)}_{X'}(\eta')=\ord^{(d)}_{X}(\eta)$, and applying Proposition \ref{prop_sucesiones}
	to $\gamma'_1,\ldots, \gamma'_{t_{\eta'}}$ we would get that:
	$$\SSl({\mathcal O}_{X,\eta})\geq\min \{\nub_{\eta'}(\gamma'_1),\ldots, \nub_{\eta'}(\gamma'_{t_\eta'})\}\geq \ord^{(d)}_{X'}(\eta')=\ord^{(d)}_{X}(\eta);$$
	
	\item or  
	$\Hord^{(d)}_{X}(\eta)=\Hord^{(d)}_{X'}(\eta')=\SSl({\mathcal O}_{X',\eta'})$,  and,  again, by Proposition \ref{prop_sucesiones} applied to the same sequence we would get that: 
	$$\SSl({\mathcal O}_{X,\eta})=\SSl({\mathcal O}_{X',\eta'}).$$
\end{itemize}

\noindent To find $\gamma'_1,\ldots, \gamma'_{t_{\eta'}}\in {\mathfrak p}_{\eta'}\subset B'$,
with the previous properties, we will proceed as follows.
\medskip

Using the same arguments as in the proof of Proposition \ref{extremal_especial} below, the  closed point $\xi\in\overline{\{\eta\}}\subseteq X$ in (\ref{non_closed_conditions}) can be selected so that in addition to (1) and (2) it also satisfies
 the following condition:
\begin{enumerate}
	 
	\item[(3)] Both points, $\xi$ and $\eta$, are in the extremal case.
\end{enumerate}
Recall that under these conditions, we have that  
\begin{equation}
	\label{desigualdad_t}
	t_{\xi}\leq t_{\eta}.
\end{equation}

Also, following the same arguments  as in the proof of Proposition \ref{extremal_especial} below we can assume that  $\nub_{\p_{\eta'}}(\theta_i)> 1$   and hence that  $\nub_{{\xi'}}(\theta_i)> 1$ for $i=1,\ldots, e$ (see (\ref{mayor_1})). 
\medskip

Suppose first that   $\Hord_{X'}^{(d)}(\eta')=\ord_{X'}^{(d)}(\eta')$.  Since 
\begin{equation}
	\label{desigualdad_orden_n_c}
	\overline{\ord}_{{\mathfrak p}_{\eta'}}(h_i(\theta_i)W^{p^{\ell_i}})\geq \overline{\ord}_{{\mathfrak p}_{\eta'}}  ({\G}_{X'})=\ord_{\alpha(\eta')}\Gd,
\end{equation}
and by Remark \ref{remark_ordenes_intermedios}, 
$$
\ord_{{\mathfrak p}_{\alpha(\eta')}}(\Gd) \leq \frac{\nu_{{\mathfrak p}_{\alpha(\eta')}}(\tilde{a}_j^{(i)})}{j},
$$
for all $i=1,\ldots, e$ and $j=1,\ldots, p^{\ell_i}-1$, the hypothesis   $\nub_{\p_{\eta'}}(\theta_i)> 1$ for $i=1\ldots, e$,  implies    
$$\frac{\nu_{\mathfrak{p}_{\alpha(\eta')}}(\tilde{a}_{p^{\ell_i}}^{(i)})}{p^{\ell_i}} >1.$$

Now,  by the discussion in \S\ref{cleaning},   after a finite number of translations of the form $\theta_i'=\theta_i+s_i$ with $s_i\in S$ it can be assumed that  for $i=1,\ldots, e$, 
$$\frac{\nu_{\mathfrak{p}_{\alpha(\eta')}}(\tilde{a}_{p^{\ell_i}}^{(i)})}{p^{\ell_i}}\geq \ord_{X'}^{(d)}(\eta').$$

Recall  that for each   translations, $\theta_i'=\theta_i+s_i$, we have that 
$$\nu_{{\mathfrak p}_{\alpha(\eta')}}(s_i)\geq \frac{\nu_{\mathfrak{p}_{\alpha(\eta')}}(\tilde{a}_{p^{\ell_i}}^{(i)})}{p^{\ell_i}}>1 $$
(see (\ref{des_h_p}) and (\ref{pendiente_hord_no_cerrado}), which implies that, after a finite number of translations, we are in the following situation: $B'=S[\theta_1,\ldots, \theta_e]$, with 
\begin{equation}
	\nub_{{\mathfrak p}_{\eta'}}(\theta_i)\geq  \ord_{X'}^{(d)}(\eta') 
\end{equation}
and 
\begin{equation}
	\nub_{{\mathfrak p}_{\eta'}}(\theta_i)>1
\end{equation} 
for $i=1,\ldots, e$.
This already implies that $\SSl(\calo_{X',\eta'})\geq \ord_{X'}^{(d)}(\eta')$. 
Since 
${\mathfrak m}_{\xi'}={\mathfrak m}_{\alpha(\xi')}+\langle \theta_1,\ldots, \theta_e\rangle $, after relabeling, we can assume that $\theta_1,\ldots, \theta_{t_{\xi'}}$ form a $\lambda_{\xi'}$-sequence. Thus 
$${\mathfrak m}_{\xi'}={\mathfrak m}_{\alpha(\xi')}+\langle \theta_1,\ldots, \theta_{t_{\xi'}} \rangle,$$
and by Lemma \ref{lucky_lemma}, 
$${\mathfrak p}_{\eta'}= {\mathfrak m}_{\alpha(\eta')}+\langle \theta_1,\ldots, \theta_{t_{\xi'}}\rangle.$$
Hence $t_{\eta'}=t_{\xi'}$ and setting $t=t_{\eta'}$, we have that  $\theta_1,\ldots, \theta_{t}$ form also a $\lambda_{\eta'}$-sequence. Since in addition $\xi'$ is in the extremal case, we can assume that $k(\xi)=k(\xi')$ (see Remark \ref{lema_lema}). Finally, we can use Proposition \ref{prop_sucesiones}  (with  $\gamma_i'=\theta_i$ for $i=1,\ldots,t$) to  conclude that 
$$\Hord^{(d)}_X(\eta)=\ord^{(d)}_X (\eta)=\min \{{\mathcal S}\text{-Sl}({\mathcal O}_{X,\eta}), \ord^{(d)}_X (\eta)\}.$$
\medskip

Suppose now that  $\Hord_{X'}^{(d)}(\eta')<\ord_{X'}^{(d)}(\eta')$. Since $\eta'$ is in the extremal case, by Proposition \ref{extremal_especial} below, we only have to consider the case where
\begin{equation}\label{desigual86}
	1<\Hord_{X'}^{(d)}(\eta')<\ord_{X'}^{(d)}(\eta').
\end{equation}

As in \S \ref{setting_prueba}, we can assume that
\begin{equation}
	\label{elemento_en_algebra}
	h_i(\theta_i)W^{p^{\ell_i}}=(\theta_i^{p^{\ell_i}}+\tilde{a}_{1}^{(i)}\theta_i^{p^{\ell_i}-1}+\ldots+\tilde{a}_{p^{\ell_i}}^{(i)})W^{p^{\ell_i}}\in \G_{B'}.
\end{equation}
By (\ref{desigual86}), there must be some  indexes $i_1,\ldots,i_c$,  for which
$$1<\nub_{\mathfrak{p}_{\eta'}}(\theta_{i_j})=\frac{\nu_{\mathfrak{p}_{\eta'}}(\tilde{a}_{p^{\ell_{i_j}}}^{(i_j)})}{p^{\ell_{i_j}}}\leq \Hord^{(d)}_{X'}(\eta').$$
If the second inequality is strict for all $i_1,\ldots,i_c$, then we can make changes  of variables of the form $\theta_{i_j}'=\theta_{i_j}+s_{i_j}$ with $s_{i_j}\in S$ and 
$$\nu_{{\mathfrak p}_{\alpha(\eta')}} (s_{i_j})\geq \frac{\nu_{\mathfrak{p}_{\eta'}}(\tilde{a}_{p^{\ell_{i_j}}}^{(i_j)})}{p^{\ell_{i_j}}}$$
(see (\ref{pendiente_hord_no_cerrado})),
such that for some index, which we can assume to be $e$,
$$\nub_{\mathfrak{p}_{\eta'}}(\theta'_e)=\Hord^{(d)}_{X'}(\eta')\leq \nub_{\mathfrak{p}_{\eta'}}(\theta'_i),$$
for $i=1,\ldots,e-1$.  Notice that from the way the changes are made, $\nub_{\xi'}(\theta_i')\geq \nub_{\xi}(\theta_i)>1$ and    $B'=S[\theta_1',\ldots, \theta_e']$
(here there is no need to localize as it is shown in the proof of \cite[Proposition 5.8]{BVComp}; see also \S \ref{cleaning}).

To summarize,
there is   a presentation of $B'$, $B'=S[\theta_1',\ldots, \theta_e']$, such that $\nub_{{\mathfrak p}_{\eta'}}(\theta_i')>1$ for $i=1,\ldots, e$  (thus $\nub_{{\xi'}}(\theta_i)>1$  for $i=1,\ldots, e$),   and so that
$$\nub_{{\mathfrak p}_{\eta'}}(\theta_e)=Sl(\mathcal{P})(\eta')= \Hord^{(d)}_{X'}(\eta').$$

\medskip

Now recall that  $\m_{\xi'}=\m_{\alpha(\xi')}+\langle \theta_1',\ldots,\theta_e'\rangle$.  We claim that we can select $t_{\xi'}$ elements among $\theta_1',\ldots, \theta_e'$ so that the order of at least one of them at ${\mathfrak p}_{\eta'}$ equals $\nub_{{\mathfrak p}_{\eta'}}(\theta_e')$. The claim follows immediately if $\theta_e'\in\m_{\xi'}\setminus \m_{\xi'}^2$. Otherwise,   we can assume, without loss of generality, that
the classes of      $\theta_1',\ldots, \theta'_{t_{\xi'}}$    are linearly independent    at $\m_{\xi'}/ \m_{\xi'}^2$
and that $\nub_{{\mathfrak p}_{\eta'}}(\theta_i')>\nub_{{\mathfrak p}_{\eta'}}(\theta_e')$ for $i=1,\ldots, t_{\xi'}$.
Then  we can  replace $\theta_1'$ by $\theta_1'+\theta'_e$, so $\m_{\xi'}=\m_{\alpha(\xi')}+\langle \theta'_1,\ldots,\theta'_{t_{\xi'}}\rangle$,
$\theta'_1,\ldots,\theta'_{t_{\xi'}}$ form  a $\lambda_{\xi'}$-sequence and $\nub_{{\mathfrak p}_{\eta'}}(\theta_1')=\nub_{{\mathfrak p}_{\eta'}}(\theta_e')$. By Lemma \ref{lucky_lemma}, ${\mathfrak p}_{\eta'}={\mathfrak p}_{\alpha(\eta')}+\langle \theta_1',\ldots,\theta'_{t_{\xi'}}\rangle$, thus  $t_{\xi'}\geq t_{\eta'}$, hence  by (\ref{desigualdad_t}),
$t_{\xi'}=t_{\eta'}$,   we have that  $\theta_1',\ldots, \theta_{t_{\xi'}}'$     is also  a $\lambda_{\eta'}$-sequence and setting $t:={t_{\eta'}}$,  by construction
\begin{equation}
	\label{igualdad_d_e_no_cerrado_bis}
	\min\{\nub_{{{\mathfrak p}_{\eta'}}}(\theta_1'), \ldots, \nub_{{\mathfrak p}_{\eta'}}(\theta'_t)\}=\Hord^{(d)}_{X'}(\eta').
\end{equation}
Note that, in general, $\nub_{{\mathfrak p}_{\eta'}}(\theta_i)\leq \nub_{\eta'}(\theta_i)$.
If these inequalities were strict for all $i=1,\ldots,t$ then we would have found a $\lambda_{\eta'}$-sequence for which 
$$\min\{\nub_{{{\eta'}}}(\theta_1'), \ldots, \nub_{{\eta'}}(\theta'_t)\}>\Hord^{(d)}_{X'}(\eta'),$$
and since $\Hord^{(d)}_{X'}(\eta')<\ord^{(d)}_{X'}(\eta')$ and we already know that the theorem holds for $B_{{\mathfrak p}_{\eta'}}$, we would get a contradiction.
Hence, there must be some index $i$ for which $\nub_{\eta'}(\theta_i)=\Hord^{(d)}_{X'}(\eta')$.
Finally, since $\xi'$ is in the extremal case, we can also assume that $k(\xi')=k(\xi)$, and apply   Proposition \ref{prop_sucesiones} to $\gamma_i'=\theta_i'$ for $i=1,\ldots, t$, from where it follows that the theorem holds for $\eta\in X$.
\end{proof}

\begin{proposition} \label{extremal_especial}
Let $X$ be an equidimensional variety of dimension $d$ defined over a perfect field $k$. Let $\zeta\in X$ be a point of multiplicity $m>1$ which is in the extremal case. If $\Hord_X^{(d)}(\zeta)<\ord_X^{(d)}(\zeta)$, then
$$\Hord_X^{(d)}(\zeta)>1.$$
\end{proposition}
\begin{proof}
\noindent{\bf Closed points.} Suppose first that $\zeta=\xi$ is a closed point in $X$. Consider a suitable \'etale extension of $({\mathcal O}_{X,\xi}, {\mathfrak m}_{\xi}, k(\xi))$ as in
\S\ref{setting_prueba}, and work on $(B', {\mathfrak m}_{\xi'}, k(\xi'))$. Following the notation and results in \S \ref{setting_prueba} (A),  we can write $B'=S[\theta_1,\ldots, \theta_e]$ with $\theta_i\in {\mathfrak m}_{\xi'}$ for $i=1,\ldots, e$. And since $\xi'$ is in the extremal case, by Lemma  \ref{base_ker}, we can assume that $\nub_{\xi'}(\theta_i)>1$ for $i=1,\ldots, e$.

Recall that
\begin{equation}\label{recordar_slope}
     Sl(\mathcal{P})(\xi')
    =\min_{i=1,\ldots,e}\left\lbrace \frac{\nu_{\alpha(\xi')}(\tilde{a}_{p^{\ell_i}}^{(i)})}{p^{\ell_i}}, \ord_{\alpha(\xi')}(\Gd) \right\rbrace,
    \end{equation}
and that for every $i\in {1,\ldots,e}$ and each $j=1,\ldots, {p^{\ell_i}}-1$,
    \begin{equation}     \label{desigualdad_intermedia_b}
\Hord_{X'}^{(d)}(\xi') <  \ord_{\alpha(\xi')}(\Gd) \leq
    \frac{\nu_{{\alpha(\xi')}}(\tilde{a}_j^{(i)})}{j},
    \end{equation}
where the first inequality
follows  from  the hypothesis in the proposition, and the second from Remark \ref{remark_ordenes_intermedios}. Thus, there must be some  $i\in \{1,\ldots, e\}$ such that,
\begin{equation} \label{EqMenorCoeff_p}
\frac{\nu_{\alpha(\xi')}(\tilde{a}_{p^{\ell_i}}^{(i)})}{p^{\ell_i}}<\ord_{\alpha(\xi')}(\Gd).
\end{equation}
For every $i$ such that (\ref{EqMenorCoeff_p}) holds, since
  $h_i(\theta_i)W^{p^{\ell_i}}\in \G_{B'}$,
    \begin{equation}
    \label{desigualdad_orden_b}
\overline{\ord}_{\xi'}(h_i(\theta_i)W^{p^{\ell_i}})=\overline{\ord}_{\xi'}(\theta_i^{p^{\ell_i}}+\tilde{a}_{1}^{(i)}\theta_i^{p^{\ell_i}-1}+
    \ldots+\tilde{a}_{p^{\ell_i}}^{(i)})W^{p^{\ell_i}}\geq \overline{\ord}_{\xi'} ({\G}_{B'})=\ord_{\alpha(\xi')}\Gd,
    \end{equation}
    (see equality (\ref{orden_G_J})). Thus,  necessarily, for those indexes $i$,
    $$\nub_{\xi'}(\theta_i)= \frac{\nu_{\alpha(\xi')}(\tilde{a}_{p^{\ell_i}}^{(1)})}{p^{\ell_i}},$$
    and since $\nub_{\xi'}(\theta_i)>1$ the result follows from the definition of $\Hord_{X'}^{(d)}(\xi')=\Hord_{X}^{(d)}(\xi)$.
    \medskip

    \noindent {\bf Non-closed points.}   Let $\zeta=\eta\in X$ be a non-closed point  of multiplicity $m\geq 1$. Denote by $\mathfrak{p}_\eta$ the prime defined by $\eta$ in some affine open set of $U\subset X$.
     Choose a closed point $\xi\in\overline{\{\eta\}}\subseteq X$ with the following conditions:
     \begin{enumerate}
         \item $\xi$ and $\eta$ have the same multiplicity $m$.
         \item $\mathcal{O}_{X,\xi}/\mathfrak{p}_{\eta}$ is a regular local ring of dimension $d-r$ for some $r\geq 1$.
         \item Both points, $\xi$ and $\eta$, are in the extremal case.
     \end{enumerate}

     Conditions (1) and (2) hold in some open affine set $U\subset X$ containing $\eta$.
     To see that condition (3) can be achieved, choose a minimal set of generators $z_1,\ldots,z_r,\gamma_1,\ldots,\gamma_{t_{\eta}}\in\calo_{X,\eta}$
     of $\mathfrak{p}_{\eta}\calo_{X,\eta}$ with $\nub_{\eta}(\gamma_i)>1$, for $i=1,\ldots,t_{\eta}$.
     Notice that after shrinking $U$, if needed, we can assume that $\mathfrak{p}_{\eta}=\langle z_1,\ldots,z_r,\gamma_1,\ldots,\gamma_{t_{\eta}} \rangle$ on $U$,
     and that for any closed point $\xi\in U\cap \overline{\{\eta\}}$, $\nub_{\xi}(\gamma_i)>1$ for $i=1,\ldots,t_{\eta}$.

     Let  $\xi\in U\cap \overline{\{\eta\}}$ be a closed point. Since condition (2) holds, we can find $z_{r+1},\ldots,z_d\in\mathfrak{m}_{\xi}$ such that
     $\mathfrak{m}_{\xi}=\langle z_1,\ldots,z_d,\gamma_1,\ldots,\gamma_{t_{\eta}} \rangle$ with $\nub_{\xi}(z_i)=1$, for $i=1,\ldots,d$. Since
     $\nub_{\xi}(\gamma_i)>1$ for $i=1,\ldots, t_{\eta}$,  (3) holds. In particular if $\overline{\gamma_i}$ denotes the class of $\gamma_i$ in ${\mathfrak m}_{\xi}/{\mathfrak m}_{\xi}^2$, then
     \begin{equation}
     \label{generan_kernel}
     \ker(\lambda_{\xi})=\langle \overline{\gamma_1},\ldots,\overline{\gamma_{t_{\eta}}}\rangle
     \end{equation}  and   $t_{\xi}\leq t_{\eta}$.

Let $B=\mathcal{O}_{X,\xi}$, and let $B\longrightarrow B'$ an \'etale extension, and $S\longrightarrow B'$ a finite morphism as in \S\ref{setting_prueba}. Denote by and $\mathfrak{p}_{\eta'}\subset B'$ the prime dominating $\mathfrak{p}_\eta B$, $\xi'$ the closed point dominating $\xi$, and $\p_{\alpha(\eta')}$ the prime $\p_{\eta'}\cap S$. By \cite[Corollary 3.2]{COA}, $\p_{\alpha(\eta')}$ determines a regular prime.  Under these assumptions, using \cite[Lemma 3.6]{COA}, we can assume that $B'=S[\theta_1,\ldots,\theta_e]$ with $\theta_i\in \p_{\eta'}$. Note that $\nub_{\p_{\eta'}}(\theta_i)\geq 1$ and $\nub_{{\xi'}}(\theta_i)\geq 1$ (see \S \ref{p_presentations_suitable} (C)). 

     Since $\ker(\lambda_{\xi})\otimes_{k(\xi)}k(\xi')=\ker(\lambda_{\xi'})$, by (\ref{generan_kernel}) and Remark \ref{suma_directa},
     \begin{equation}
\label{generadores_mxl}
{\mathfrak m}_{\xi'}={\mathfrak m}_{\alpha(\xi')}+\langle \gamma_1,\ldots, \gamma_{t_{\eta}}\rangle.
     \end{equation}
     Thus,  by Lemma \ref{lucky_lemma},
\begin{equation}
\label{generadores_primo}
{\mathfrak p}_{\eta'}={\mathfrak p}_{\alpha(\eta')}+\langle \gamma_1,\ldots, \gamma_{t_{\eta}}\rangle.
\end{equation}
By Remark \ref{remark_lema_lema},   
maybe after translating  the $\theta_i$, we can assume  that $B'=S[\theta_1',\ldots, \theta_e']$  and that 
\begin{equation}\label{mayor_1}
\min\{\nub_{{\mathfrak p}_{\eta'}}(\theta'_i): i=1,\ldots, e\}\geq\min \{\nub_{{\mathfrak p}_{\eta'}}(\gamma_i): i=1,\ldots, s\}>1.
\end{equation}
Now, using   (\ref{slope_p_no_cerrado}) and the definition of $\Hord_{X'}(\eta')$, the proof of the proposition follows using a similar argument as the one we used for closed points   (see \S \ref{setting_prueba} (C)), thus $1<\Hord_{X'}(\eta')=\Hord_{X}(\eta)$.
\end{proof}

The following example illustrates that, for a given $d$-dimensional variety $X$,  there may be  non-closed points $\eta\in X$ with $\SSl({\mathcal O}_{X,\eta})=1$ but $\ord^{(d)}(\eta)>1$. Thus the last part of the first statement
of Theorem \ref{maintheorem} might not hold for non-closed points.
  
\begin{example}
 	Let  $p\in {\mathbb Z}_{>0}$ be a  prime and let $X$ be the  hypersurface in $V^{(3)}:=\Spec({\mathbb F}_p[x,y_1,y_2])$   defined by 
 	$f=x^p-y_1^p y_2$. Then ${\mathfrak p}=\langle x,y_1\rangle$ determines a non-closed point of maximum multiplicity $p$  which is not in the extremal case. The Rees algebra 
 	\begin{equation}
 		\label{p_pres_ejemplo}
 		\G^{(3)}=\Diff \left({\mathbb F}_p[x,y_1,y_2][(x^p-y_1^p y_2)W^p  ]\right)={\mathbb F}_p[x,y_1,y_2][y_1^pW^{p-1},   (x^p-y_1^p y_2)W^p].
 	\end{equation} 
 	represents the stratum of $p$-fold points of $X$. Let $\xi$ be the closed point corresponding to ${\mathfrak m}=\langle x,y_1,y_2\rangle$. Then the natural inclusion ${\mathbb F}_p[y_1,y_2]\subset {\mathbb F}_p[x,y_1,y_2]$ is $\G^{(3)}$-admissible at $\xi$ and provides a presentation of $B={\mathbb F}_p[x,y_1,y_2]/\langle f\rangle $ as in \S\ref{setting_prueba}. The Rees algebra 
 	$$\G^{(2)}={\mathbb F}_p[y_1,y_2][y_1^pW^{p-1}],$$
 	is an elimination algebra for $\G^{(3)}$. Notice that (\ref{p_pres_ejemplo}) is a $p$-presentation of $\G^{(3)}$ which is already in normal form at $\eta$, and that 
 	$$\frac{\nub_{\eta}(y_1^p y_2)}{p}=\frac{p}{p}=1.$$
 	On the other hand, setting ${\mathfrak q}:={\mathfrak p}\cap {\mathbb F}_p[y_1,y_2]$, we have that 
$\ord^{(2)}_X(\eta)=\ord_{\mathfrak q}(\G^{(2)})=\frac{p}{p-1}$.
Thus, $\Hord_{X}^{(2)}(\eta)=1< \ord^{(2)}_X(\eta)$, even though $\eta$ is not in the extremal case.
\end{example}
\bigskip

\noindent
\textbf{Funding.}
The authors were partially supported by PGC2018-095392-B-I00.
The second author  was partially  supported from the Spanish Ministry of Economy and Competitiveness, through the ``Severo Ochoa'' Program for Centres of Excellence in R\&D (SEV-2015-0554).
\medskip

\noindent
\textbf{Competing interest.}
The authors have no competing interests to declare that are relevant to the content of this article.

\bibliographystyle{plain}

\end{document}